\documentclass{article}
\usepackage{amsmath,amssymb,amsthm}
  \usepackage{paralist}
  \usepackage{graphics} 
 \usepackage{epstopdf}
\usepackage{mathtools}
\usepackage{upgreek}
\usepackage[svgnames]{xcolor}
\usepackage{mathrsfs}
\usepackage[mathscr]{euscript}
\usepackage{caption}
\usepackage{subcaption}
\usepackage{booktabs}
\usepackage{multirow}
\usepackage{siunitx}
\usepackage{csquotes}
\usepackage{enumitem}

\newtheorem{theorem}{Theorem}[section]
\newtheorem{corollary}[theorem]{Corollary}

\newtheorem{lemma}[theorem]{Lemma}

\theoremstyle{definition}

\newtheorem{remark}[theorem]{Remark}

\usepackage{cite}
\newcommand{\fz}{\frac}
\newcommand{\prz}[2]{ \frac{\partial{#1}}{\partial{#2}} }

\newcommand{\ol}{\overline}

\renewcommand{\Omega}{\varOmega}
\renewcommand{\Gamma}{\varGamma}
\renewcommand{\Psi}{\varPsi}
\renewcommand{\Pi}{\varPi}
\newcommand{\R}{\mathbb{R}}
\newcommand{\N}{\mathbb{N}}
\newcommand{\Z}{\mathbb{Z}}

\newcommand{\defeq}{\vcentcolon=}

\allowdisplaybreaks[4]

\newcommand{\cB}{\color{blue}}
\newcommand{\cK}{\color{black}}

\title{ Theoretical and numerical studies for energy estimates of the shallow water equations with a transmission boundary condition}
\date{}

\begin{document}
\maketitle

\centerline{\scshape Md. Masum Murshed$^{1,2*}$, Kouta Futai$^{1}$, Masato Kimura$^{3}$ and Hirofumi Notsu$^{3,4}$ }
\medskip
	\centerline{\footnotesize 1. Division of Mathematical and Physical Sciences, Kanazawa University, Kanazawa {\cK 920-1192}, Japan}
\medskip
	\centerline{\footnotesize 2. Department of Mathematics, University of Rajshahi, Rajshahi 6205, Bangladesh}
\medskip
        \centerline{\footnotesize 3. Faculty of Mathematics and Physics, Kanazawa University, Kanazawa {\cK 920-1192}, Japan}
\medskip
        \centerline{\footnotesize 4. Japan Science and Technology Agency, PRESTO, Kawaguchi 332-0012, Japan.}
\medskip
\centerline{\footnotesize  Emails: \cB mmmurshed82@gmail.com, futai.k.1275@stu.kanazawa-u.ac.jp,}
\centerline{\footnotesize \cB mkimura@se.kanazawa-u.ac.jp and notsu@se.kanazawa-u.ac.jp} 



\begin{abstract}
\cK
Energy estimates of the shallow water equations~(SWEs) with a transmission boundary condition are studied theoretically and numerically.
In the theoretical part, using a suitable energy, we begin with deriving an equality which implies an energy estimate of the SWEs with the Dirichlet and the slip boundary conditions.
For the SWEs with a transmission boundary condition, an inequality for the energy estimate is proved under some assumptions to be satisfied in practical computation.
Hence, it is recognized that the transmission boundary condition is reasonable in the sense that the inequality holds true.
In the numerical part, based on the theoretical results, the energy estimate of the SWEs with a transmission boundary condition is confirmed numerically by a finite difference method~(FDM).
The choice of a positive constant~$c_0$ used in the transmission boundary condition is investigated additionally.
Furthermore, we present numerical results by a Lagrange--Galerkin scheme, which are similar to those by the FDM.
From the numerical results, it is found that the transmission boundary condition works well numerically.
\end{abstract}

\section{Introduction}\label{sec1}
The shallow water equations~(SWEs) are often used for the simulation of tsunami/ storm surge in the bay.
In such simulation there are some boundaries in the open sea, see Figure \ref{bob}. In a real situation, if wave propagates towards such boundaries in the open sea, then there should not have any reflection on these boundaries.
Therefore, in the simulation a special type of boundary condition should be imposed on these boundaries.
In this study, following~\cite{Kanayamatsunami}, {\cK we employ} a transmission boundary condition on the boundaries in the open sea which is capable to remove these kind of artificial reflection.
\par
Many studies have been conducted on the prediction of the surges due to tropical storms for the Bay of Bengal region covering the coast of Bangladesh and the east coast of India, see, e.g., \cite{das72,debsarma09,Jonhs80,Paul2012b,Paul2013,Paul2014,paulpolar2017,paul2018a,paul2018b,Roy99c,Roy04}.
Almost all of the studies mentioned here have employed a radiation type boundary condition for the boundaries in the open sea, which is very similar to the transmission boundary condition used in~\cite{Kanayamatsunami}.
It is to be noted here that these papers mainly present numerical results without mathematical discussion on the stability of the model with this kind of boundary condition.
\begin{figure}[!htbp]
	\centering
	\includegraphics[width=2.0 in]{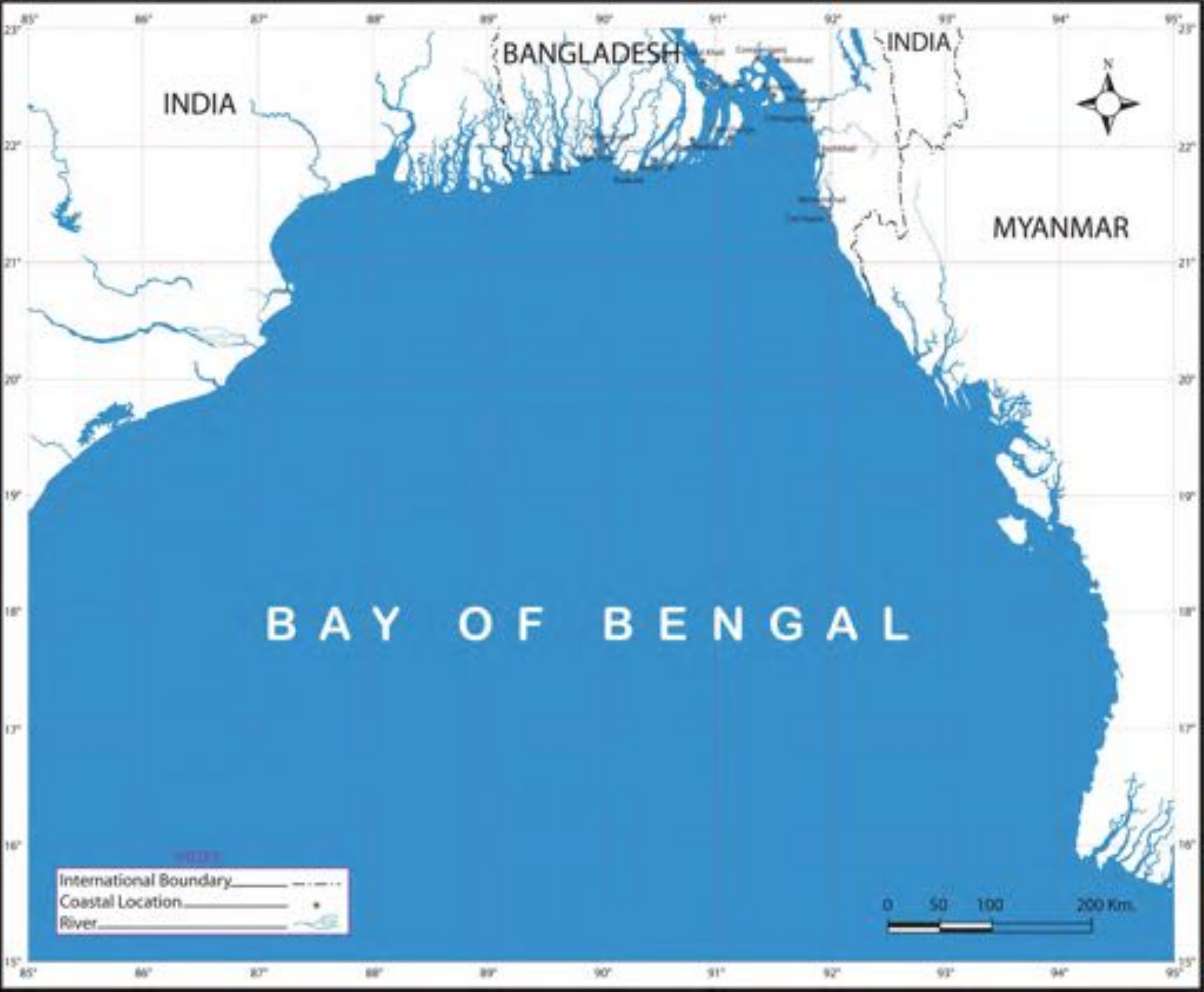}
	\caption{The Bay of Bengal and the coastal region of Bangladesh }\label{bob}
\end{figure}
\par
The SWEs can be considered as a coupled system of a pure convection equation for {\cK the function~$\phi$ of total wave height} and a simplified Navier--Stokes equation for the velocity $u = (u_1, u_2)^T$ by averaging function values in $x_3$-direction.
It is known that a boundary data for $\phi$ is necessary on the so-called inflow boundary, where $u \cdot n < 0$ is satisfied for the outward unit normal vector~$n$.
We can easily know whether the Dirichlet data for~$\phi$ is required or not on the Dirichlet {\cK and} the slip boundaries for~$u$, since the sign of~$u \cdot n$ is known {\cK a priori}.
On the transmission boundary~$\Gamma_T$, however, the boundary condition for~$u$ and~$\phi$ is mysterious and problematic from both computational and mathematical view points.
The transmission boundary condition of the form
\begin{equation}\label{eq00}
   u(x, t)=c(x)\frac{\eta(x, t)}{\phi(x, t)}n(x)
\end{equation}
is often used on~$\Gamma_T$, where $c(x)$ is a given positive valued function and $\eta(x, t)=\phi(x, t) - \zeta(x)$ is the elevation from the reference height for a given depth function~$\zeta$.
\par
Let $\phi_h^{k}$ and $ u_h^k$ be the approximations of $\phi^k \defeq \phi(\cdot, t^k)$ and $ u^k \defeq  u(\cdot, t^k)$, respectively, where $t^k \defeq k\Delta t \ (k\in \mathbb Z)$ for a time increment $\Delta t$.
In our computation we get $\phi_h^{k+1}$ by using $u_h^k$, and then $u_h^{k+1}$ by using the condition~\eqref{eq00} as the Dirichlet boundary condition for~$u$.
But at the same time we should consider~\eqref{eq00} as the boundary data for $\phi_h^{k+1}$ if the position is on inflow boundary, i.e., $ u_h^{k+1}\cdot{  n}<0$.
In fact, if $ u_h^{k+1} \cdot n < 0$, we need to give the value of $\phi_h^{k+1}$ which is unknown.
\par
We have an experience of computation of the SWEs with the transmission boundary condition by a Lagrange--Galerkin~(LG) method, where the LG method is based on the time discretization of the material derivative,
\[
\frac{\phi^{k+1}(x)-\phi^{k}(x- u^k(x)\Delta t)}{\Delta t}.
\]
The position $x- u^k(x)\Delta t $ is the so-called upwind point of~$x$ with respect to~$ u^k$.
In the computation a \enquote{nearest} boundary value of~$\phi^k$ is used if the upwind point places outside the domain, and the LG method works without boundary data for $\phi^{k+1}$ even if $u^{k+1} \cdot n < 0$.
In the LG method the problem on the transmission boundary seems to be solved numerically, but it is still problematic mathematically.
\par
In this paper, in order to understand the transmission boundary condition mathematically, we study the stability of the {\cK SWEs} in terms of a suitable energy, and confirm the stability numerically by a finite difference method~(FDM).
Since the LG method solves the problem implicitly by using the upwind point and is not suitable to understand the problem with the transmission boundary condition mathematically, we employ an FDM to avoid the special numerical {\cK treatment} in the LG method.
It is to be noted here that we can show a {\cK (successful)} energy estimate of the SWEs, when only the Dirichlet and the slip boundary conditions are employed, cf. Corollary~\ref{cor1}-(ii),
\cK
where such discussions have been done under the periodic boundary condition, e.g.,~\cite{BreDes-2006,Luc-2009}.
\cK
As far as we know, however, there is no mathematical results on the energy estimate of the SWEs with the transmission boundary condition.
\par
The stability is considered {\cK theoretically} with respect to the energy as follows.
\cK
Introducing a suitable energy, we begin with deriving an equality that 
the time-derivative of the energy consists of four terms, where three
terms are line integrals over the boundary and the other term is an
integral over the whole domain which is always non-positive, cf. Theorem~\ref{th1}.
Since the three line integrals vanish over the Dirichlet and the slip
boundaries, as a result, we have the three line integrals over the
transmission boundary and the integral over the domain, cf. Corollary~\ref{cor1}-(i).
An energy estimate is obviously obtained if there is no transmission boundary, cf. Corollary~\ref{cor1}-(ii).
In addition, we obtain that a sum of two line integrals over the
transmission boundary is non-positive under some conditions to be
satisfied in real computations, cf. Theorem~\ref{th2}.
Although, at present, the mathematical results do not derive the
stability estimate of the SWEs with the transmission boundary condition directly,
we have good information and can study the stability numerically by
using the theoretical results.
In the latter half of the paper, Sections~\ref{sec4} and~\ref{sec5}, the energy estimates derived in Section~\ref{sec3} are studied numerically.
\par
This paper organized as follows.
The problem is stated mathematically in Section~\ref{sec2}.
The energy estimate is theoretically studied in Section~\ref{sec3}.
Numerical results by an FDM are shown to see that the transmission boundary condition works well and that {\cK the solution is stable} in terms of the energy in Section~\ref{sec4}.
Here, the effect of a positive constant to be used in the transmission boundary condition is investigated.
In Section~\ref{sec5}, numerical results by an LG method are shown to support those by the FDM.
Finally, conclusions are given in Section~\ref{sec6}.
\cK
%
\section{\cK Statement of the problem}\label{sec2}
In this section, we state the mathematical problem to be considered in this paper.
Let $\Omega \subset \R^2$ be a bounded domain and~$T$ a positive constant.
We consider the problem; find $(\phi, u): \ol{\Omega} \times [0, T] \to \mathbb{R}\times \mathbb {R}^2$ such that
\begin{equation}\label{prob}
 \left\{
 \ 
  \begin{aligned}
   &  \frac{\partial\phi}{\partial t}+ \nabla \cdot(\phi u ) =0 \ &&\text{in} \ \Omega\times (0,T),\\
   & \rho\phi \Bigl[ \frac{\partial u}{\partial t}+({  u}\cdot\nabla) {  u} \Bigr] - 2\mu \nabla \cdot (\phi D(  u))+\rho g\phi\nabla \eta =0 \ &&\text{in} \ \Omega\times (0,T),\\
  & \phi =\eta+\zeta         \ \           &&\text{in} \ \Omega\times (0,T),\\
    \end{aligned}
\right.
\end{equation}
with boundary conditions
\begin{align}
& u = 0 && \hspace*{-1cm} \text{on} \ \Gamma_D\times (0,T), 
\label{0dbc} \\
& (D(u) n) \times n = 0, \ \  u \cdot n = 0  && \hspace*{-1cm} \text{on} \ \Gamma_S\times (0,T), 
\label{sbc} \\
& u = c \frac{\eta}{\phi} n && \hspace*{-1cm} \text{on} \  \Gamma_T\times (0,T), 
\label{tbc}
\end{align}
and initial conditions
\begin{align}
u =  u^0, \quad \eta = \eta^0 \qquad \text{in} \ \Omega, \ \text{at} \ t=0,
\label{IC}
\end{align}
where
\cK
$\phi$ is the total height of wave, 
$u=(u_1,u_2)^T$ is the velocity, 
\cK
$\eta: \ol{\Omega} \times [0, T] \to \mathbb{R}$ is the water level from the reference {\cK height}, 
$\zeta(x) > 0~(x \in \ol{\Omega})$ is the depth of water from the reference height, see Figure~\ref{domain}, 
$D(u) \defeq \left({\nabla u+(\nabla u )^T}\right)/2$ is the strain-rate tensor, 
$n$ is the unit outward normal vector, 
$\Gamma \defeq \partial\Omega$ is the boundary of $\Omega$, 
we assume that $\Gamma$ consists of non-overlapped three parts, 
$\Gamma_{D}$,  $\Gamma_{S}$ and $\Gamma_{T}$, i.e., 
$\overline{\Gamma}=\overline{\Gamma}_{D}\cup \overline{\Gamma}_S\cup \overline{\Gamma}_T$,  $\Gamma_{D}\cap\Gamma_{S}=\varnothing$, $\Gamma_{S}\cap\Gamma_{T}=\varnothing$,
$\Gamma_{T}\cap\Gamma_{D}=\varnothing$, the subscripts \enquote{$D$}, \enquote{$S$}, and \enquote{$T$} mean Dirichlet, slip, and transmission boundaries, respectively, 
$\rho > 0$ is the density of water, 
$\mu>0$ is {\cK the viscosity},
$g > 0$ is the acceleration due to gravity, and $c(x) \defeq c_0 \sqrt{g\zeta(x)}$ with a positive constant $c_0$.
{\cK In the rest of paper, we assume $\zeta \in C^1(\ol{\Omega})$.}
It is important to note here that the equations in~\eqref{prob} are derived
\cK
in~\cite{kanayama2006} by considering one-layer viscous SWEs.
\cK
\begin{figure}[!htpp]
	\centering
	\includegraphics[width=2.5in]{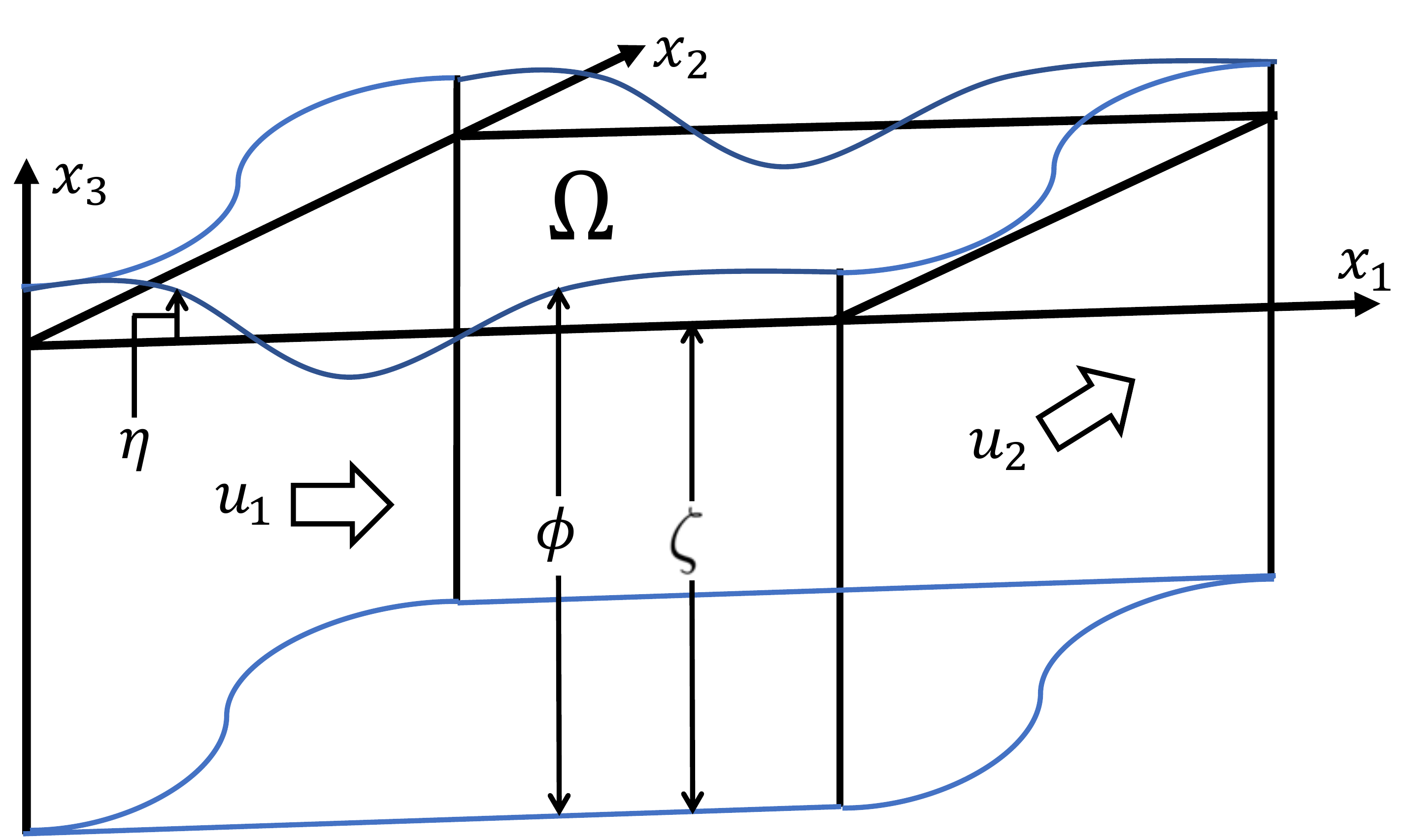}
	\caption{Model domain}\label{domain}
\end{figure}
\section{\cK Energy estimate}\label{sec3}
\cK
In this section, we define the total energy and study the stability of solutions to the problem stated in Section~\ref{sec2} {\cK in terms of the energy.}
\cK
For a solution of~\eqref{prob} 
\cK
the total energy~$E(t)$ at time $t\in [0, T]$ is defined by
\begin{align}
	E(t) \defeq E_1(t) + E_2(t), 
	\label{te}
\end{align}
where $E_1(t)$ and $E_2(t)$ are the kinetic and the potential energies defined by
\begin{align*}
	E_1(t) \defeq \int_{\Omega}\frac{\rho}{2}\phi| u|^2dx,
	\quad
	E_2(t) \defeq \int_{\Omega}\frac{\rho g |\eta|^2}{2}dx.
\end{align*}
\cK
Let symbols~$I_i(t; \Gamma)$, $i=1, \ldots, 3$, and $I_4(t; \Omega)$, $t\in [0, T]$, be integrations defined by
\cK
\begin{align*}
\cK
I_1(t; \Gamma) &\defeq - \frac{\rho}{2} \int_{\Gamma} \phi| u|^2  u \cdot {  n} \, ds, &
\cK
I_2(t; \Gamma) &\defeq -\rho g \int_{\Gamma} \phi\eta  u \cdot {  n} \, ds, \\
\cK
I_3(t; \Gamma) &\defeq 2\mu \int_{\Gamma} \phi \bigl[ D( u){  n} \bigr] \cdot  u \, ds, &
\cK
I_4(t; \Omega) &\defeq -2\mu \int_{\Omega} \phi|D( u)|^2 \, dx.
\end{align*}
These are used in the rest of this paper.
\cK
\begin{theorem}\label{th1}
Suppose that a pair of smooth functions $(\phi, u): \ol{\Omega} \times [0,T] \to \mathbb{R} \times \mathbb {R}^2 $ satisfies~\eqref{prob}.
Then, we have
\begin{align}
\frac{d}{dt}E(t) & = \cK \sum_{i=1}^{\cK 3} I_i(t; \Gamma) + I_4(t; \Omega).
\label{ED}
\end{align}
\end{theorem}
\par
We prove Theorem~\ref{th1} after preparing a lemma.
\begin{lemma}\label{lem1}
For smooth functions $\phi: \ol{\Omega} \times [0,T] \to \R$ and $u: \ol{\Omega} \times [0,T] \to \mathbb {R}^2 $, we have the following. 
\begin{align*}
(i) & \ \frac{ \partial }{ \partial t}(\phi  u)+\nabla \cdot[(\phi  u)\otimes  u ]=\left(  \frac{\partial\phi}{\partial t}+ \nabla \cdot(\phi u )   \right) u + \phi \left( \frac{ \partial  u}{ \partial t}+( u \cdot\nabla )  u\right),\\
(ii) & \ \int_{\Omega}(\nabla \cdot[(\phi  u)\otimes  u ])\cdot  u dx= \frac{1}{2} \int_{\Gamma}\phi| u|^2 u\cdot {  n} ds+\frac{1}{2}\int_{\Omega}[\nabla \cdot(\phi  u)]| u|^2dx.
\end{align*}
\end{lemma}
\begin{proof}
We prove~$(i)$.
From the identity,
\begin{align}
\cK
\nabla \cdot[(\phi  u)\otimes  u ] = [\nabla \cdot(\phi  u)] u+\phi[( u\cdot\nabla) u],
\label{proof_lemma_1}
\end{align}
we have
\begin{align*}
	\frac{ \partial }{ \partial t}(\phi  u)+\nabla \cdot[(\phi  u)\otimes  u ] 
	& =\frac{\partial\phi}{\partial t} u+\phi\frac{\partial  u}{\partial t} + [\nabla \cdot(\phi  u)] u + \phi [( u\cdot\nabla) u] \\
& =\left(  \frac{\partial\phi}{\partial t}+ \nabla \cdot(\phi u )   \right) u + \phi \left( \frac{ \partial  u}{ \partial t}+( u \cdot\nabla )  u \right),
\end{align*}
which completes the proof of~$(i)$.
\par
We prove $(ii)$.
\cK
Denoting the left hand side of~$(ii)$ by~$J$ and using the integration by parts formula,
\cK
we have
\begin{align}
J & = \int_{\Gamma}([(\phi  u)\otimes   u ]n)\cdot   u ds - \int_{\Omega} [(\phi u)\otimes   u] :  \nabla u \, dx
\notag \\
& = \int_{\Gamma}\phi|  u|^2  u\cdot  {  n} ds-\int_{\Omega}\phi (u\cdot\nabla)  u \cdot u\,dx,
\label{eq:I}
\end{align}
where $A : B = \sum_{i,j=1}^{{\cK 2}} A_{ij} B_{ij}$.
We, therefore, have
\begin{align*}
\qquad\qquad\qquad J & =\int_{\Omega} [\nabla\cdot(\phi   u)|  u|^2+(\phi(  u\cdot\nabla)  u)\cdot   u]\,dx && \mbox{(from~\eqref{proof_lemma_1})} \\
\qquad\qquad\qquad & =\int_{\Omega} [\nabla\cdot (\phi u)] |u|^2\,dx + \int_{\Gamma}\phi|  u|^2  u \cdot n \, ds - J && \mbox{(from~\eqref{eq:I})},
\end{align*}
which implies the desired result of~$(ii)$.
\end{proof}
\begin{proof}[Proof of Theorem~\ref{th1}]
Differentiating~\eqref{te} with respect to $t$, we get
\begin{align}
\frac{d}{dt} E(t) = \frac{d}{dt} E_1(t) + \frac{d}{dt} E_2(t).
\label{proof_thm_1}
\end{align}
\cK
We compute $\fz{d}{dt} E_1(t)$ and $\fz{d}{dt} E_2(t)$ separately.
\par
Firstly, $\fz{d}{dt} E_1(t)$ is computed as follows.
From Lemma~\ref{lem1}-$(i)$ and the first equation of~\eqref{prob}, we have
\cK
\begin{align*}
\phi \Bigl[ \frac{ \partial  u}{ \partial t}+( u \cdot\nabla )  u \Bigr]
=\frac{ \partial }{ \partial t}(\phi  u)+\nabla \cdot[(\phi  u)\otimes  u],
\end{align*}
which implies
\begin{equation}\label{Eq1.5}
\rho \Bigl[ \frac{ \partial }{ \partial t}(\phi  u)+\nabla \cdot[(\phi  u)\otimes  u ] \Bigr] - 2 \mu \nabla \cdot \bigl[ \phi D(  u) \bigr] + \rho g\phi\nabla \eta=0.
\end{equation}
Multiplying~\eqref{Eq1.5} by $ {u}$ and integrating with respect to $x$ over $\Omega$, we get
\begin{align}
	& \rho \int_{\Omega}\Bigl[ \frac{\partial}{\partial t}(\phi  u) \Bigr] \cdot  u~dx + \rho\int_{\Omega} \Bigl[\nabla \cdot [(\phi  u)\otimes  u ] \Bigr] \cdot  u~dx - 2\mu \int_{\Omega}\Bigl[ \nabla \cdot ( \phi D( u) ) \Bigr] \cdot  u~dx \notag \\
	& + \rho g \int_{\Omega}\phi\nabla \eta \cdot  u~dx = 0.
	\label{eq:proof_thm1_1}
\end{align}
\cK
From the equation~\eqref{eq:proof_thm1_1} above and the next two identities:
\begin{align*}
& \rho \int_{\Omega} \Bigl[ \frac{ \partial }{ \partial t}(\phi  u) \Bigr] \cdot  u \, dx + \rho \int_{\Omega} \Bigl[ \nabla \cdot[(\phi  u)\otimes  u ] \Bigr] \cdot {  u}~dx \\
& \quad = \rho \int_{\Omega} \Bigl( \prz{\phi}{t} | u|^2 + \phi \prz{ {u}}{t} \cdot  u \Bigr)\, dx + \frac{\rho}{2} \int_{\Gamma}\phi| u|^2 u\cdot {  n} \ ds + \frac{\rho}{2}\int_{\Omega} \bigl[ \nabla\cdot(\phi   u) \bigr]|{  u}|^2 \, dx \\
& \qquad\qquad\qquad\qquad\qquad\qquad\qquad\qquad\qquad\qquad\qquad\qquad\quad \mbox{(from Lemma~\ref{lem1}-$(ii)$)} \\
& \quad = \rho \int_{\Omega} \Bigl( \frac{1}{2} \prz{\phi}{t} | u|^2 + \phi  u \cdot \prz{ {u}}{t} \Bigr) dx + \frac{\rho}{2} \int_{\Gamma}\phi| u|^2 u\cdot {  n} \, ds \ \ \mbox{(from the first  eq. of~\eqref{prob})} \\
& \quad =\frac{d}{dt} \Bigl[ \frac{\rho}{2}\int_{\Omega}\phi| u|^2~dx \Bigr] + \frac{\rho}{2} \int_{\Gamma}\phi| u|^2 u\cdot {  n} \, ds 
=\frac{d}{dt} E_1(t) - {\cK I_1(t; \Gamma),}
\end{align*}
\begin{align*}
- 2 \mu \int_{\Omega}\Bigl[ \nabla \cdot (\phi D( u)) \Bigr] \cdot  u~dx
& = - 2\mu \int_\Gamma \phi \bigl[ D( u) n \bigr] \cdot u~ds + 2\mu \int_\Omega \phi | D( u) |^2~dx \\
& = - I_3(t; \Gamma) - I_4(t; \Omega),
\end{align*}
we obtain
\begin{align}
\fz{d}{dt} E_1 (t) &= \cK I_1(t; \Gamma) + I_3(t; \Gamma) + I_4(t; \Omega) - \rho g \int_{\Omega}  \nabla\eta\cdot (\phi u)~dx.
\label{Eq1.7}
\end{align}
\par
Secondly, $\fz{d}{dt} E_2(t)$ is computed as follows:
\begin{align}
	\fz{d}{dt} E_2(t)
	& = \frac{d}{dt} \Bigl[ \frac{\rho g}{2} \int_{\Omega}|\eta|^2 dx \Bigr] \notag\\
	& = \rho g \int_{\Omega} \eta \prz{\eta}{t} \, dx \notag\\
	& = \rho g \int_{\Omega} \eta \prz{\phi}{t} \, dx 
\hspace*{115pt}
( \text{from the third  eq. of (\ref{prob})}) \notag\\
	&= \rho g \int_{\Omega} \eta \bigl[ -\nabla\cdot(\phi  u) \bigr] dx 
	\qquad\qquad\qquad\qquad\ \mbox{(from the first  eq. of~\eqref{prob})} \notag\\
	&= - \rho g \int_{\Omega} \nabla\cdot(\eta \phi  u) dx + \rho g \int_{\Omega} \nabla\eta\cdot (\phi u) \, dx \notag\\
	&= {\cK I_2(t; \Gamma)} + \rho g \int_{\Omega} \nabla\eta\cdot (\phi u) \, dx.
\label{Eq1.8}
\end{align}
The result~\eqref{ED} follows by adding~\eqref{Eq1.7} and~\eqref{Eq1.8} and recalling~\eqref{proof_thm_1}.
\end{proof}
\cK
\begin{corollary}\label{cor1}
$(i)$
Suppose that a pair of smooth functions $(\phi, u): \ol{\Omega} \times [0,T] \to \mathbb{R} \times \mathbb{R}^2$ satisfies~\eqref{prob} with {\cK \eqref{0dbc}-\eqref{tbc}.}
Then, we have
\begin{align}
\frac{d}{dt}E(t) 
& \cK = \sum_{i=1}^3 I_i (t; \Gamma_T) + I_4(t; \Omega).
\label{EDT}
\end{align}
$(ii)$
Furthermore, if $\Gamma=\Gamma_D\cup\Gamma_S$ and $\phi(x,t) > 0~((x,t) \in \overline{\Omega}\times [0, T])$, we have
\begin{equation}\label{EDT3}
\frac{d}{dt}E(t)= {\cK I_4(t; \Omega)} \le 0.
\end{equation}
\end{corollary}
\begin{proof}
On $\Gamma_S$, from the first equation of~\eqref{sbc}, there exists a scalar function $w: \ol{\Omega} \times [0,T] \to \mathbb{R}$ such that $D(u) n = w(x,t) n$, which implies
$$
\bigl[D( u) n \bigr] \cdot u = (w n)\cdot u=w( u\cdot n) = 0.
$$
\cK
Hence, the result~\eqref{EDT} is established from Theorem~\ref{th1} with~\eqref{0dbc} and~\eqref{sbc}.
\par
When $\Gamma = \Gamma_D \cup \Gamma_S$, i.e., $\Gamma_T = \varnothing$, the identity~\eqref{EDT} derives~\eqref{EDT3}.
\end{proof}
\begin{theorem}\label{th2}
\cK
Suppose that a pair of smooth functions $(\phi, u): \ol{\Omega} \times [0,T] \to \mathbb{R} \times \mathbb{R}^2$ satisfies~\eqref{prob} with~\eqref{0dbc}-\eqref{tbc}, and an inequality
\begin{align}
\phi(x,t)>0, \quad (x,t) \in \overline{\Gamma}_T \times [0, T],
\label{cond:phi_positive}
\end{align}
and that there exists $\alpha \in (0,1)$ such that
\begin{align}
\eta (x,t) & \ge -\alpha \zeta(x), \quad x\in\overline{\Gamma}_T,\ t\in [0, T],
\label{cond:eta_zeta} \\
0 & < c_0 \le \sqrt{\fz{2}{\alpha}} \, (1-\alpha).
\label{cond:c_0}
\end{align}
\cK
Then, we have the following estimates:
\begin{equation}\label{eq_th2_I1I2}
I_1(t; \Gamma_T) + I_2(t; \Gamma_T) \le 0,
\end{equation}
in particular,
\begin{equation}\label{eq_th2}
\frac{d}{dt}E(t) \le \cK I_3(t; \Gamma_T).
\end{equation}
\end{theorem}
\begin{proof}
\cK
From~\eqref{EDT}, we prove~\eqref{eq_th2_I1I2} which implies~\eqref{eq_th2}, since $I_4(t; \Omega)$ is always non-positive.
\cK
We have
\begin{align*}
	\cK \sum_{i=1}^2 I_i(t; \Gamma_T)
	& = - \rho\int_{\Gamma_T}\phi ( u \cdot {  n}) \Bigl[ g\eta+ \frac{1}{2} |u|^2 \Bigr] \, ds \\
	& = - \rho\int_{\Gamma_T}\phi c\frac{\eta}{\phi}\Bigl[ g \eta+ \frac{1}{2}c_0^2g\zeta\frac{\eta^2}{\phi^2} \Bigr] \,ds \\
	& =-\rho g\int_{\Gamma_T} c\eta^2 \Bigl[ 1+\frac{c_0^2}{2}\frac{\zeta\eta}{(\zeta+\eta)^2} \Bigr] \, ds.
\end{align*}
Let $f(r) \defeq r/(1+r)^2$.
From $f^\prime (r)=(1-r)/(1+r)^3$, it holds that
$f(r_1) \le f(r_2)$, $-1 \le r_1 \le r_2 \le 1$.
Since $-1 \le  -\alpha \le \eta/\zeta \le 1$, we obtain $f(-\alpha) \le f (\eta/\zeta)$, i.e.,
\[
- \fz{\alpha}{(1-\alpha)^2} \le \fz{\eta \zeta}{(\zeta + \eta)^2},
\]
which implies that 
\[
	\cK
	\sum_{i=1}^2 I_i(t; \Gamma_T)
	\cK
	\le -\rho g\int_{\Gamma_T}c\eta^2\left\{1-\frac{c_0^2\alpha}{2(1-\alpha)^2} \right\}ds\le 0
\]
from the condition on $c_0$, i.e., $0 < c_0 \le \sqrt{2/\alpha} \, (1-\alpha)$.
\end{proof}
\cK
\begin{remark}
From Theorem~\ref{th2}, we can say that the transmission boundary condition~\eqref{tbc} is reasonable in the sense of~\eqref{eq_th2_I1I2} under the conditions~\eqref{cond:phi_positive}-\eqref{cond:c_0} to be satisfied in practical computation.
Although the sign of~$\fz{d}{dt}E(t)$ is as yet unknown due to~$I_3 (t; \Gamma_T)$, Theorem~\ref{th2} has meaningful for the energy estimate of the SWEs with the transmission boundary condition.
In fact, we observe numerically that $I_2(t; \Gamma)$ is dominant and $\sum_{i=1}^3 I_i(t; \Gamma)$ is negative, while $I_1(t; \Gamma)$ and $I_3(t; \Gamma)$ may be positive, cf. Subsection~\ref{subsec4.2}.
\end{remark}
\begin{remark}
The condition~\eqref{cond:c_0} is not strict in the practical computation, where $\alpha$ and $c_0$ are chosen typically as, e.g., $\alpha = 0.01$ and $c_0 = 0.9$~\cite{Kanayamatsunami}.
These satisfy~\eqref{cond:c_0}, since $\sqrt{2/\alpha} \, (1-\alpha) \approx 14$.
\end{remark}
\cK
\section{Numerical results by a finite difference scheme}\label{sec4}
In this section, we present numerical results by a finite difference scheme for problem  \eqref{prob}--\eqref{IC} with $\Omega=(0,L)^2$ for a positive constant $L$, $T=100$, $\zeta=a \cK > 0$,  $\mu=10^3$, $g=9.8\times10^{-3}$, $\rho=10^{12}$, $\eta^0=c_1 {\cK \exp(-100|x-p|^2)} \ (c_1>0, p \in\Omega)$.
These values are in km~(length), kg~(mass) and s~(time).
We set $\Gamma_S=\emptyset$ for simplicity.
\cK
We consider five cases of~$\Gamma_T$: 
\begin{align*}
(i) & \ \Gamma_T=\emptyset, 
\qquad
(ii) \ \Gamma_T=\Gamma_{\rm top}, 
\qquad
(iii) \ \Gamma_T=\Gamma_{\rm top}\cup\Gamma_{\rm right}\cup \{(L, L)\}, \\
(iv) & \ \Gamma_T=\Gamma_{\rm top}\cup\Gamma_{\rm right}\cup\Gamma_{\rm left}\cup \{(L, L)\}\cup \{(0, L)\}, 
\qquad
(v) \ \Gamma_T=\Gamma,
\end{align*}
for
$\Gamma_{\rm top} \defeq \{(x_1,L);0<x_1<L\}$,
$\Gamma_{\rm right} \defeq \{(L, x_2);0<x_2<L\}$,
$\Gamma_{\rm left} \defeq \{(0, x_2);0<x_2<L\}$,
and set $\Gamma_D=\Gamma\setminus \Gamma_T$. For the above cases $(ii)$-$(v)$, $c_0=0.9$ is taken  following~\cite{Kanayamatsunami}. 
\cK
\subsection{A finite difference scheme}\label{subsec4.1}
 Let $N \in\mathbb{N}$ and $\Delta t > 0$ be given, and let $h \defeq L/N$ and $ N_T \defeq \lfloor{{T}/{\Delta t}}\rfloor$, $x_{i,j} \defeq (ih,jh)^T \in \mathbb{R}^2~(i, j \in {\mathbb Z})$, $\Omega_h \defeq \{x_{i,j}\in \Omega; i, j \in {\mathbb Z}\}$, $\overline{\Omega}_h \defeq \{x_{i,j} \in \overline{\Omega}; i, j \in {\mathbb Z}\}$, $\Gamma_{hD} \defeq \{x_{i,j}\in \overline{\Gamma}_{D}; i, j \in {\mathbb Z}\}$, $\Gamma_{hT} \defeq \{x_{i,j}\in {\Gamma}_{T}; i, j \in {\mathbb Z}\}$.
Let $u_h^0: \overline{\Omega}_h \to \mathbb{R}^2$ and $\phi_h^0: \overline{\Omega}_h \to \mathbb{R}$ be given approximate functions of $u^0$ and $\phi^0$, respectively.
Our finite difference scheme is to find $\{(\phi_h^k,  u_h^k)(x_{i,j});~x_{i,j} \in \overline{\Omega}_h, k=1, \ldots, N_T\}$ such that, for $k=0, \ldots, N_T-1$,
\begin{align}
\left\{
\ 
\begin{aligned}
& \frac{\phi^{k+1}_h-\phi^{k}_h}{\Delta t}+ \left(\nabla_h\cdot  u^k_h\right)\phi^{k}_h +  u^k_h\cdot\nabla_h^{up}\phi^{k}_h = 0 && \text{on} \ \overline{\Omega}_h,\\
& \rho\phi^k_h\left(\frac{ u^{k+1}_h- u^{k}_h}{\Delta t}+(u^{k}_h\cdot\nabla_h)u^{k}_h \right) \\
& \qquad\qquad -2\mu\nabla_h\cdot\left(\phi^k_h {\cK D_h}( u^k_h)\right) + \rho g \phi^k_h\nabla_h \eta^k_h = 0 && \text{on} \ {\Omega}_h,\\
& \phi^k_h = \eta^k_h + \zeta && \text{on} \ {\Omega}_h,\\
& u_h^{k+1} =  u_D^{k+1} && \text{on} \ {\Gamma}_{hD},\\
& u_h^{k+1} = c \frac{\phi_h^{k+1}-\zeta}{\phi_h^{k+1}}{  n} && \text{on} \ {\Gamma}_{hT},
\end{aligned}
\right.
\label{scheme}
\end{align}
\cK
where $\nabla_h = (\nabla_{h1}, \nabla_{h2})^T$ and $\nabla_h^{up} = (\nabla_{h1}^{up}, \nabla_{h2}^{up})^T$ represent the (standard) central and upwind (with respect to $u_h^k$) difference operators, respectively,
and $D_h(v_h) \defeq [ (\nabla_h v_h) + (\nabla_h v_h)^T]/2: \ol{\Omega}_h \to \R^{2 \times 2}_{\rm sym}$ for $v_h: \ol{\Omega}_h \to \R^2$.
If the required points for the operators~$\nabla_h$ and~$\nabla_h^{up}$ are not in~$\ol{\Omega}_h$, one sided difference is used.
\cK
\subsection{Numerical results for five cases of boundary settings}\label{subsec4.2}
Numerical simulations are carried out {\cK by scheme~\eqref{scheme} for} $L=10$, $a=1$, $u^0=0$, $c_1=0.01$, ${  p}=(5,5)^T$, $N = 1,000$ and $\Delta t = 0.05~(N_T=2,000)$.
\cK
Figure~\ref{Simulatedresult} shows color contours of~$\eta_h^k$ for $k = 0$, $500$, $1,000$, $1,500$ and $2,000$, which correspond to times $t=0, 25, 50, 75$ and $100$, respectively,
\cK
where~$(i)$-$(v)$ represent simulated results for the cases~$(i)$-$(v)$ stated at the beginning of this section.
It can be clearly found that the artificial reflection is almost removed on the transmission boundaries {\cK for the cases $(ii)$-$(v)$.} 
\begin{figure}[!htbp]
\begin{minipage}[b]{0.08\linewidth}
\centering
{$(i)$}
\vspace{1cm}
\end{minipage}
\begin{minipage}[b]{0.155\linewidth}
\centering
\includegraphics[width=0.99\linewidth]{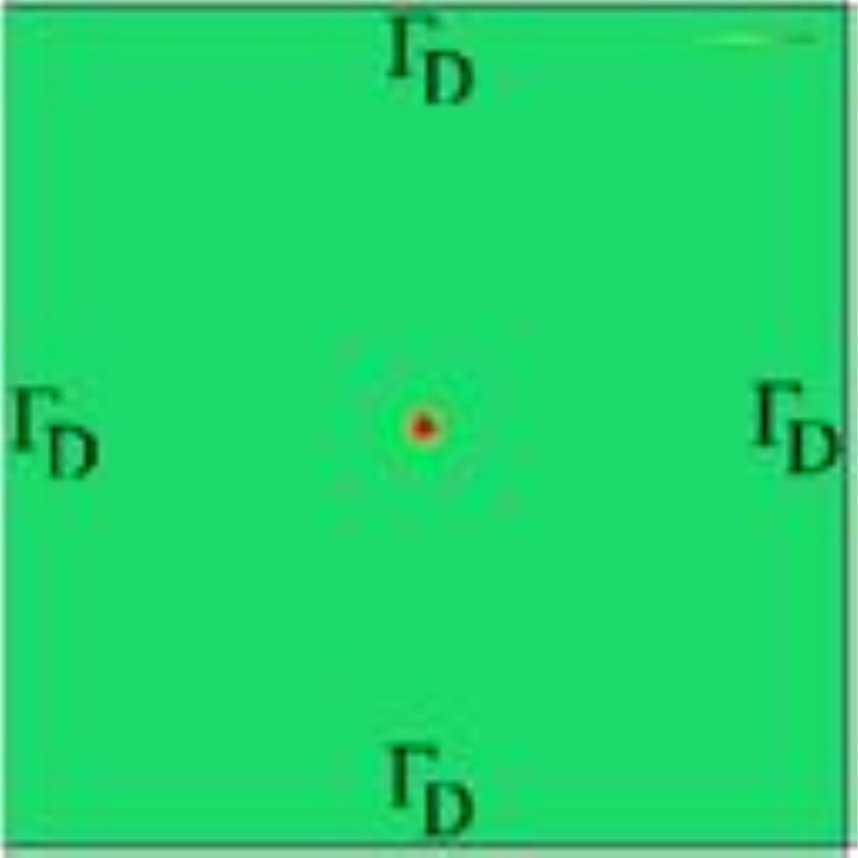}
{$t=0$}
\end{minipage}
\begin{minipage}[b]{0.155\linewidth}
\centering
\includegraphics[width=0.99\linewidth]{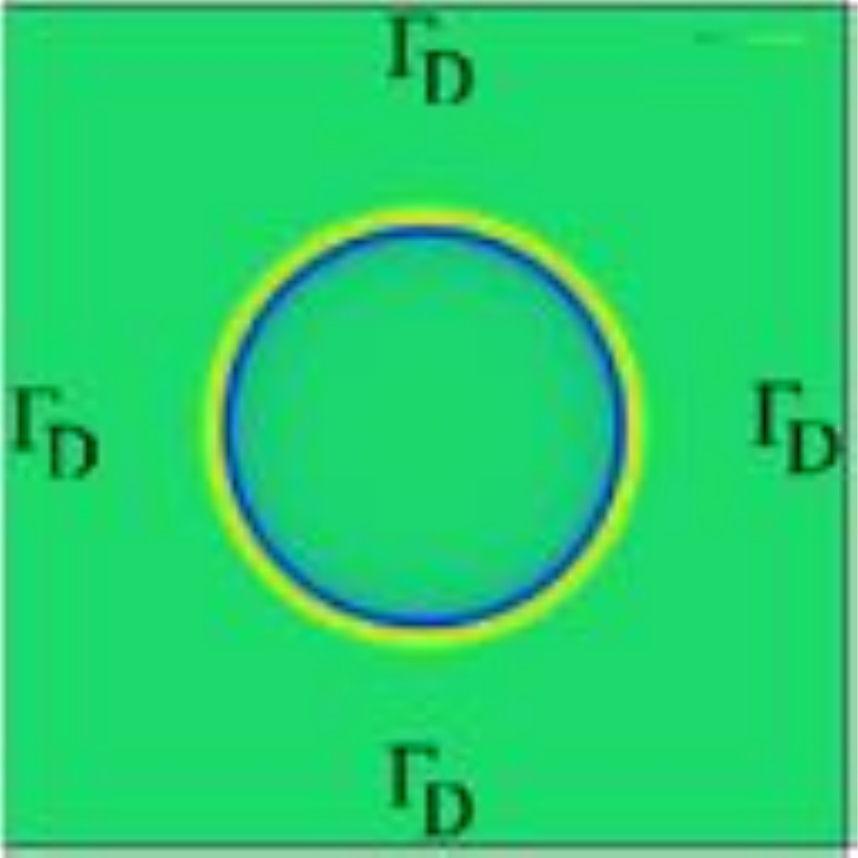}
{$t=25$}
\end{minipage}
\begin{minipage}[b]{0.155\linewidth}
\centering
\includegraphics[width=0.99\linewidth]{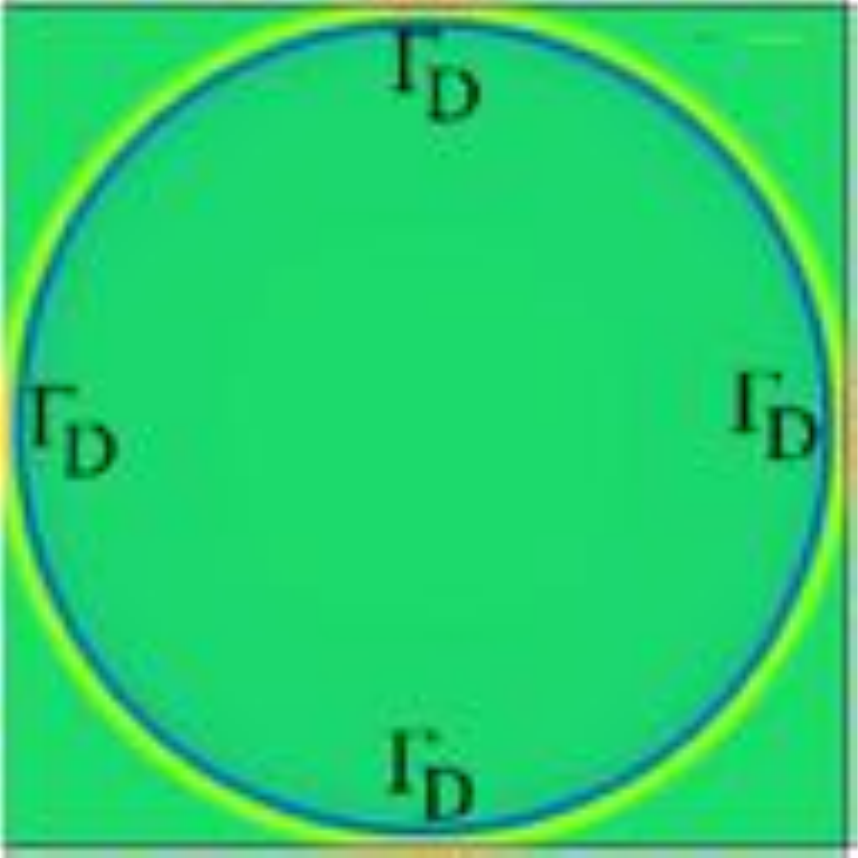}
{$t=50$}
\end{minipage}
\begin{minipage}[b]{0.155\linewidth}
\vspace{0.5cm}
\includegraphics[width=0.99\linewidth]{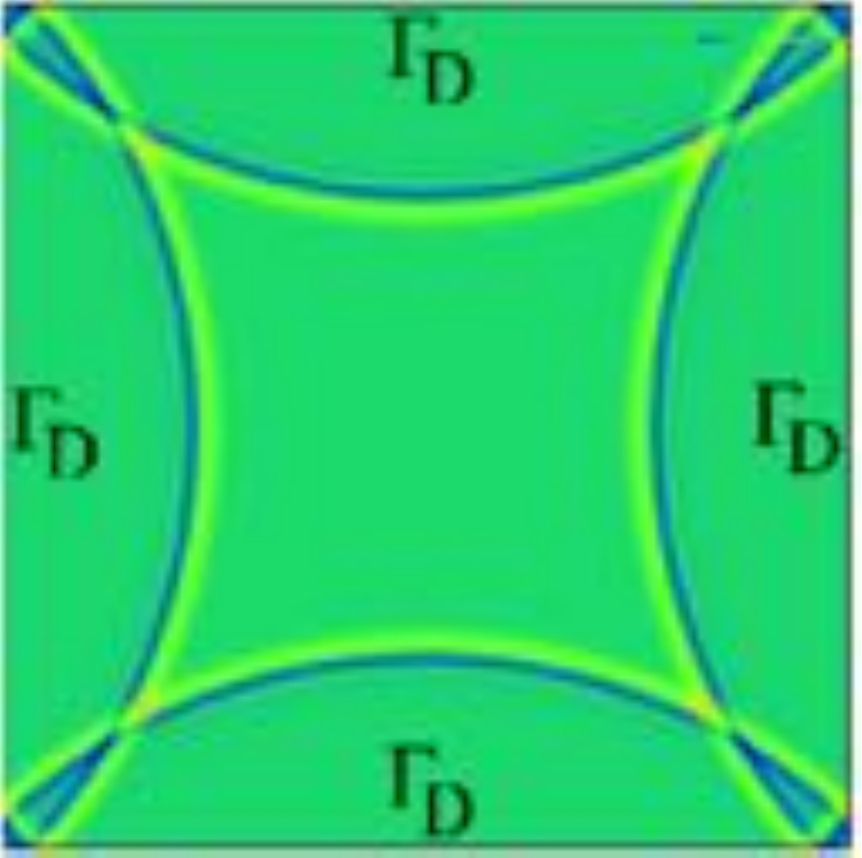}
{$t=75$}
\end{minipage}
\begin{minipage}[b]{0.155\linewidth}
\vspace{0.5cm}
\centering
\includegraphics[width=0.99\linewidth]{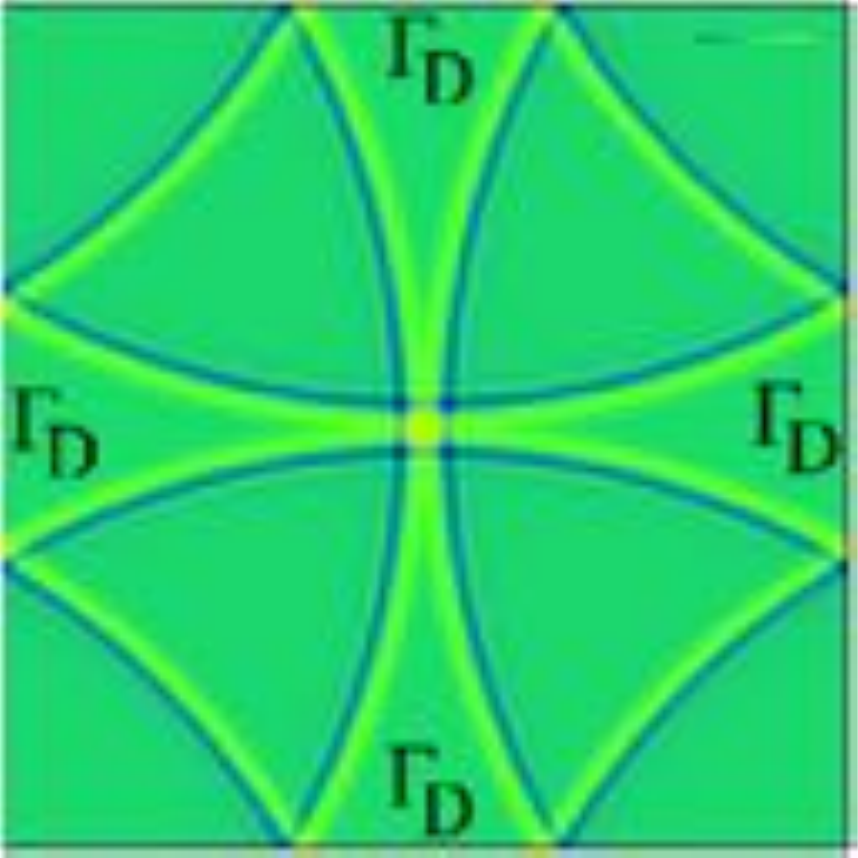}
{$t=100$}
\end{minipage}
\begin{minipage}[b]{0.062\linewidth}
\vspace{0.5cm}
\centering
\includegraphics[width=0.99\linewidth]{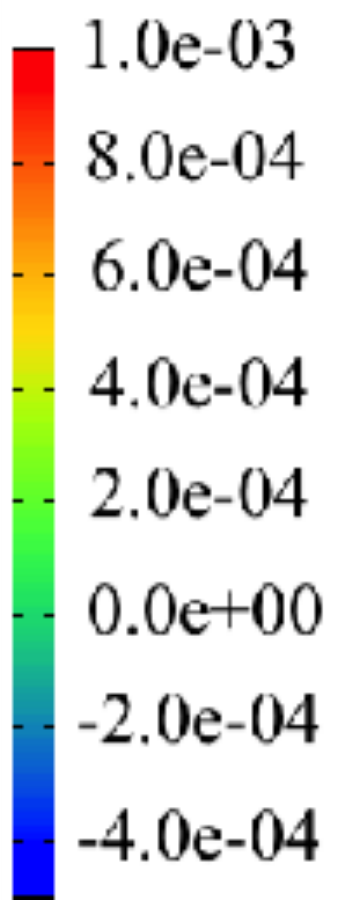}
\vspace{0.00cm}
\end{minipage}

\begin{minipage}[b]{0.08\linewidth}
\centering
{$(ii)$}
\vspace{1cm}
\end{minipage}
\begin{minipage}[b]{0.155\linewidth}
\centering
\includegraphics[width=0.99\linewidth]{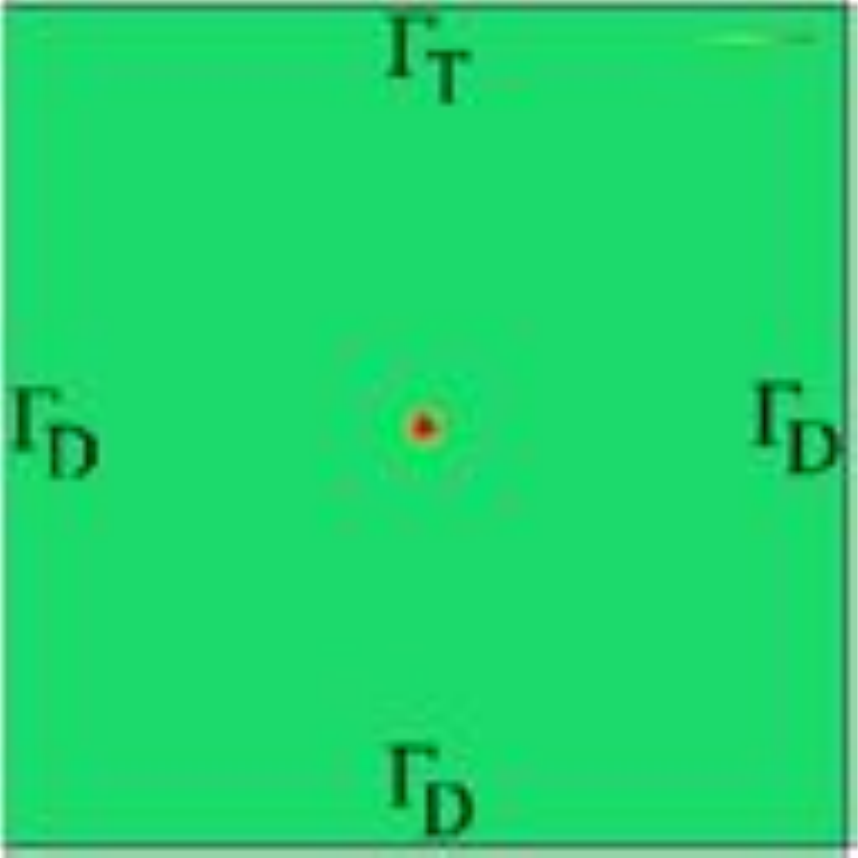}
{$t=0$}
\end{minipage}
\begin{minipage}[b]{0.155\linewidth}
\centering
\includegraphics[width=0.99\linewidth]{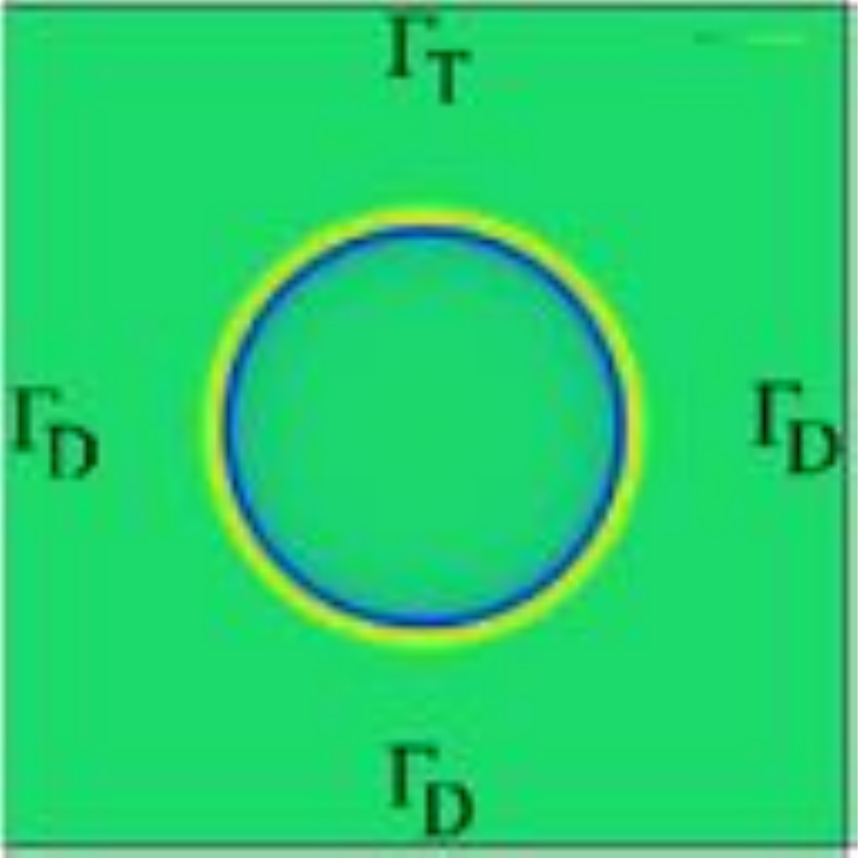}
{$t=25$}
\end{minipage}
\begin{minipage}[b]{0.155\linewidth}
\centering
\includegraphics[width=0.99\linewidth]{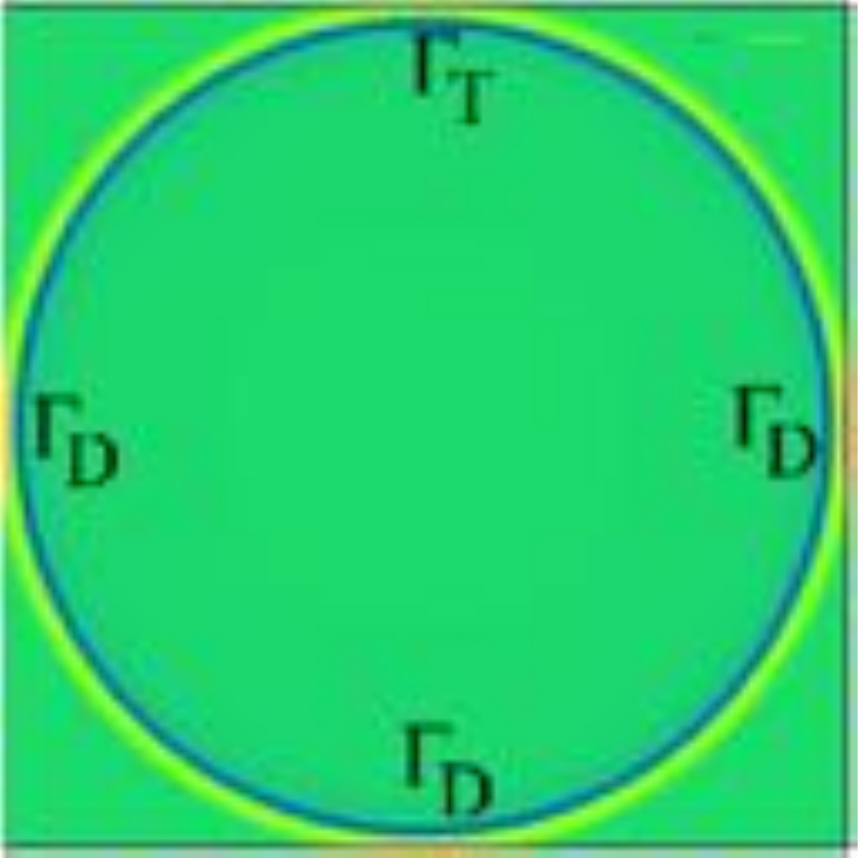}
{$t=50$}
\end{minipage}
\begin{minipage}[b]{0.155\linewidth}
\vspace{0.5cm}
\centering
\includegraphics[width=0.99\linewidth]{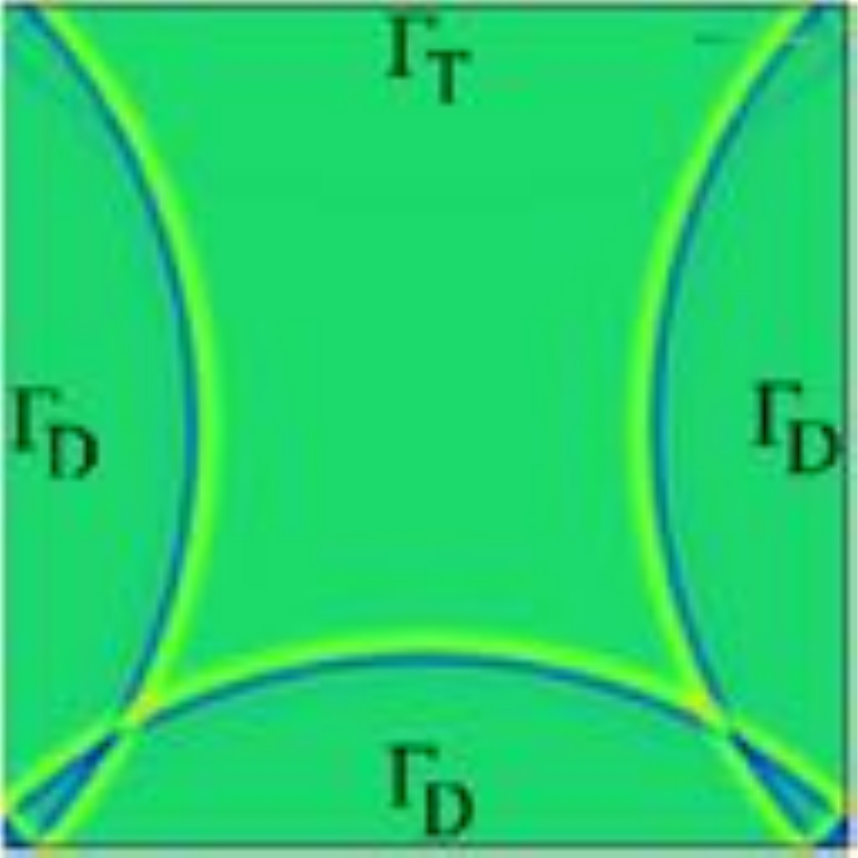}
{$t=75$}
\end{minipage}
\begin{minipage}[b]{0.155\linewidth}
\vspace{0.5cm}
\centering
\includegraphics[width=0.99\linewidth]{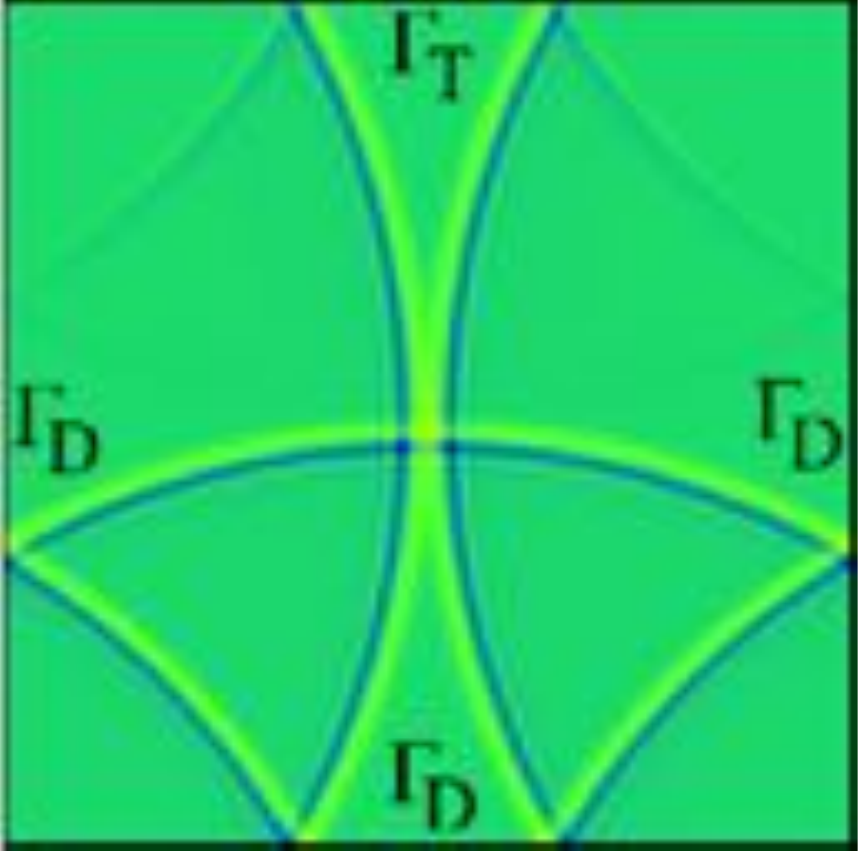}
{$t=100$}
\end{minipage}
\begin{minipage}[b]{0.062\linewidth}
\vspace{0.5cm}
\centering
\includegraphics[width=0.99\linewidth]{Colour_bar}
\vspace{0.00cm}
\end{minipage}
  
\begin{minipage}[b]{0.08\linewidth}
\centering
{$(iii)$}
\vspace{1cm}
\end{minipage}
\begin{minipage}[b]{0.155\linewidth}
\centering
\includegraphics[width=0.99\linewidth]{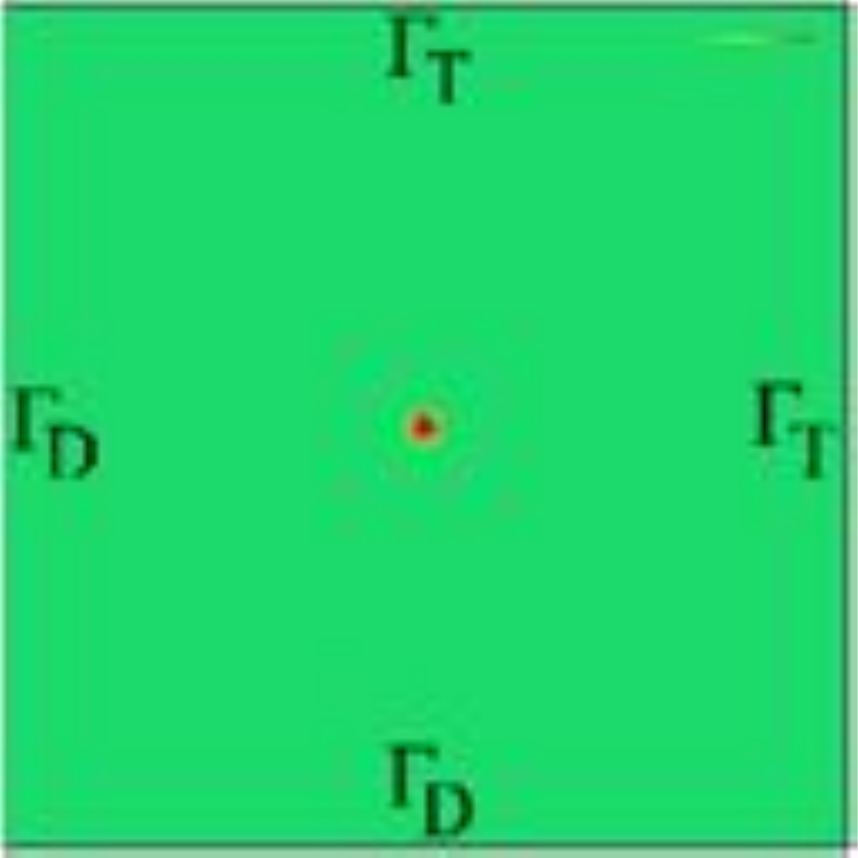}
{$t=0$}
\end{minipage}
\begin{minipage}[b]{0.155\linewidth}
\centering
\includegraphics[width=0.99\linewidth]{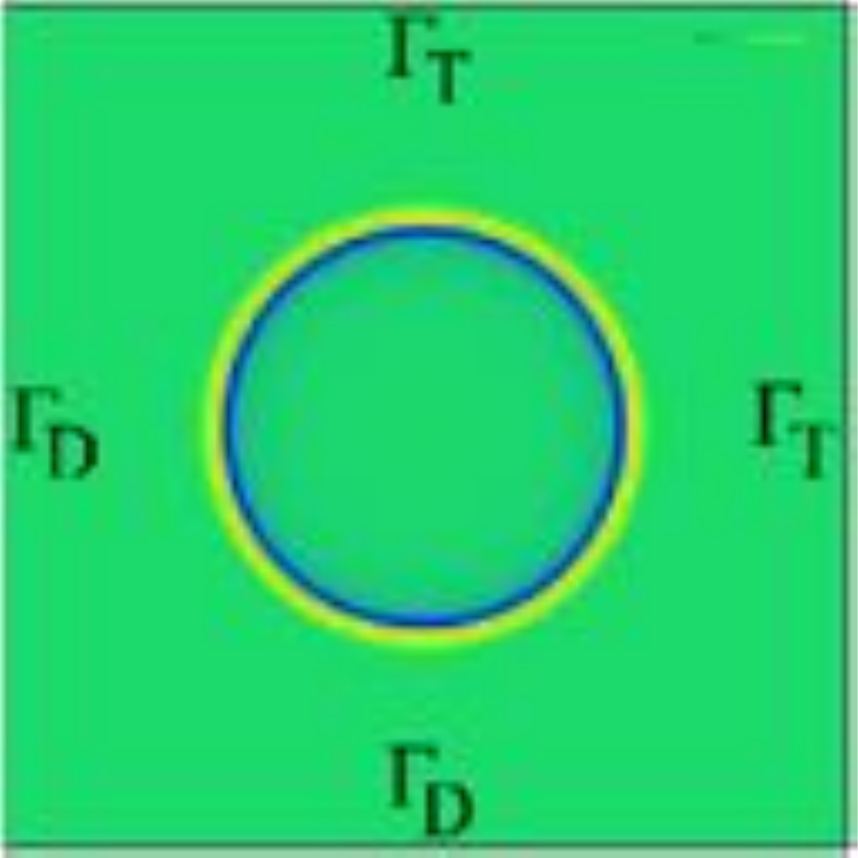}
{$t=25$}
\end{minipage}
\begin{minipage}[b]{0.155\linewidth}
\centering
\includegraphics[width=0.99\linewidth]{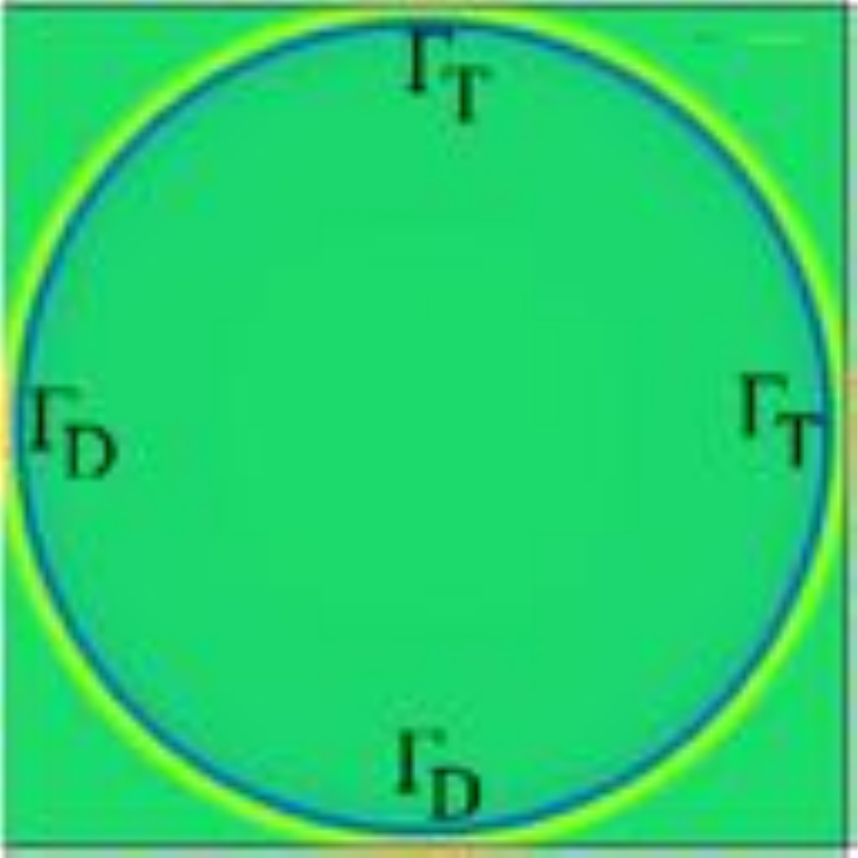}
{$t=50$}
\end{minipage}
\begin{minipage}[b]{0.155\linewidth}
\vspace{0.5cm}
\centering
\includegraphics[width=0.99\linewidth]{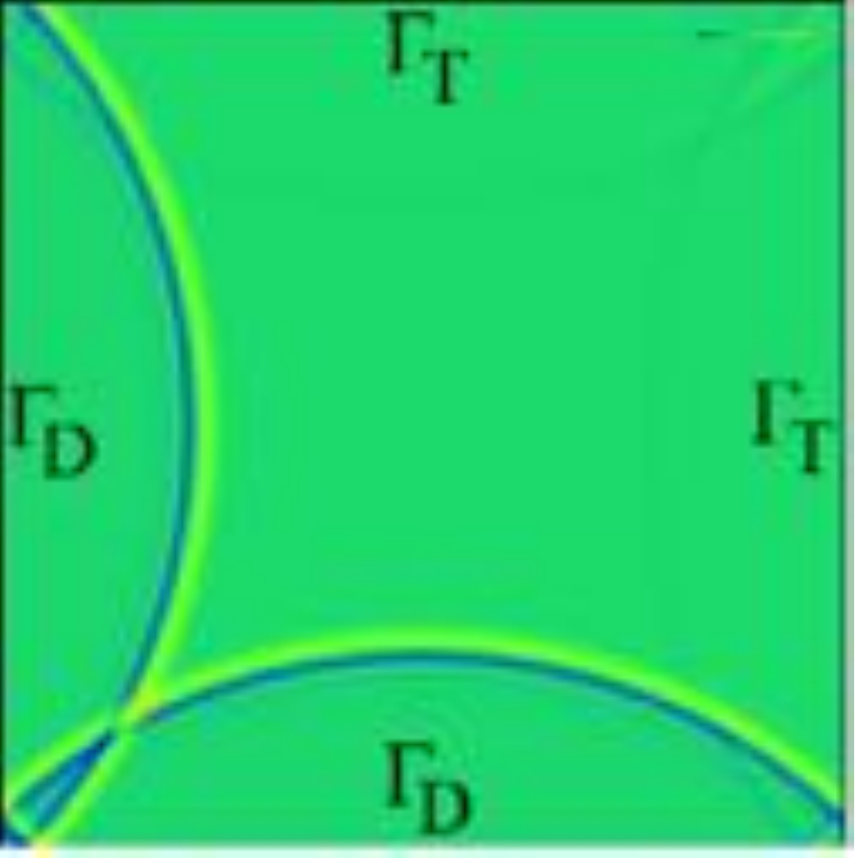}
{$t=75$}
\end{minipage}
\begin{minipage}[b]{0.155\linewidth}
\vspace{0.5cm}
\centering
\includegraphics[width=0.99\linewidth]{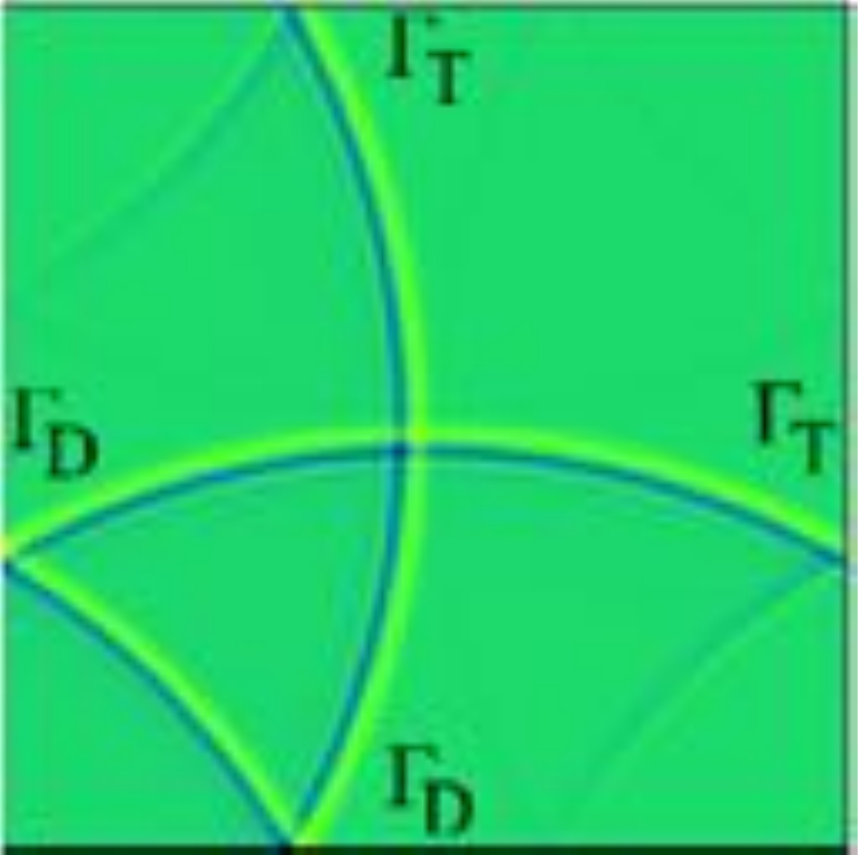}
{$t=100$}
\end{minipage}
\begin{minipage}[b]{0.062\linewidth}
\vspace{0.5cm}
\centering
\includegraphics[width=0.99\linewidth]{Colour_bar}
\vspace{0.00cm}
\end{minipage}

\begin{minipage}[b]{0.08\linewidth}
\centering
{$(iv)$}
\vspace{1cm}
\end{minipage}
\begin{minipage}[b]{0.155\linewidth}
\centering
\includegraphics[width=0.99\linewidth]{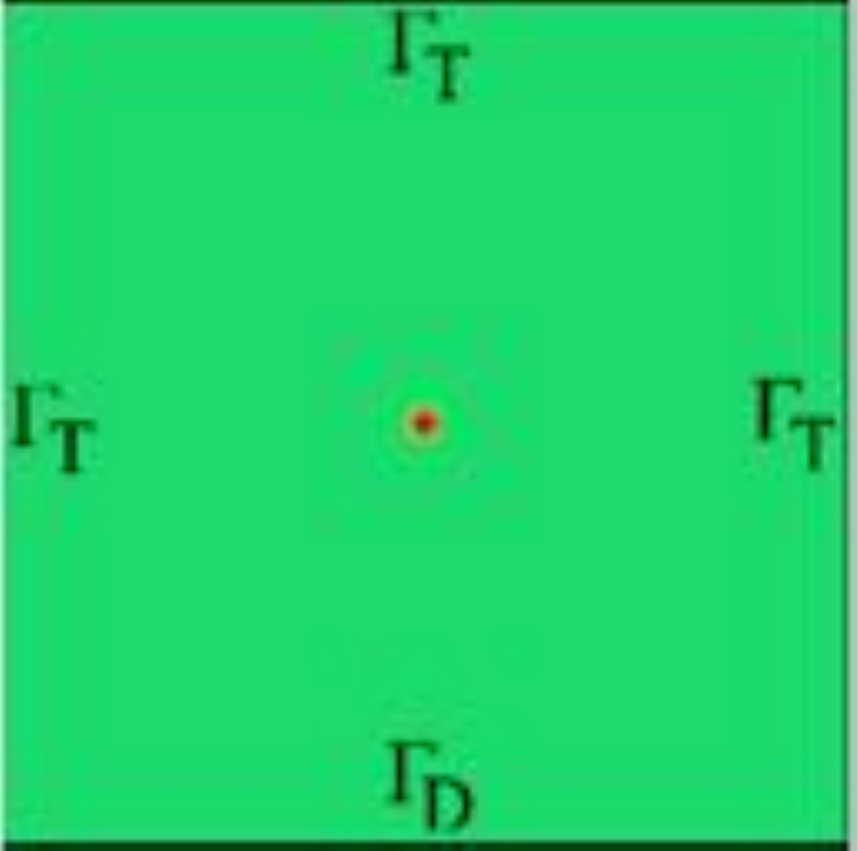}
{$t=0$}
\end{minipage}
\begin{minipage}[b]{0.155\linewidth}
\centering
\includegraphics[width=0.99\linewidth]{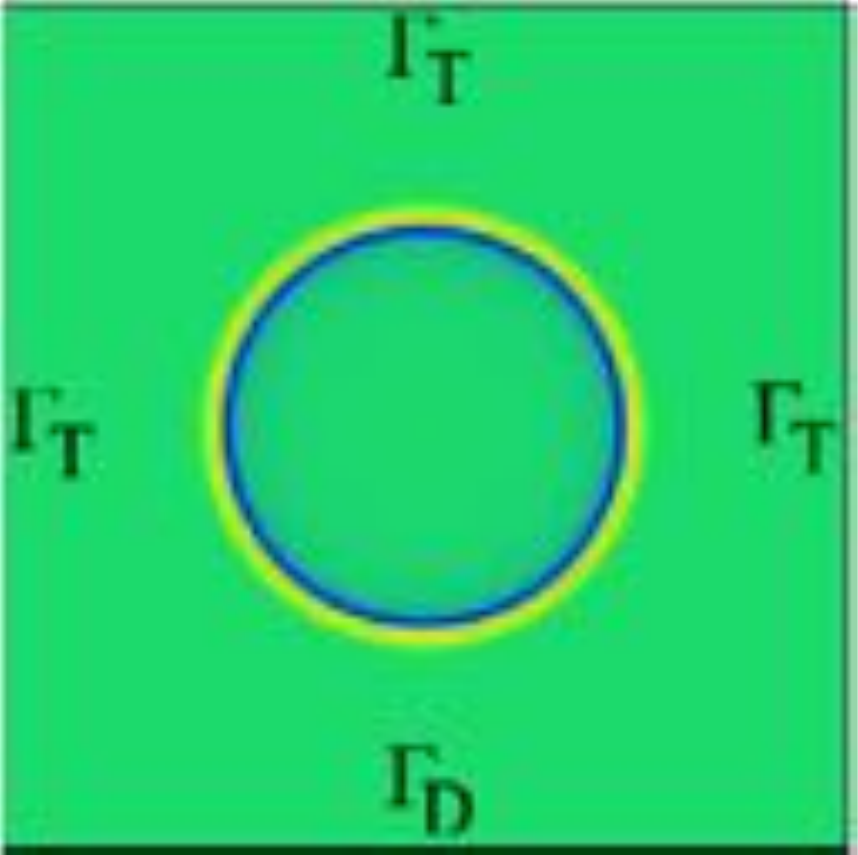}
{$t=25$}
\end{minipage}
\begin{minipage}[b]{0.155\linewidth}
\centering
\includegraphics[width=0.99\linewidth]{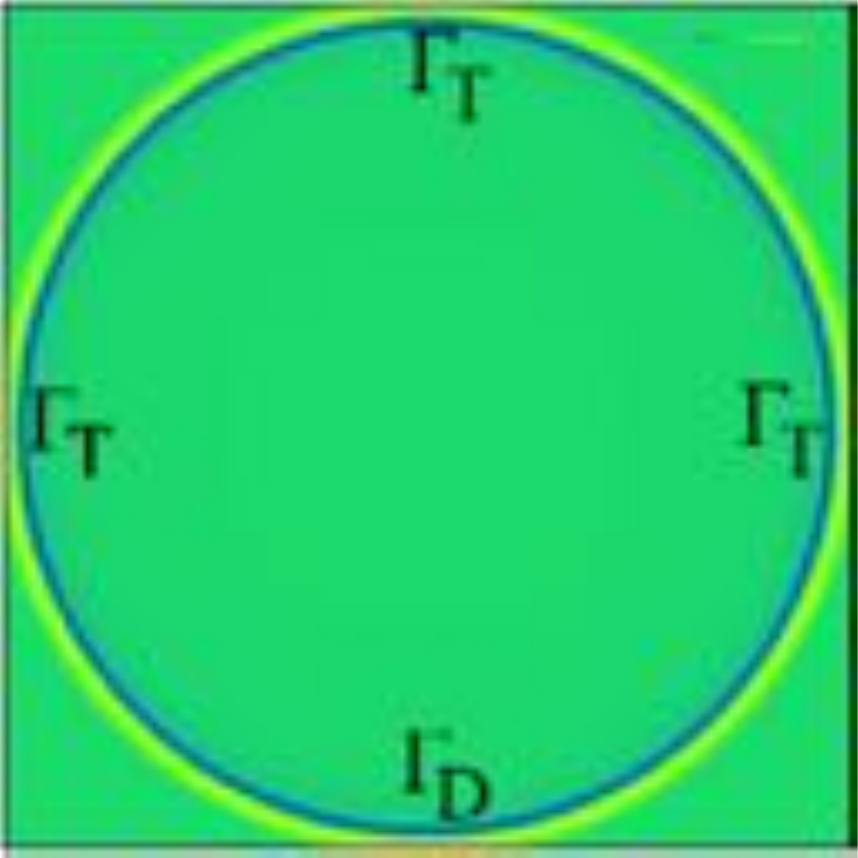}
{$t=50$}
\end{minipage}
\begin{minipage}[b]{0.155\linewidth}
\vspace{0.5cm}
\centering
\includegraphics[width=0.99\linewidth]{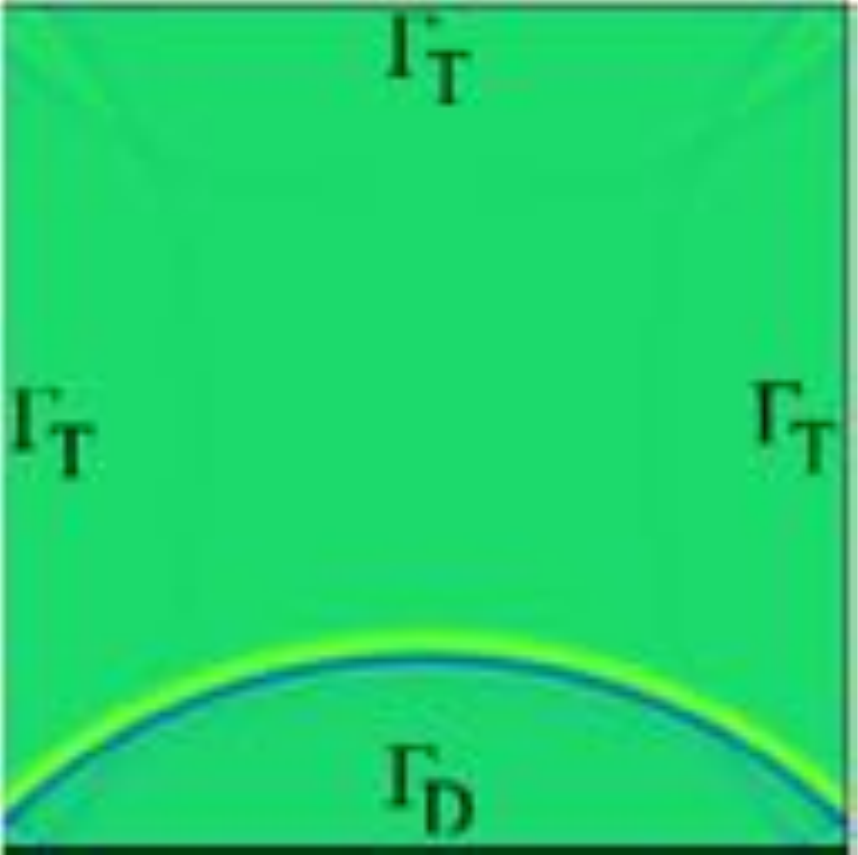}
{$t=75$}
\end{minipage}
\begin{minipage}[b]{0.155\linewidth}
\vspace{0.5cm}
\centering
\includegraphics[width=0.99\linewidth]{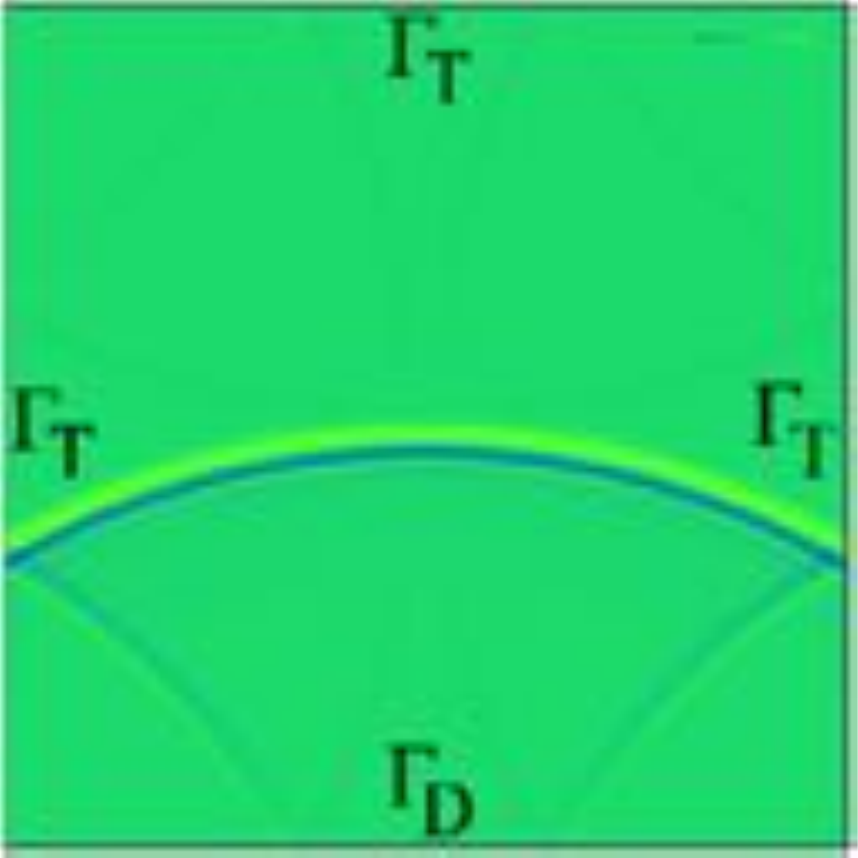}
{$t=100$}
\end{minipage}
\begin{minipage}[b]{0.062\linewidth}
\vspace{0.5cm}
\centering
\includegraphics[width=0.99\linewidth]{Colour_bar}
\vspace{0.00cm}
\end{minipage}

\begin{minipage}[b]{0.08\linewidth}
\centering
{$(v)$}
\vspace{1cm}
\end{minipage}
\begin{minipage}[b]{0.155\linewidth}
\centering
\includegraphics[width=0.99\linewidth]{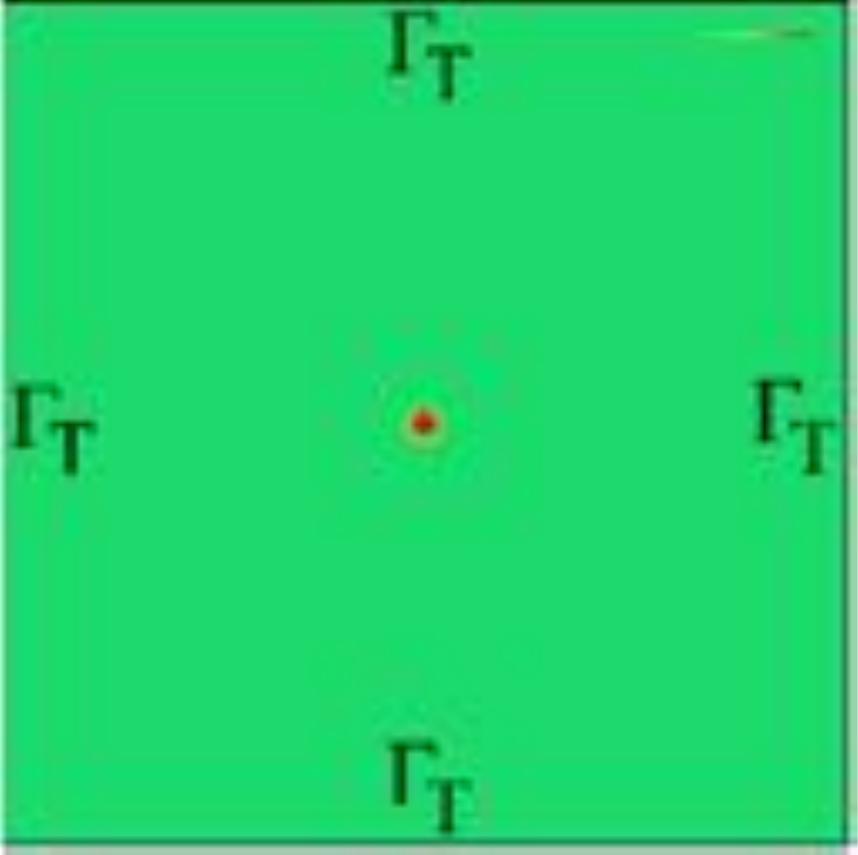}
{$t=0$}
\end{minipage}
\begin{minipage}[b]{0.155\linewidth}
\centering
\includegraphics[width=0.99\linewidth]{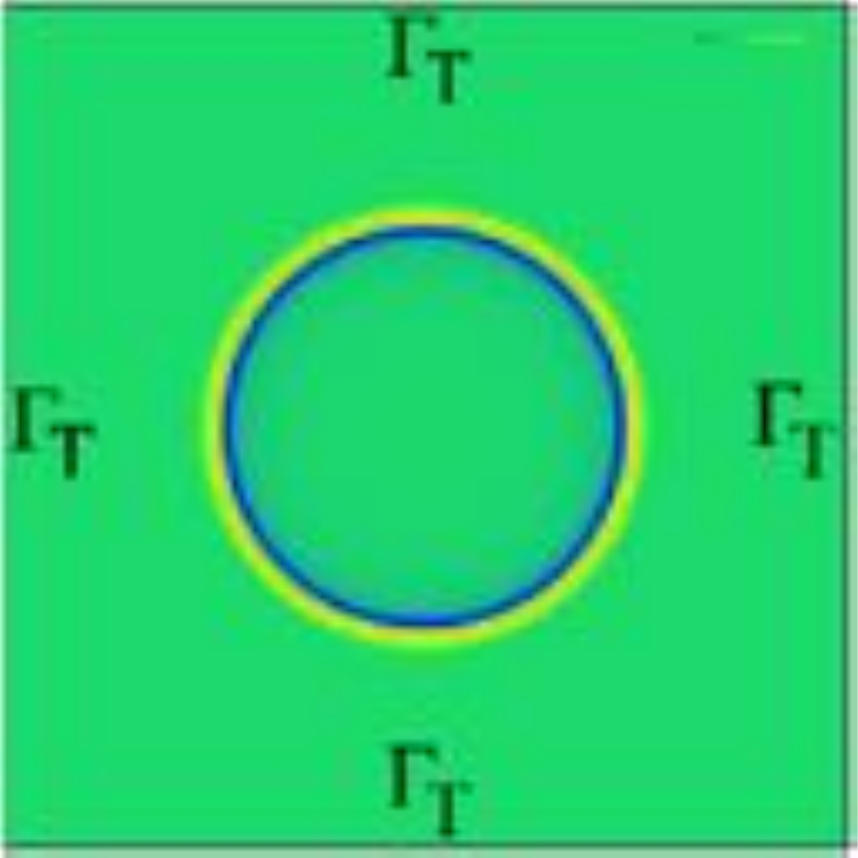}
{$t=25$}
\end{minipage}
\begin{minipage}[b]{0.155\linewidth}
\centering
\includegraphics[width=0.99\linewidth]{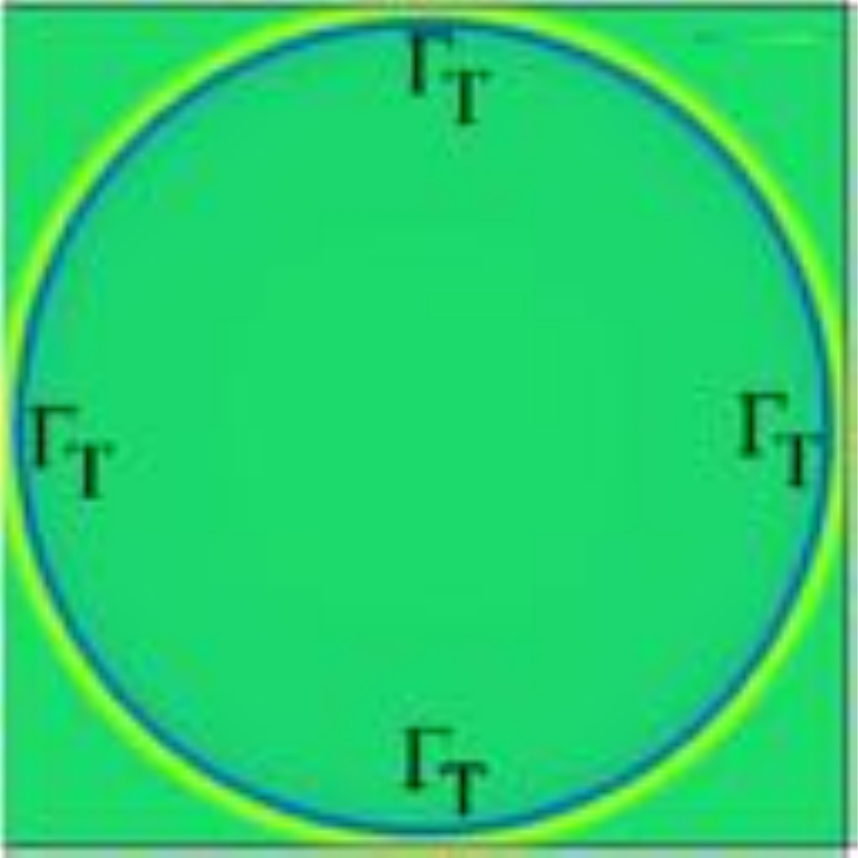}
{$t=50$}
\end{minipage}
\begin{minipage}[b]{0.155\linewidth}
\vspace{0.5cm}
\centering
\includegraphics[width=0.99\linewidth]{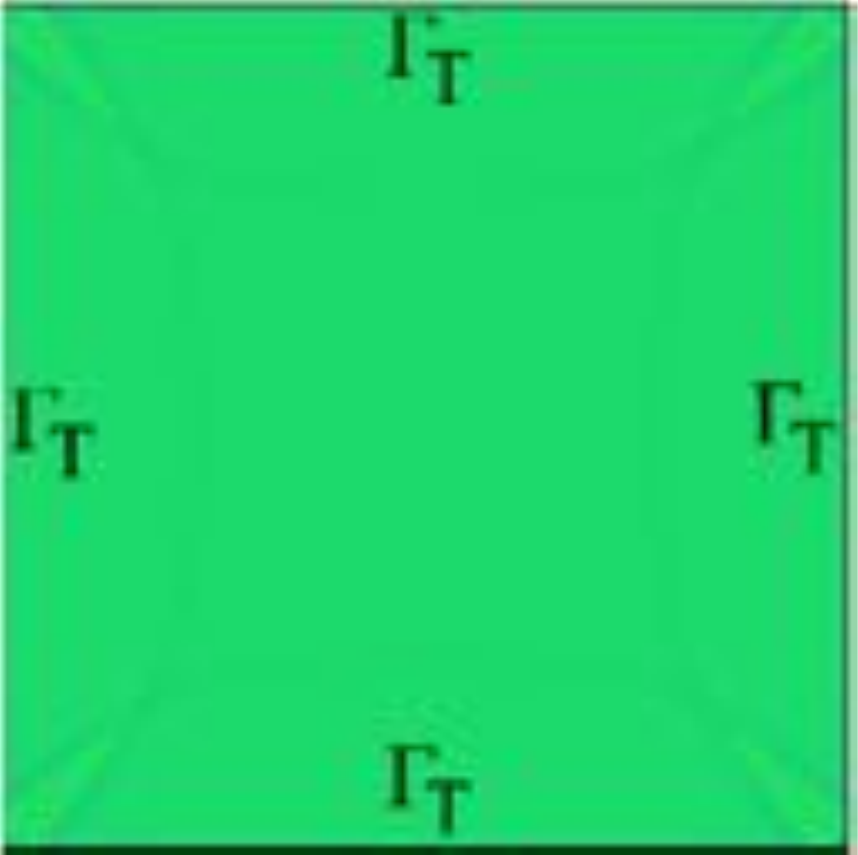}
{$t=75$}
\end{minipage}
\begin{minipage}[b]{0.155\linewidth}
\vspace{0.5cm}
\centering
\includegraphics[width=0.99\linewidth]{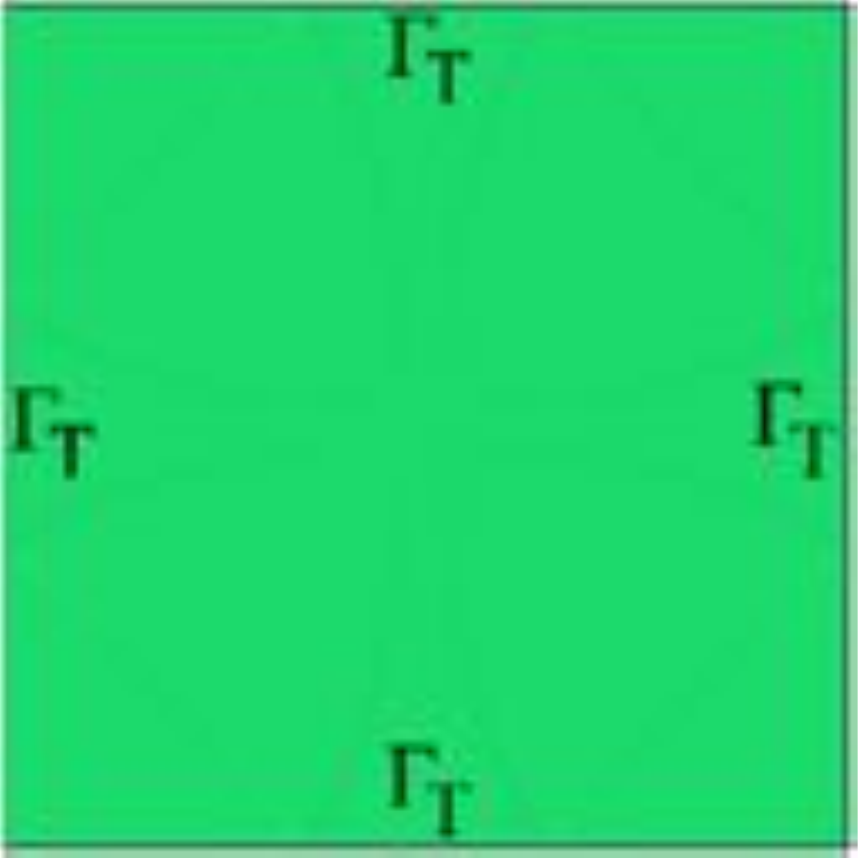}
{$t=100$}
\end{minipage}
\begin{minipage}[b]{0.062\linewidth}
\vspace{0.5cm}
\centering
\includegraphics[width=0.99\linewidth]{Colour_bar}
\vspace{0.00cm}
\end{minipage}
\caption{\cK Color contours of $\eta_h^k$ by finite difference scheme~\eqref{scheme} for the five cases~$(i)$-$(v)$ discussed in Subsection~\ref{subsec4.2}.}
\label{Simulatedresult}
\end{figure}
\cK
\subsection{\cK Numerical study of energy estimate}\label{subsec4.3}
\begin{figure}[!htbp]
\begin{minipage}{0.04\linewidth}
\centering
$(i)$
\vspace{5.6em}\\
$(ii)$
\vspace{5.6em}\\
$(iii)$
\vspace{5.6em}\\
$(iv)$
\vspace{5.6em}\\
$(v)$
\end{minipage}
\begin{minipage}{0.31\linewidth}
\centering
$E_h^k$ \\
\includegraphics[width=0.98\linewidth]{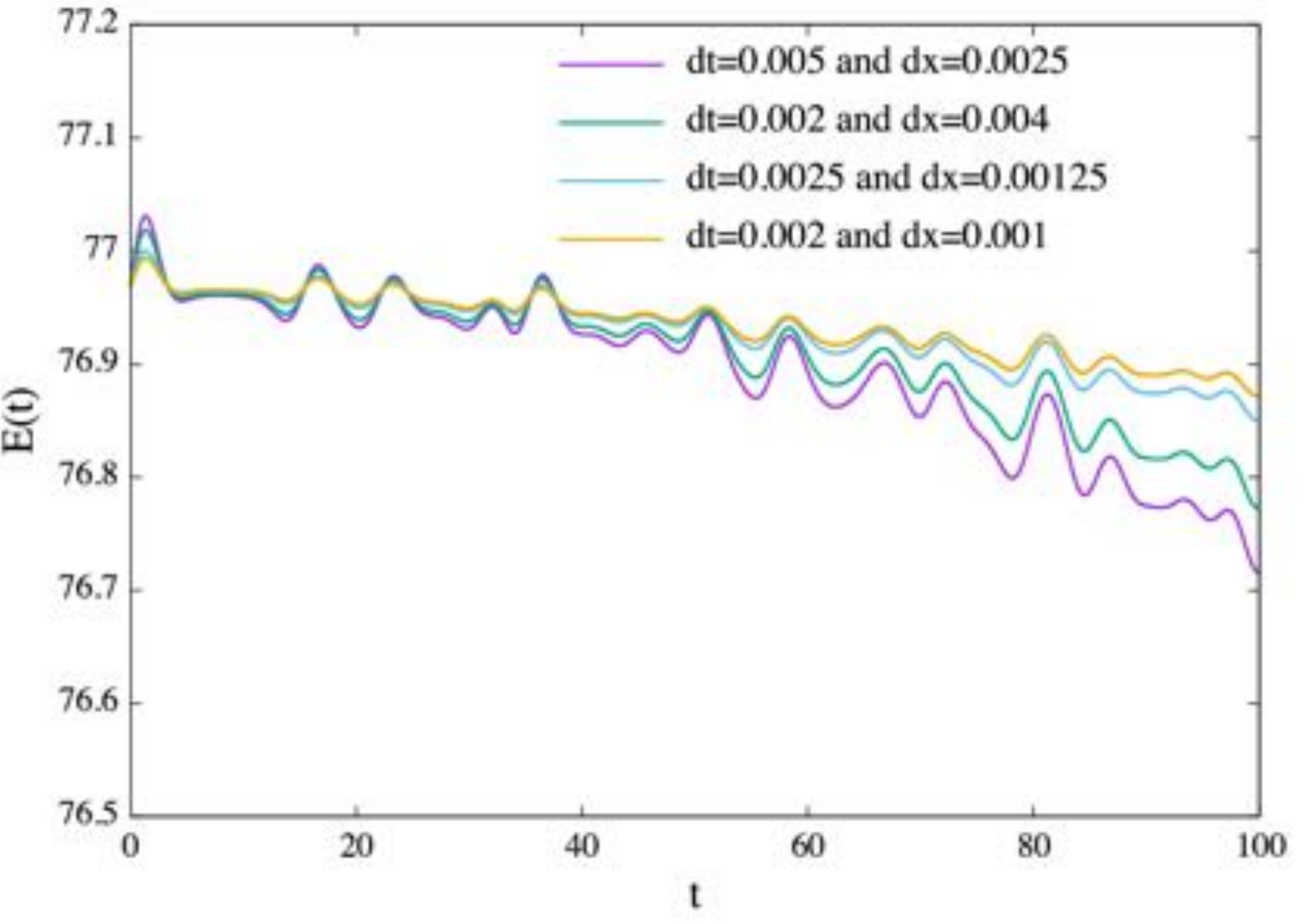} \\
\includegraphics[width=0.98\linewidth]{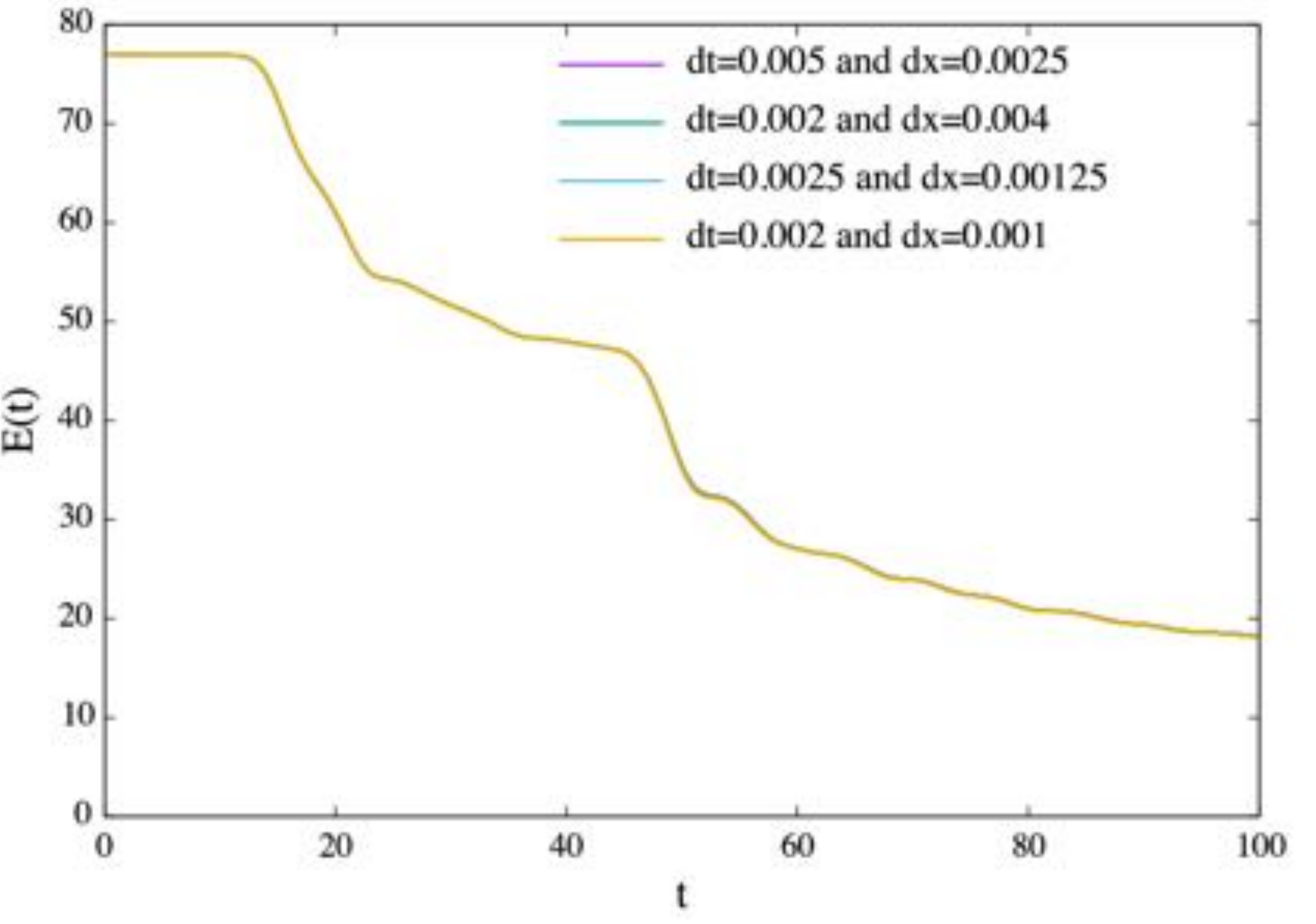} \\
\includegraphics[width=0.98\linewidth]{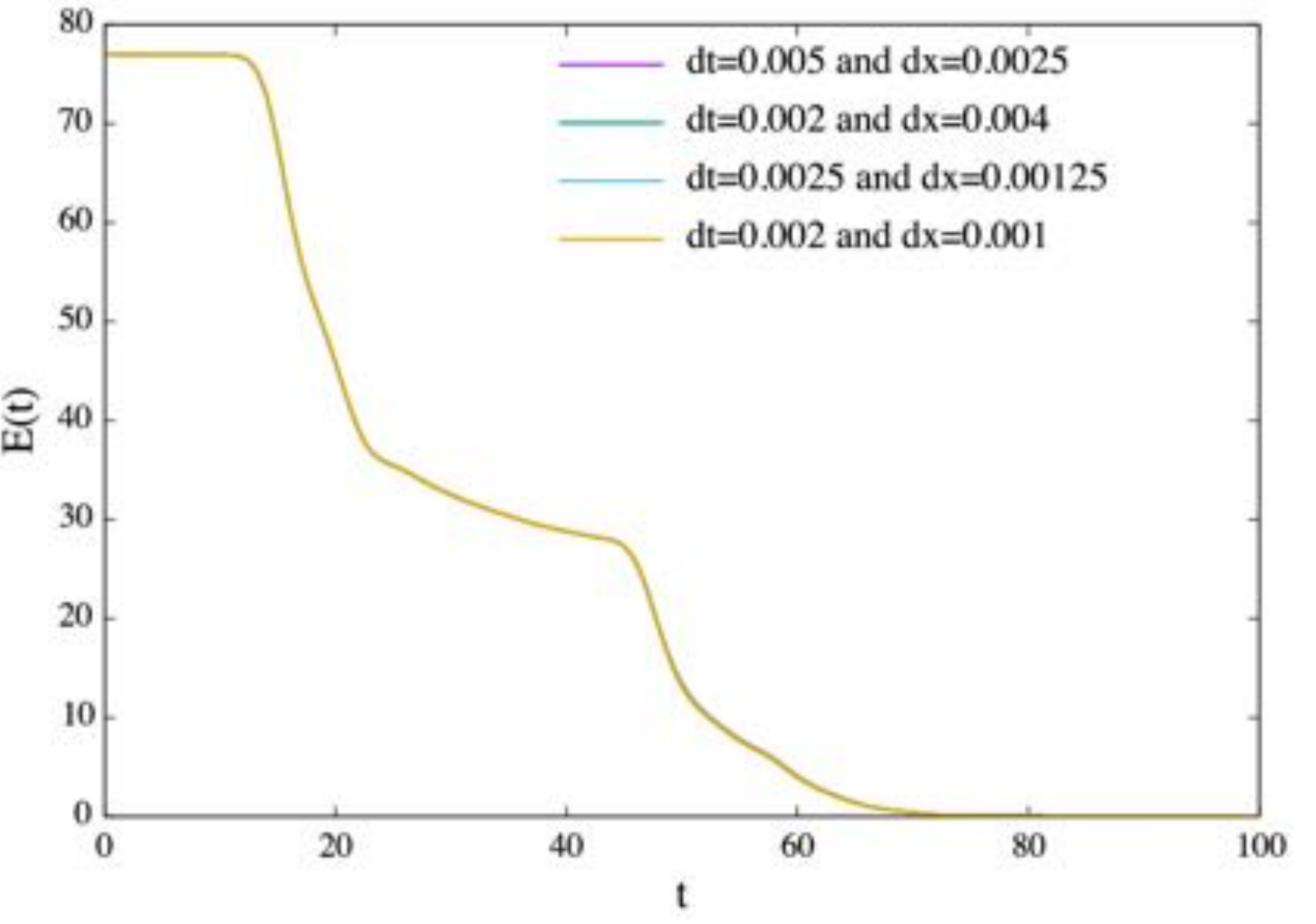} \\
\includegraphics[width=0.98\linewidth]{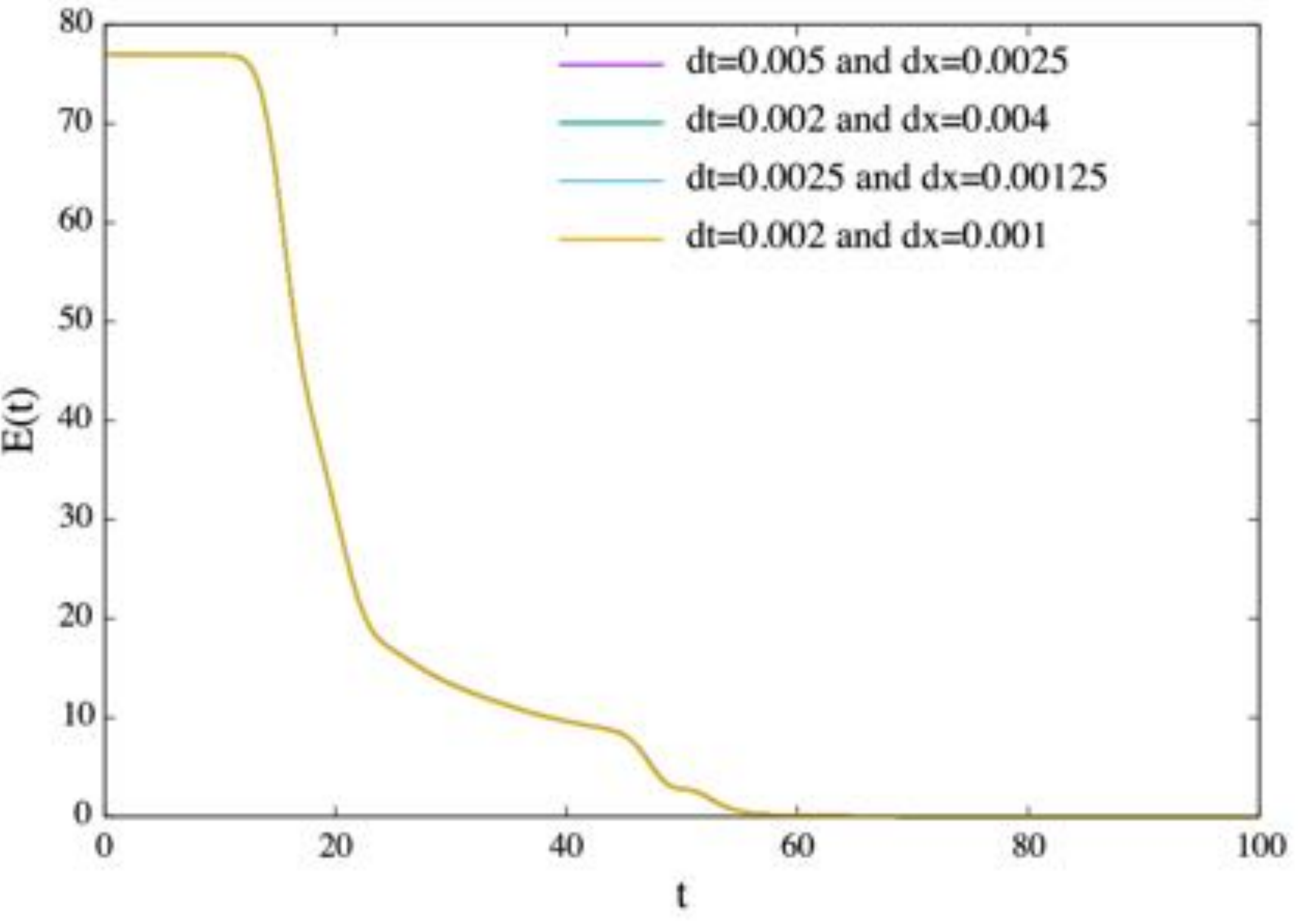} \\
\includegraphics[width=0.98\linewidth]{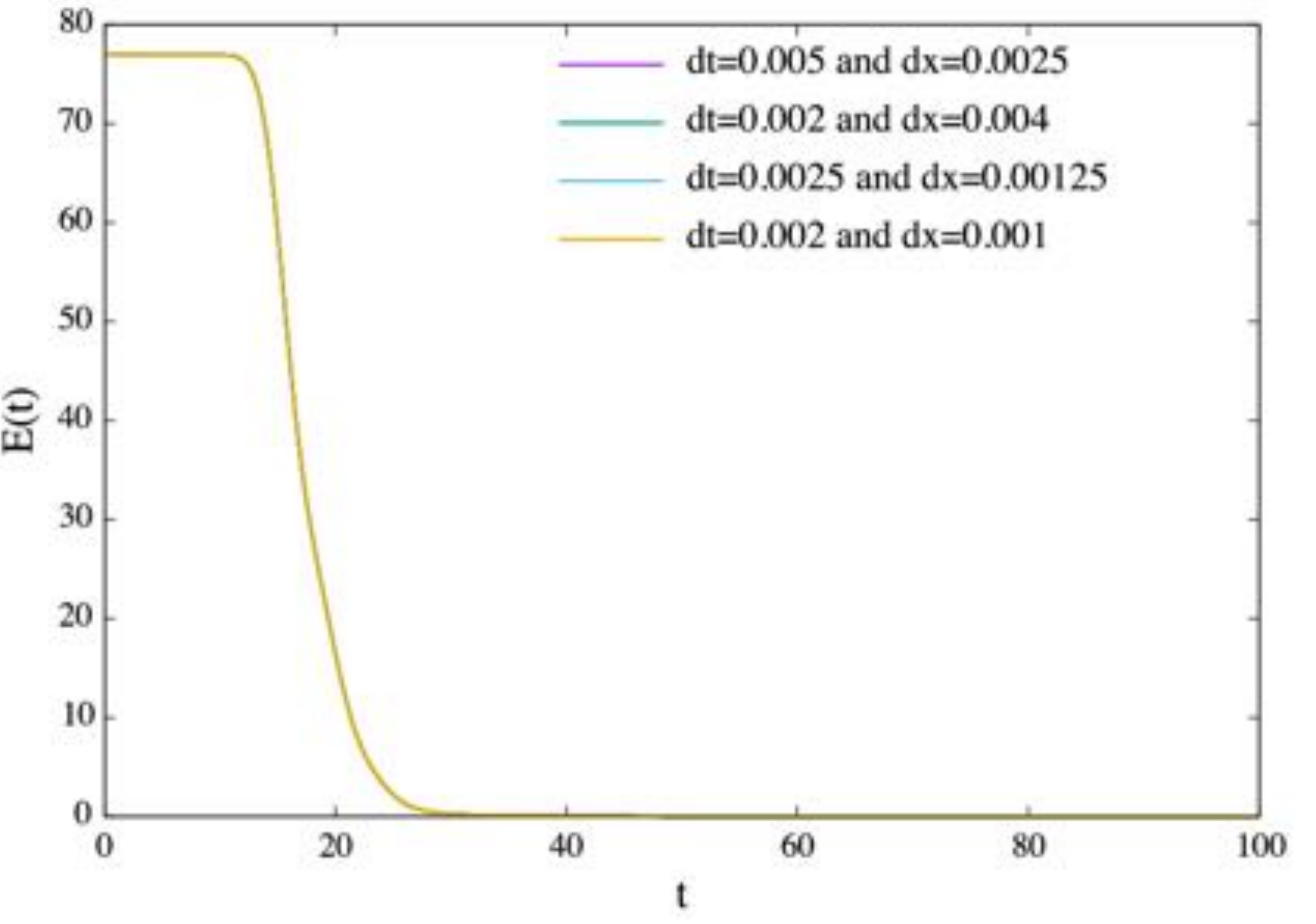}
\end{minipage}
\begin{minipage}{0.31\linewidth}
\centering
$\sum_{i=1}^4 I_{hi}^k \approx \tfrac{d}{dt} E(t^k)$ \\
\includegraphics[width=0.98\linewidth]{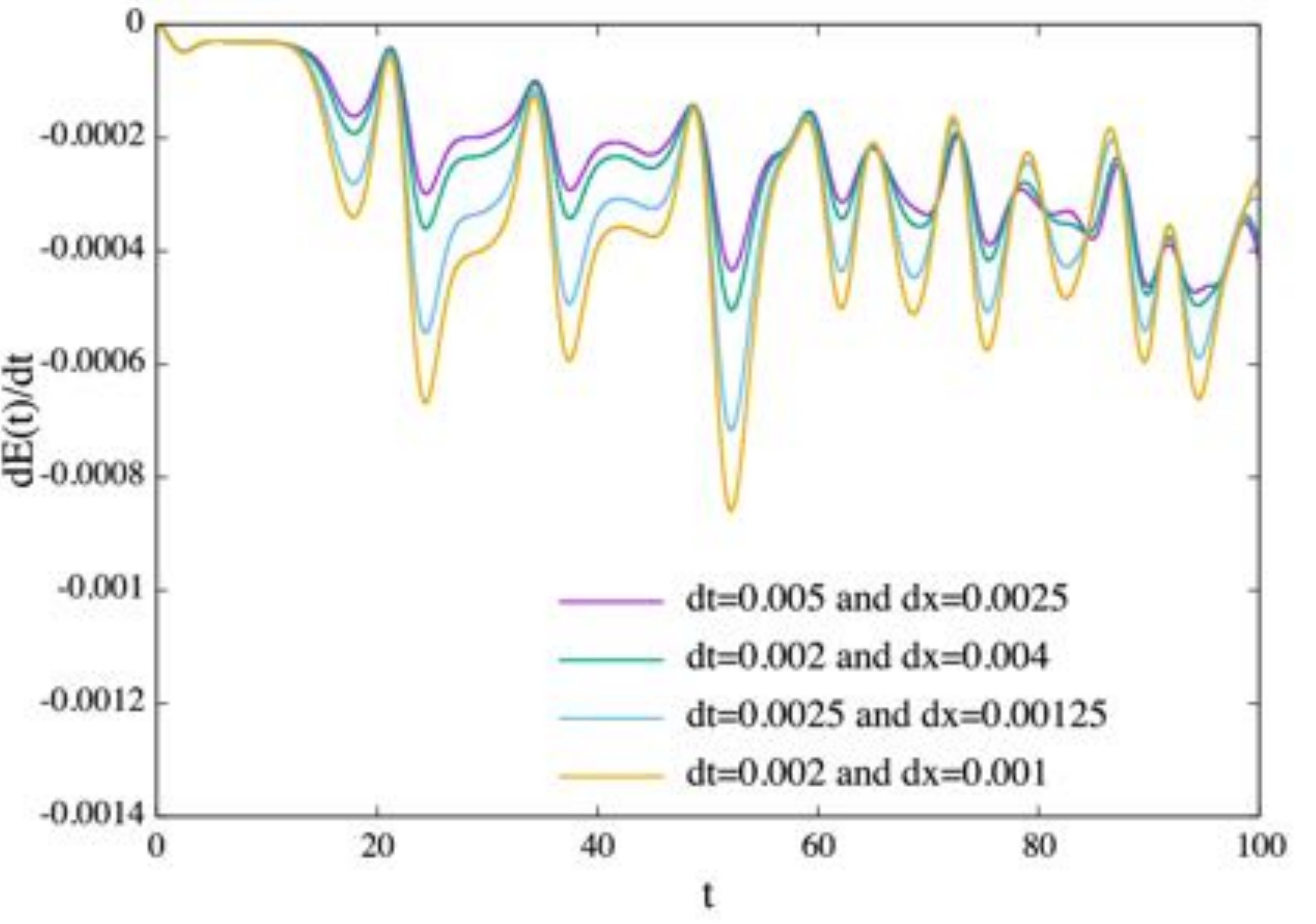} \\
\includegraphics[width=0.98\linewidth]{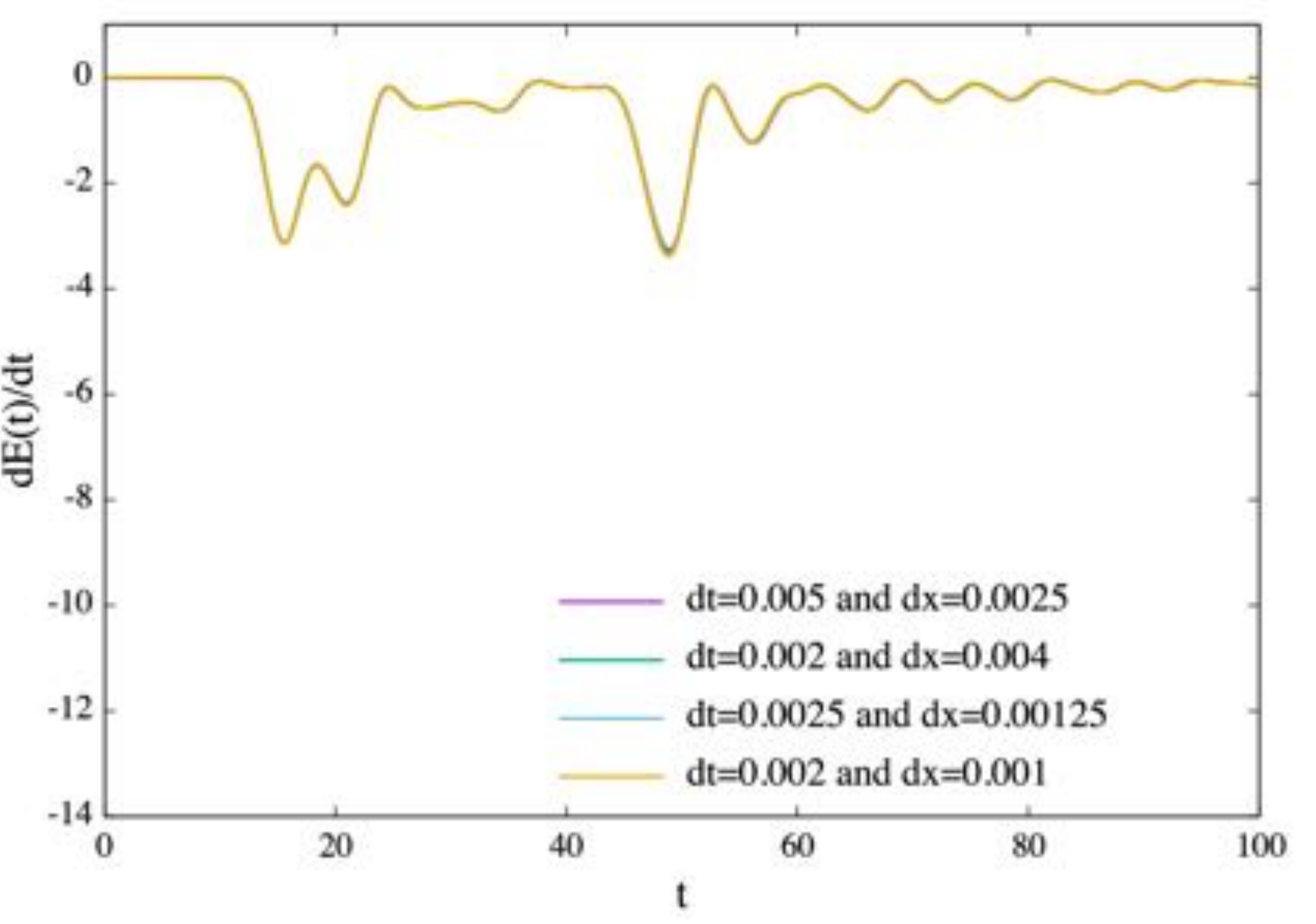} \\
\includegraphics[width=0.98\linewidth]{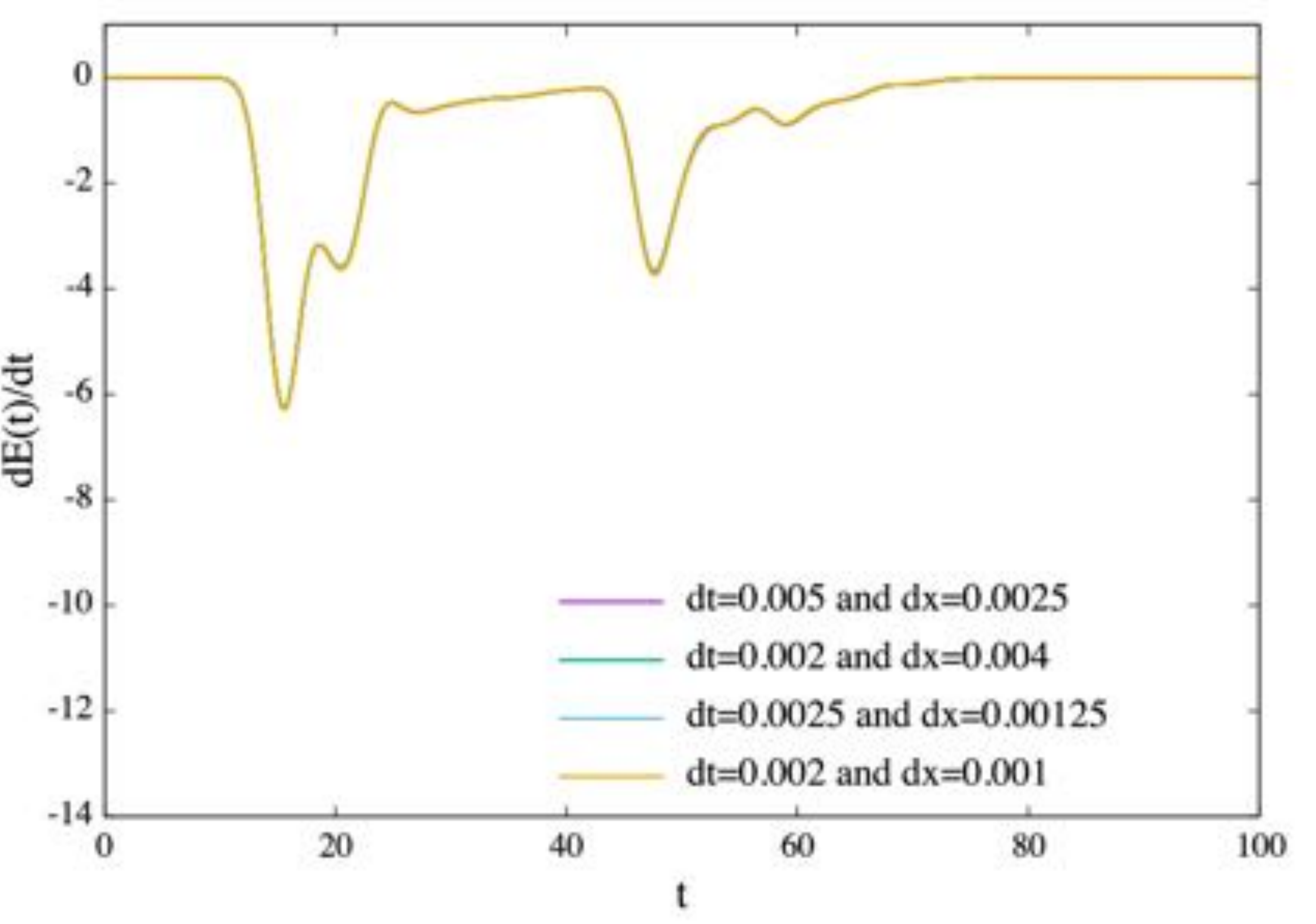} \\
\includegraphics[width=0.98\linewidth]{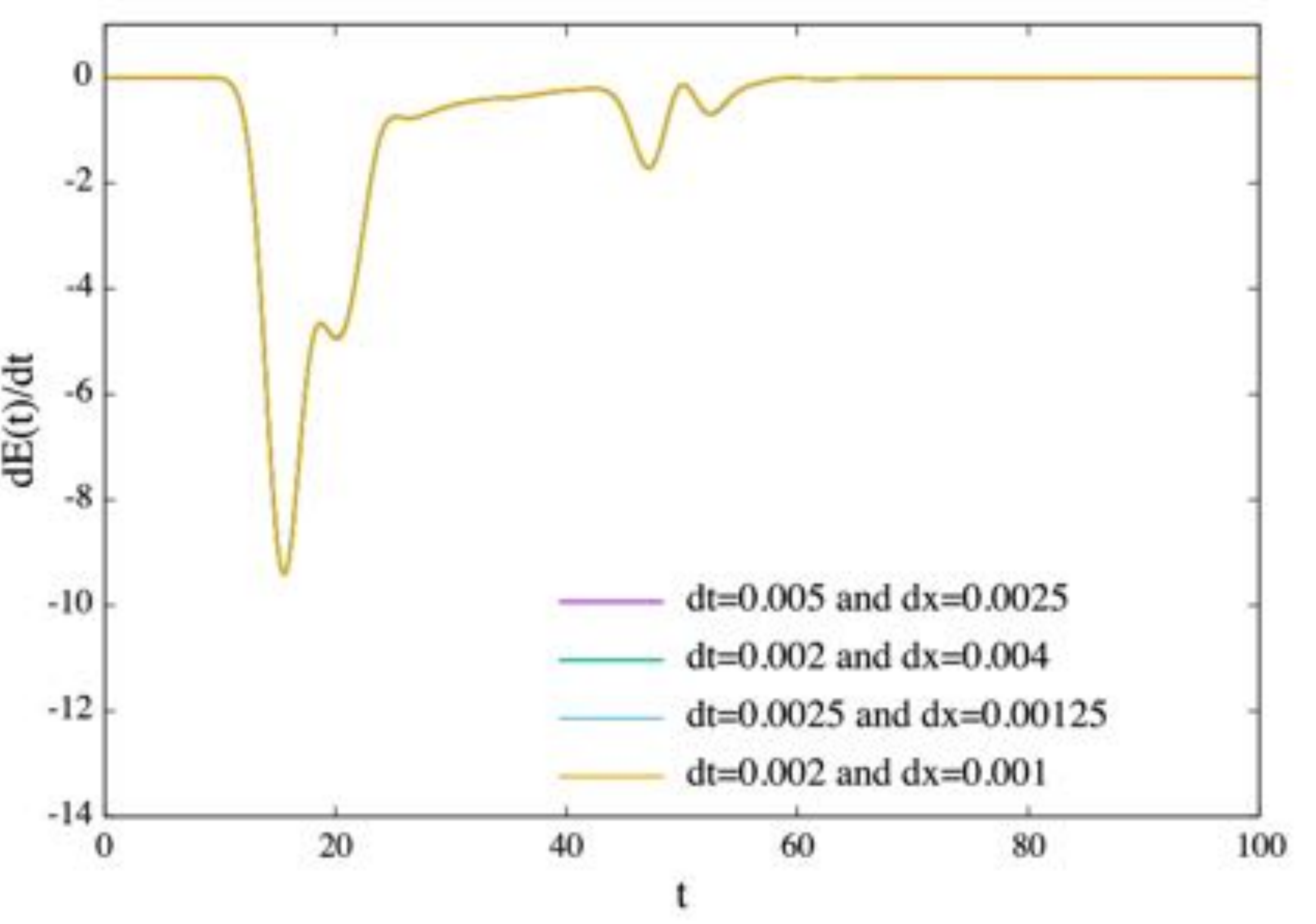} \\
\includegraphics[width=0.98\linewidth]{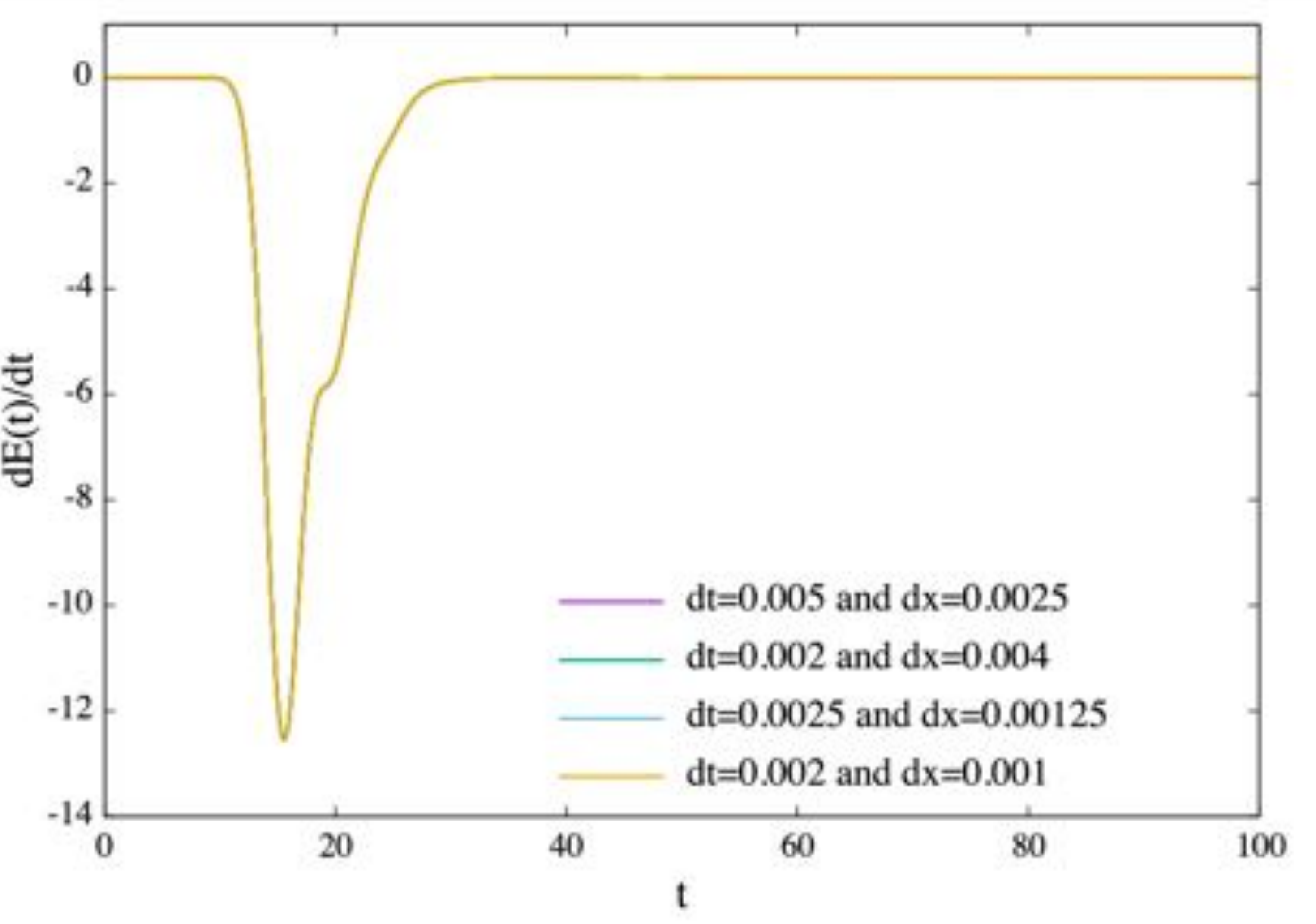}
\end{minipage}
\begin{minipage}{0.31\linewidth}
\centering
$I_{hi}^k$, $i=1, \ldots, 4$ \\
\includegraphics[width=0.98\linewidth]{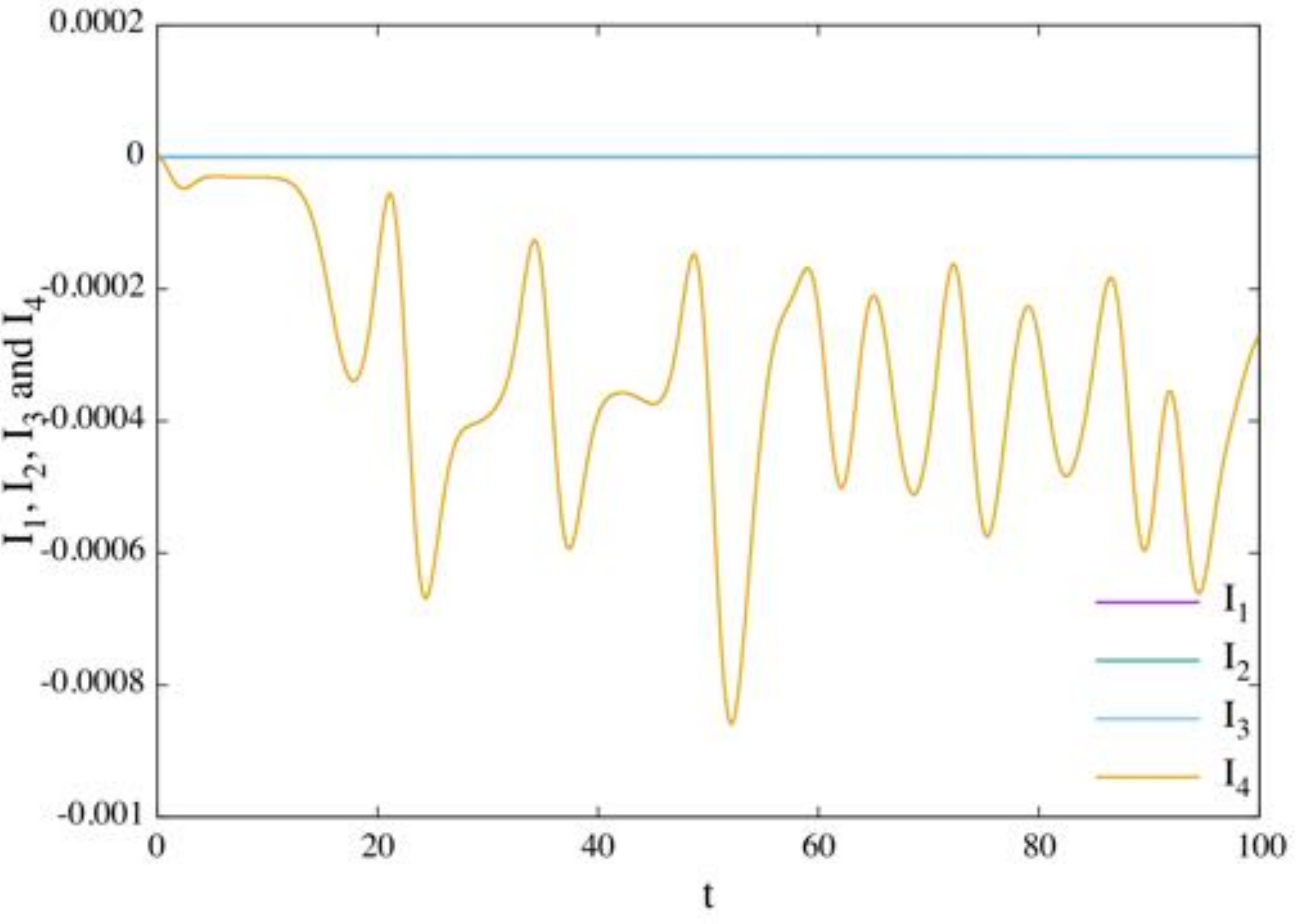} \\
\includegraphics[width=0.98\linewidth]{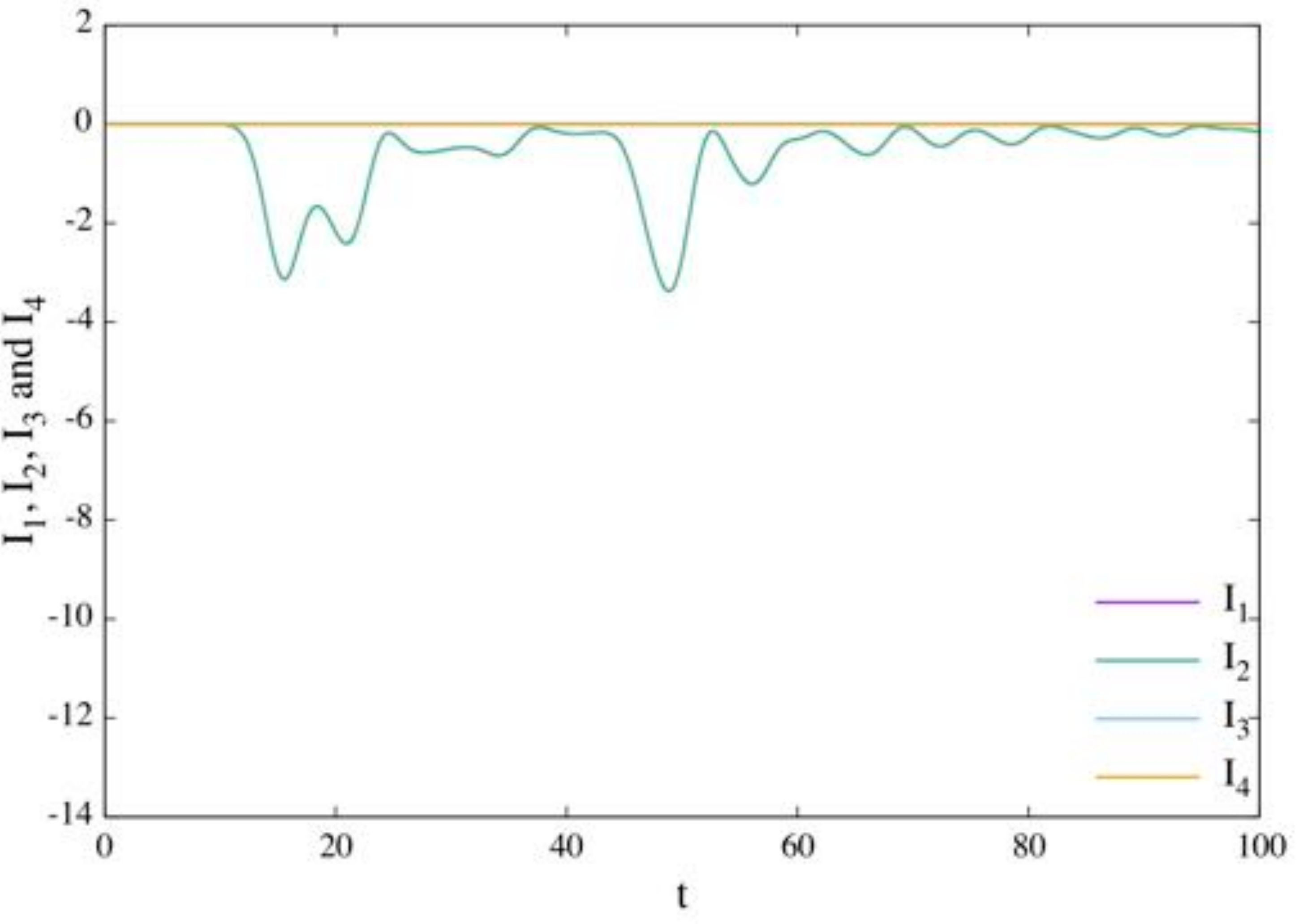} \\
\includegraphics[width=0.98\linewidth]{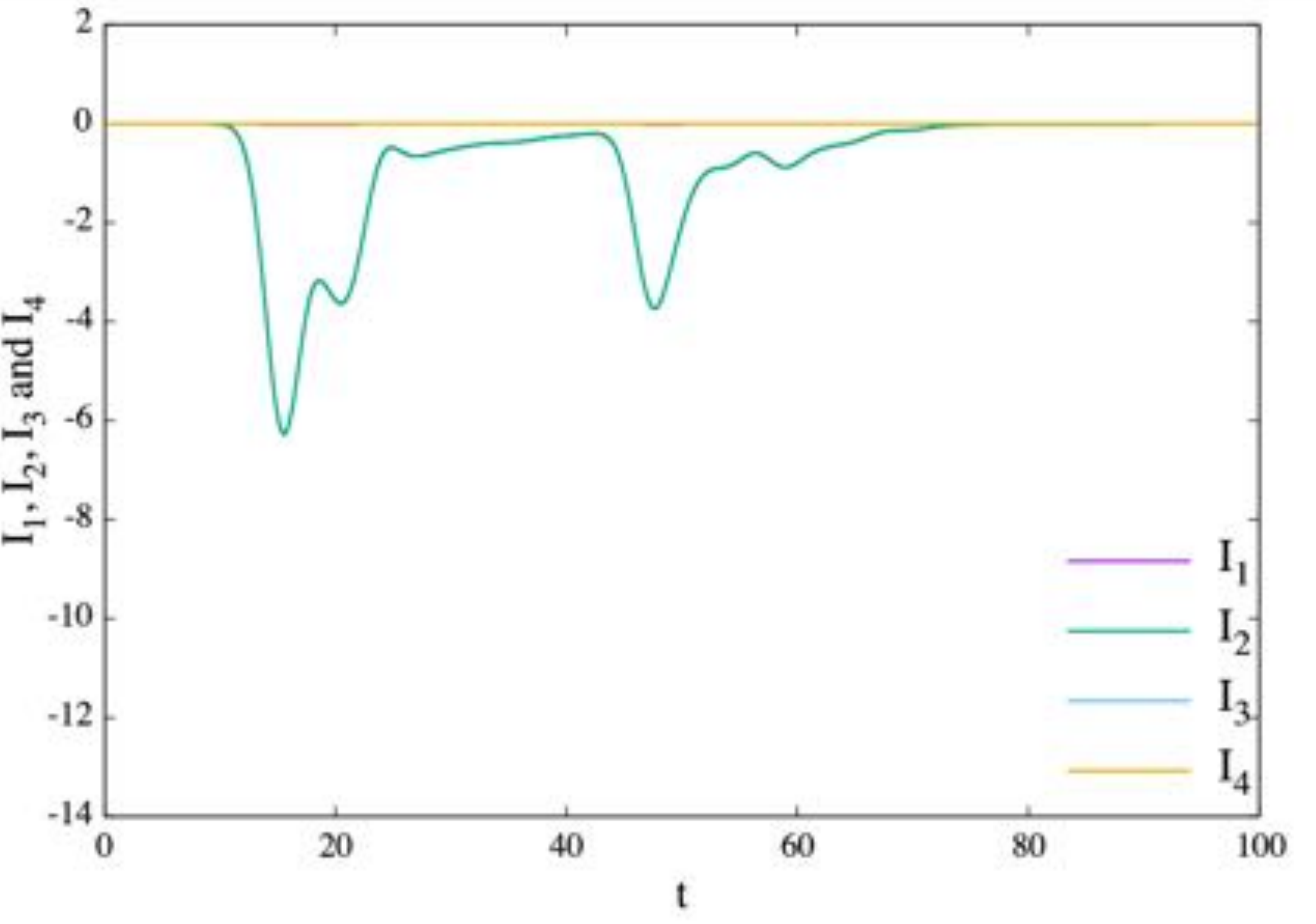} \\
\includegraphics[width=0.98\linewidth]{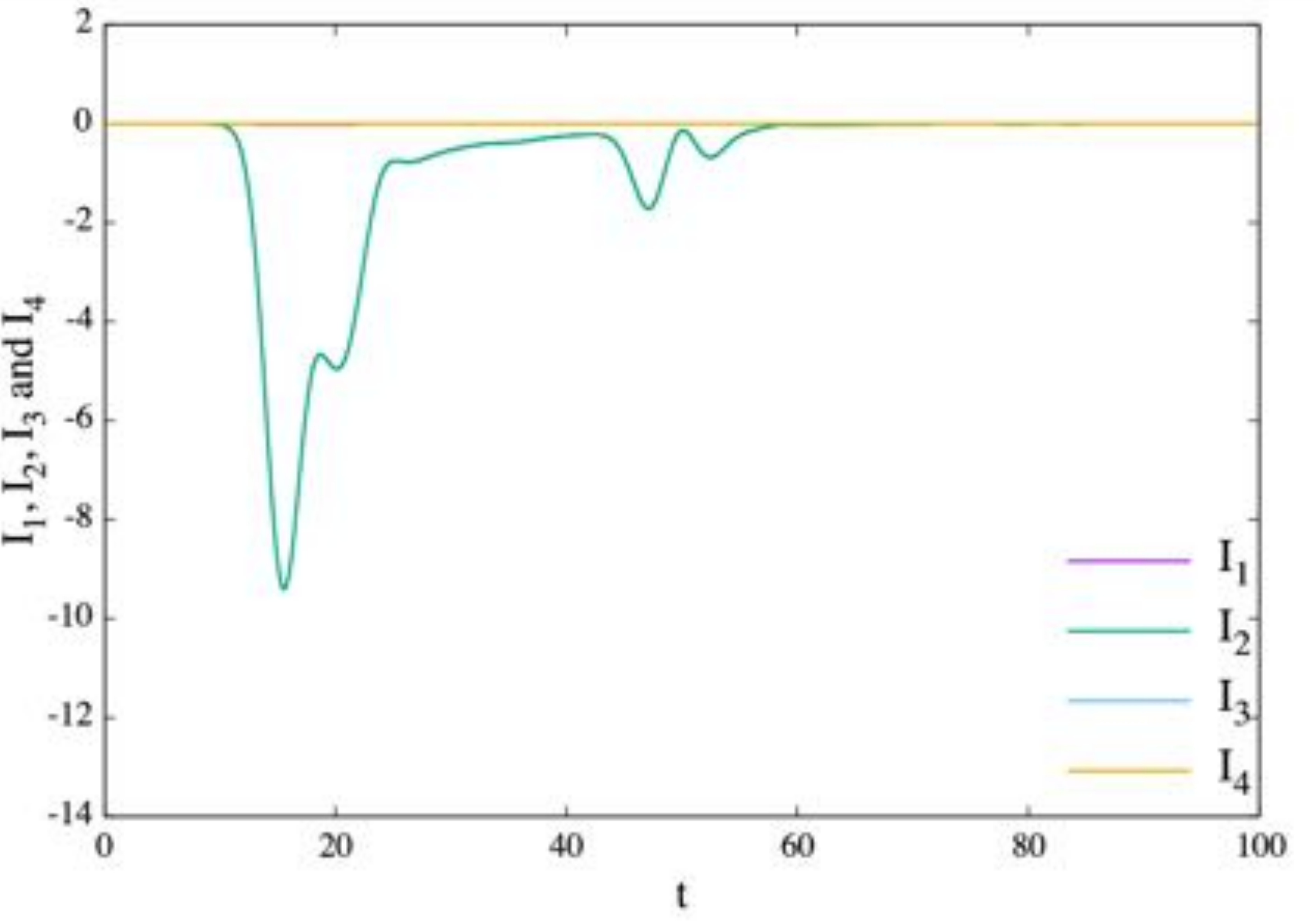} \\
\includegraphics[width=0.98\linewidth]{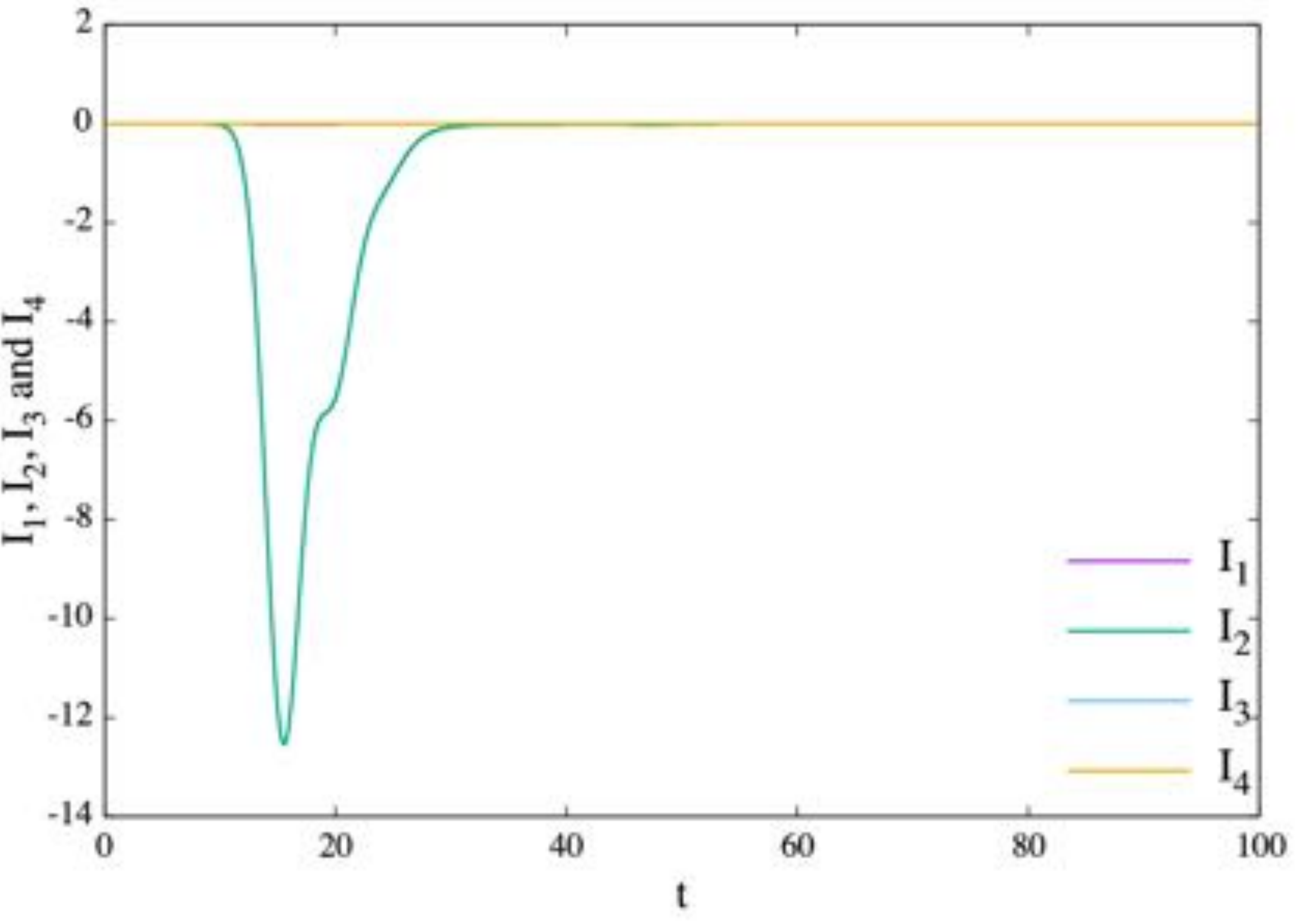}
\end{minipage}
\caption{Graphs of $E_h^k$~(left), 
$\sum_{i=1}^4 I_{hi}^k \approx \fz{d}{dt} E(t)$~(center)
and $I_{hi}^k$, $i=1, \ldots, 4$,~(right) 
versus $t=t^k~(\ge 0, k\in\Z)$ for the five cases~$(i)$-$(v)$.
}
\label{energy_etc}
\end{figure}
In this subsection, we study the stability of solutions to the problem~\eqref{prob}-\eqref{IC} numerically by scheme~\eqref{scheme} in terms of the energy~$E(t)$ defined in~\eqref{te}.
The values of~$E(t^k)$ and $I_i(t^k; \Gamma)$, $i=1, 2, 3$, $I_4(t^k; \Omega)$ are approximately computed by
using solution~$\{ ( u_h^k, \phi_h^k )\}_{k=1}^{N_T}$ with $\{\eta_h^k\}_{k=1}^{N_T}$ of scheme~\eqref{scheme} as
\begin{align*}
E(t^k) \approx E_h^k & \defeq \frac{\rho}{2} h^2 \sum_{x_{i,j}\in \Omega_h} \phi_h^k (x_{i,j}) \bigl|  u_h^k (x_{i,j}) \bigr|^2 + \frac{\rho g}{2} h^2 \sum_{x_{i,j}\in \Omega_h} \bigl| \eta_h^k (x_{i,j}) \bigr|^2, \\
\cK
I_1(t^k; \Gamma) \approx I_{h1}^k & \defeq \frac{\rho}{2} \int_{\Gamma} ( \Pi_h \phi_h^k ) \bigl| \Pi_h  u_h^k \bigr|^2 (\Pi_h  u_h^k) \cdot {  n} \, ds, \\
\cK
I_2(t^k; \Gamma) \approx I_{h2}^k & \defeq - \rho g \int_{\Gamma} ( \Pi_h \phi_h^k ) ( \Pi_h \eta_h^k ) (\Pi_h  u_h^k ) \cdot {  n} \, ds, \\
\cK
I_3(t^k; \Gamma) \approx I_{h3}^k & \defeq 2 \mu \sum_{m=1}^N \Bigl( \int_{\ell_m^{\rm T}} + \int_{\ell_m^{\rm B}} + \int_{\ell_m^{\rm L}} + \int_{\ell_m^{\rm R}} \Bigr) ( \Pi_h \phi_h^k ) (D (\Pi_h  u_h^k ) {  n} ) \cdot (\Pi_h  u_h^k ) \, ds, \\
\cK
I_4(t^k; \Omega) \approx I_{h4}^k & \defeq -2\mu \, h^2 \sum_{i,j=1}^N (\Pi_h \phi_h^k) (x_{i-1/2,j-1/2}) \bigl| D_h (u_h^k)(x_{i-1/2,j-1/2}) \bigr|^2,
\end{align*}
where $\Pi_h f_h \in C(\ol{\Omega}; \R)$ is the bilinear interpolation of $f_h: \ol{\Omega}_h \to \R$ for $f_h = \phi_h, \eta_h$,
$\Pi_h u_h = (\Pi_h u_{h1}, \Pi_h u_{h2})^T \in C(\ol{\Omega}; \R^2)$,
the boundary is represented as $\ol{\Gamma} = \bigcup_{m=1}^N (\ol{\ell}_m^{\rm T} \cup \ol{\ell}_m^{\rm B} \cup \ol{\ell}_m^{\rm L} \cup \ol{\ell}_m^{\rm R})$ for line segments $\ell_m^{\rm T}$, $\ell_m^{\rm B}$, $\ell_m^{\rm L}$, $\ell_m^{\rm R}$ defined by
$\ell_m^{\rm T} \defeq \ol{x_{m-1,N} x_{m,N}}$,
$\ell_m^{\rm B} \defeq \ol{x_{m-1,0} x_{m,0}}$,
$\ell_m^{\rm L} \defeq \ol{x_{0,m-1} x_{0,m}}$,
$\ell_m^{\rm R} \defeq \ol{x_{N,m-1} x_{N,m}}$,
and the domain is represented as $\ol{\Omega} = \bigcup_{i,j=1}^N \ol{\omega}_{i-1/2, j-1/2}$ for $\omega_{i-1/2, j-1/2} \defeq ((i-1)h, ih) \times ((j-1)h, jh) \subset \Omega$ with the area~$h^2$.
\par
\cK
Numerical simulations for the problem~\eqref{prob}-\eqref{IC} with $L=1$, $a=0.1$, $u^0=0$, $c_1=0.001$, $p=(0.5,0.5)^T$ are carried out by scheme~\eqref{scheme} with $\Delta t = 2h$ for $h=L/N=1/N$, $N=400$, $500$, $800$ and $1,000$.
\cK
The results are presented in Figure~\ref{energy_etc}, where $(i)$-$(v)$ in the figure represent the cases $(i)$-$(v)$ at the beginning of this section.
\cK
The graphs of $E_h^k$ and $\sum_{i=1}^4 I_{hi}^k$ versus $t=t^k~(k\in\N)$ are presented in the left and center figures, respectively.
\cK
There are four lines in each figure, but the lines are almost overlapped in the cases of $(ii)$-$(v)$.
In the case of $(i)$ the graphs are qualitatively similar.
\cK
The right figures show the graphs of $I_{hi}^k$, $i=1, \ldots, 4$ versus $t=t^k~(k\in\N)$.
The maximum and minimum values of $I_{hi}^k$, $i=1, \ldots, 4$, on the time interval $(0, T)$ are presented in Table~\ref{table:1}.
\par
From the numerical results presented in Figure~\ref{energy_etc}, it can be found that the total energy is {\cK mainly} decreasing with respect to time.
In the case of~$(i)$, i.e., $\Gamma = \Gamma_D$, we can see that at the early period the graphs are increasing in the left figure, while the values are small.
From the center figures it can be clearly seen that the sum~$\sum_{i=1}^4 I_{hi}^k$ corresponding to the derivative of the total energy is always non-positive, which confirms the stability of solutions to the model numerically.
From right figures and Table~\ref{table:1}, it can be observed that the value of $I_{h2}$ is dominating negatively over $I_{h1}$ and $I_{h3}$ so that the sum $\sum_{i=1}^4 I_{hi}$ becomes non-positive always.
\begin{table}[!htbp]
	\caption{Maximum and minimum values of~$I_{hi}^k$, $i=1,\ldots, 4$, with respect to the number of transmission boundaries}\label{table:1}
\cK
\begin{tabular}{cccrrrr}
	\toprule
	\multirow{1}{*}{$\Gamma_T$} & \multirow{1}{*}{$\Gamma_D$} & & \multicolumn{1}{c}{$I_{h1}$} & \multicolumn{1}{c}{$I_{h2}$} & \multicolumn{1}{c}{$I_{h3}$} & \multicolumn{1}{c}{$I_{h4}$} \\
	\midrule
	\multirow{2}{*}{0}         &  \multirow{2}{*}{4} & Max &  0.00 & 0.00 & 0.00 & 0.00 \\
           &             & Min & 0.00 & 0.00 &  0.00 & $-8.59 \times 10^{-4}$ \\
	\midrule
\multirow{2}{*}{1}         &  \multirow{2}{*}{3}   & Max & $1.10 \times 10^{-4}$ & 0.00 & $1.44 \times 10^{-6}$ & 0.00 \\
           &             & Min & $-2.59 \times 10^{-3}$ & $-3.37$ & $-1.25\times 10^{-6}$ & $-3.76\times 10^{-4}$ \\
 \midrule
\multirow{2}{*}{2}         &  \multirow{2}{*}{2}   & Max       & $1.86 \times 10^{-4}$ & 0.00 & $1.72 \times 10^{-6}$ & 0.00 \\
           &          & Min    & $-3.38 \times 10^{-3}$ & -6.27 & $-2.50 \times 10^{-6}$ & $-2.29 \times 10^{-4}$ \\
 \midrule
\multirow{2}{*}{3}         &  \multirow{2}{*}{1}    & Max       & $1.43 \times 10^{-4}$ & 0.00 & $2.58 \times 10^{-6}$ & 0.00 \\
           &        & Min      & $-5.06 \times 10^{-3}$ & $-9.40$ & $-3.75 \times 10^{-6}$ & $-1.73 \times 10^{-4}$ \\
 \midrule
 \multirow{2}{*}{4}         &  \multirow{2}{*}{0}      & Max     & $2.87 \times 10^{-4}$ & 0.00 & $3.47 \times 10^{-6}$ & 0.00 \\
           &          & Min    & $-6.75 \times 10^{-3}$ & $-12.54$ & $-5.01 \times 10^{-6}$ & $-1.13 \times 10^{-4}$ \\
    \bottomrule
  \end{tabular}
\end{table}
\cK
\subsection{Choice of $c_0$}\label{subsec4.4}
In this subsection, 
\cK
we study the value of~$c_0$ by solving the problem with $L=1$, $a=0.1$, ${  p}=(0.5,0.5)^T$ and $c_0=0.1, 0.2, \ldots, 1.2$ and $1.5$ by scheme~\eqref{scheme} with  $N=400$, $\Delta t=0.005~(N_T=20,000)$, and by computing
\cK
\begin{align*}
S_h^k(c_0) & \defeq \biggl\{ h^2\sum_{x_{i,j}\in\ol{\Omega}_h} \eta_h^k (x_{i,j})^2 \biggr\}^{\fz{1}{2}}
\approx S (t^k;c_0) \defeq \biggl\{ \int_\Omega|\eta^k(x)|^2 dx \biggr\}^{\fz{1}{2}} = \|\eta^k\|_{L^2(\Omega)}, \\
\mathcal{S}_h (c_0) & \defeq \biggl\{ \Delta t\sum_{k=0}^{N_T} S_h^k(c_0)^2 \biggr\}^{\fz{1}{2}}
\approx \mathcal{S} (c_0) \defeq \biggl\{ \int_0^T |S(t^k;c_0)|^2 dt \biggr\}^{\fz{1}{2}}
= \|\eta\|_{L^2(0,T; L^2(\Omega))}.
\end{align*}
\par
We set six different cases for the initial values of $\eta^0$ and $ u^0$ as follows.
\begin{align*}
& \text{ Case I:} && \eta^0=c_1 \exp \bigl(-100 \, \bigl|x-(0.5,0.5)\bigr|^2 \bigr), && u^0=0, \\
& \text{ Case II:} && \eta^0=c_1 \exp \bigl(-200 \, \bigl|x-(0.5,0.5)\bigr|^2\bigr), && u^0=0, \\
& \text{ Case III:} && \eta^0=c_1 \exp \bigl(-100 \, \bigl|x-(0,0.5)\bigr|^2\bigr), && u^0=0, \\
& \text{ Case IV:} && \eta^0=c_1 \exp \bigl(-100 \, \bigl|x\bigr|^2\bigr),          && u^0=0, \\
& \text{ Case V:} && \eta^0=c_1 \exp \bigl(-100 \, \bigl|x-(0.5,0.5)\bigr|^2 \bigr), && u^0=10^{-4} \, (1, 1)^T, \\
& \text{ Case VI:} && \eta^0=c_1 \exp \bigl(-100 \, \bigl|x-(0.5,0.5)\bigr|^2 \bigr), && u^0=10^{-4} \, (1,-1)^T,
\end{align*}
where $c_1=10^{-3}$.
\par
Since the artificial reflection should be removed after the time the wave touches the transmission boundary, we find a value of~$c_0$ which provides the minimum of~$\mathcal{S}_h(c_0)$.
\cK
The results are presented in Table~\ref{table:2}, from where it can be concluded that for the case of zero initial velocity the suitable value of~$c_0$ lies in~$[0.7, 1.0]$ and for the case of nonzero initial velocity we cannot say anything yet.
\begin{table}[h!]
  \caption{$c_0$ and  $\mathcal{S}_h(c_0)$}\label{table:2}
  \begin{tabular}{lrrrrrr}
    \toprule
    \multicolumn{1}{c}{$c_0$}  &  \multicolumn{6}{c}{$\mathcal{S}_h(c_0)$}  \\
	\cline{2-7}
          &\multicolumn{1}{c}{ Case I} & \multicolumn{1}{c}{Case II} & \multicolumn{1}{c}{Case III} &\multicolumn{1}{c}{ Case IV} & \multicolumn{1}{c}{Case V} & \multicolumn{1}{c}{Case VI} \\
  \midrule 
0.1 & 12.17\phantom{00} & 8.52\phantom{00} & 8.14 & 5.47 &  44.48  &  44.48 \\
0.2 & 9.88\phantom{00}  & 6.96\phantom{00} & 6.35 & 4.04 &  34.36  &  34.37  \\
0.3 & 8.84\phantom{00}  & 6.24\phantom{00} & 5.52 & 3.35 &  28.74  &  28.75 \\
0.4 & 8.27\phantom{00}  & 5.84\phantom{00} & 5.08 & 2.98 &  25.23  &  25.24\\
0.5 & 7.93\phantom{00}  & 5.61\phantom{00} & 4.84 & 2.79 &  22.82  &  22.83 \\
0.6 & 7.71\phantom{00}  & 5.46\phantom{00} & 4.71 & 2.69 &  21.05  &  21.06 \\
0.7 & 7.58\phantom{00}  & 5.37\phantom{00} & 4.65 & \underline{2.66} &  19.69  &  19.69 \\
0.8 & 7.51\phantom{00}  & 5.32\phantom{00} & \underline{4.63} & 2.67 &  18.60  &  18.61 \\
0.9 & \underline{7.4792}  & \underline{5.2951} & 4.64 & 2.70 &  17.71  &  17.72 \\
1.0 & 7.4795  & 5.2959 & 4.68 & 2.75 &  16.98 & 16.98 \\
1.1 & 7.50\phantom{00}  & 5.31\phantom{00} & 4.73 & 2.82 &  16.36  &  16.36 \\
1.2 & 7.55\phantom{00}  & 5.34\phantom{00} & 4.79 & 2.89 &  15.83  &  15.84 \\
1.5 & 7.75\phantom{00}  & 5.49\phantom{00} & 5.02 & 3.12 &  14.66  &   14.66 \\
 \bottomrule
  \end{tabular}
\end{table}
\section{Numerical results by an LG scheme}\label{sec5}
In this section, we present an LG scheme for the same problem described in Subsection~\ref{subsec4.2}.
\par
Let $\mathcal{T}_h= \{K\}$ be a triangulation of $\Omega$, 
{\cK and $M_h$ the so-called P1~(piecewise linear) finite element space.}
We set $\Psi_h \defeq M_h$ for the water level~$\eta$, and
\begin{align*}
V_h (\psi_h) & \defeq 
\left\{
v_h \in M_h^2;\ 
\begin{aligned}
v_h(P) &= c(P)\frac{\psi_h(P)-\zeta(P)}{\psi_h(P)} n(P), &&\forall P: \mbox{node on $\Gamma_T$}, \\
v_h(Q) &= 0, &&\forall Q: \text{node on} \ \Gamma_D
\end{aligned}
\right\}
\end{align*}
for the velocity~${ u}$.
The LG scheme is to find $\{(\phi^k_h, u^k_h)\}_{k=1}^{N_T}\subset \Psi_h \times V_h$ such that, for $k=1, \ldots, N_T$,
\begin{align}
\left\{
\begin{aligned}
\int_\Omega\frac{\phi_h^k-\widetilde{\phi}^{k-1}_h\circ X^{k-1}_{1h}\gamma^{k-1}_h}{\Delta t} \psi_h \, dx & = 0, \qquad\qquad\qquad \forall \psi_h \in \Psi_h, \\
\rho \int_\Omega \phi^k_h \frac{u^k_h-\widetilde{u}^{k-1}_h\circ X^{k-1}_{1h}}{\Delta t}\cdot v_h \, dx & + 2 \mu \int_\Omega \phi_h^k D( u_h^k) : D(v_h) \, dx \\
+ \rho g \int_\Omega\phi_h^k\nabla \eta_h^k \cdot v_h \, dx & = 0, 
\qquad\qquad\qquad \forall v_h \in V_h, \\
\phi^k_h & = \eta^k_h+\Pi_h^{\rm FEM} \zeta, 
\end{aligned}
\right.
\label{scheme:LG}
\end{align}
where
$X_{1h}^k(x) \defeq x- u^k_h(x)\Delta t$,
$\gamma^k_h: \Omega\to \mathbb{R}$ is defined by
\[
\gamma^k_h(x) \defeq {\rm det} \Bigl( \frac{\partial X^k_{1h}(x)}{\partial x } \Bigr),
\]
\cK
the symbol ``\,$\circ$\,'' represents the composition of functions, i.e., 
\cK
$[ v_h \circ X^k_{1h} ] (x) \defeq v_h(X^k_{1h}(x))$,
$\Pi_h^{\rm FEM}: C(\ol{\Omega}) \to M_h$ is the Lagrange interpolation operator,
and
\cK
\[
\widetilde\psi_h(x)=
	\left\{
	\begin{aligned}
	& \psi_h(x), & x &\in \overline{\Omega}, \\
	& \psi_h(P_x), & x &\in \R^2 \setminus \overline{\Omega},
	\end{aligned}
	\right.
\]
where $P_x \in \Gamma$ is a ``nearest'' nodal point from~$x$.
\cK
In each step, firstly, $\phi_h^k \in \Psi_h$ is obtained from the first equation of scheme~\eqref{scheme:LG}.
Secondly, $u^k_h\in V_h$ is obtained by using~$\phi_h^k$ from the second equation.
In the first equation of~\eqref{scheme:LG}, the idea of mass conservative Lagrange--Galerkin scheme~\cite{RuiTab-2010} is employed.
\cK
\par
A numerical simulation is carried out by LG scheme~\eqref{scheme:LG} with the same setting described in Subsection~\ref{subsec4.2} except $\Delta t$, where it is set as $\Delta t=0.0625$.
The results are shown in Figure~\ref{Simulatedresult_LG}.
The figures are similar to the ones in Figure~\ref{Simulatedresult} obtained by finite difference scheme~\eqref{scheme} and support the results in Figure~\ref{Simulatedresult}.
\cK
\begin{figure}[!htbp]
\begin{minipage}[b]{0.08\linewidth}
\centering
{$(i)$}
\vspace{1cm}
\end{minipage}
\begin{minipage}[b]{0.155\linewidth}
\centering
\includegraphics[width=0.99\linewidth]{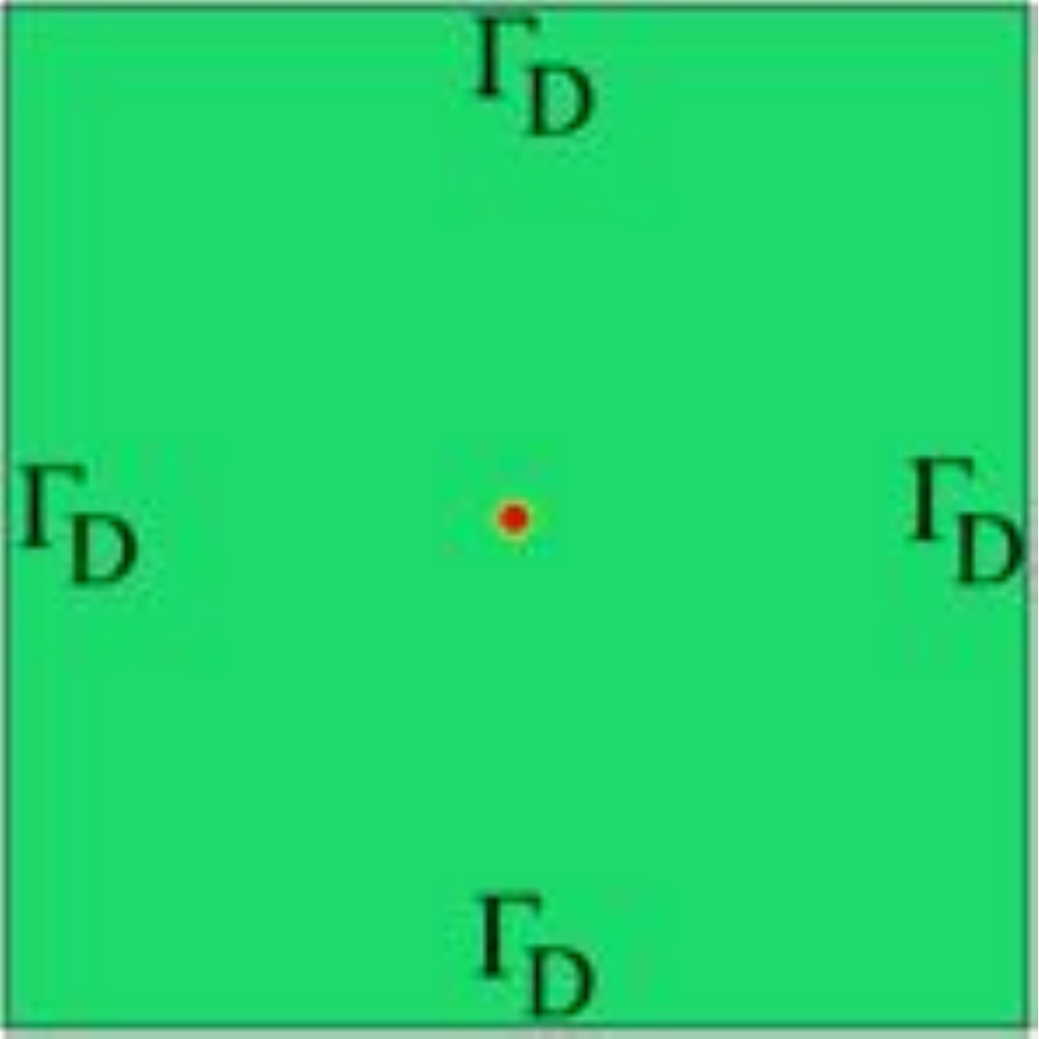}
{$t=0$}
\end{minipage}
\begin{minipage}[b]{0.155\linewidth}
\centering
\includegraphics[width=0.99\linewidth]{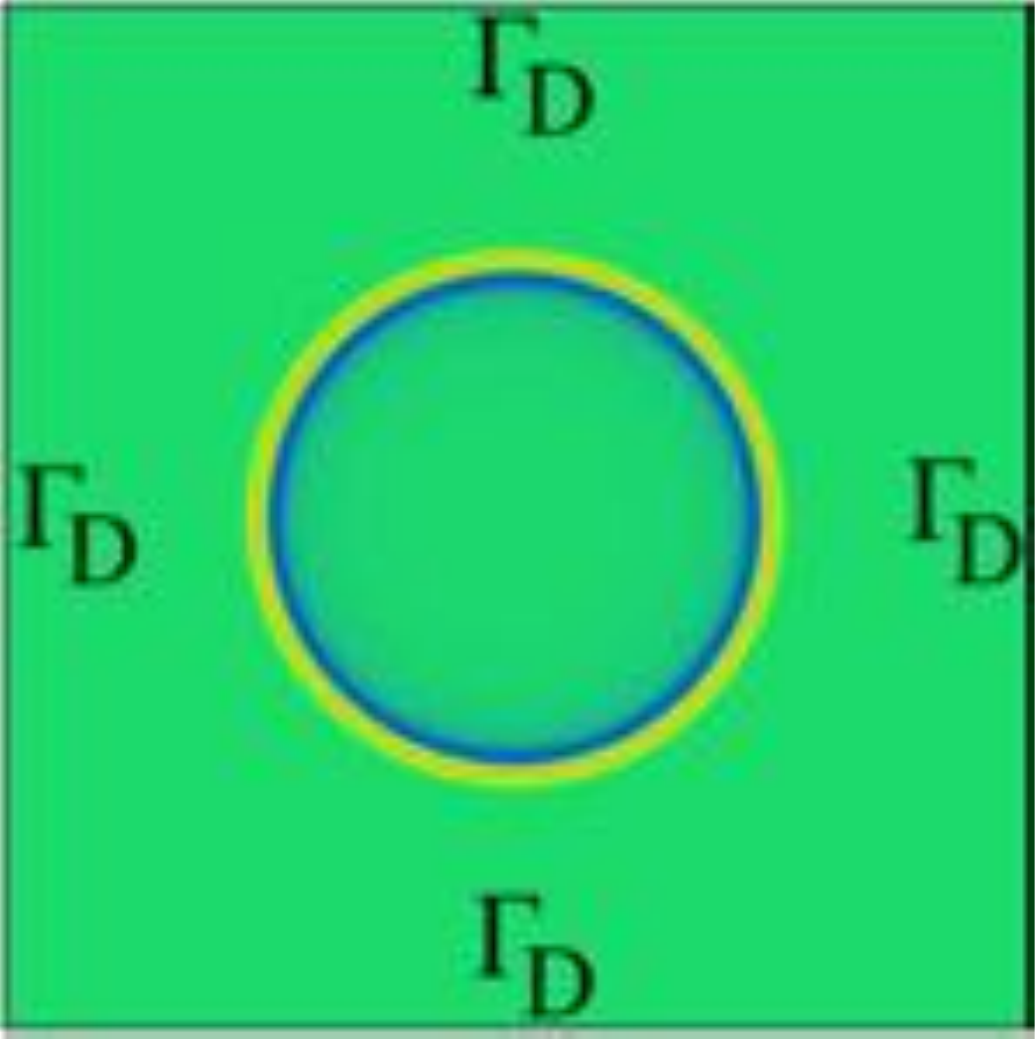}
{$t=25$}
\end{minipage}
\begin{minipage}[b]{0.155\linewidth}
\centering
\includegraphics[width=0.99\linewidth]{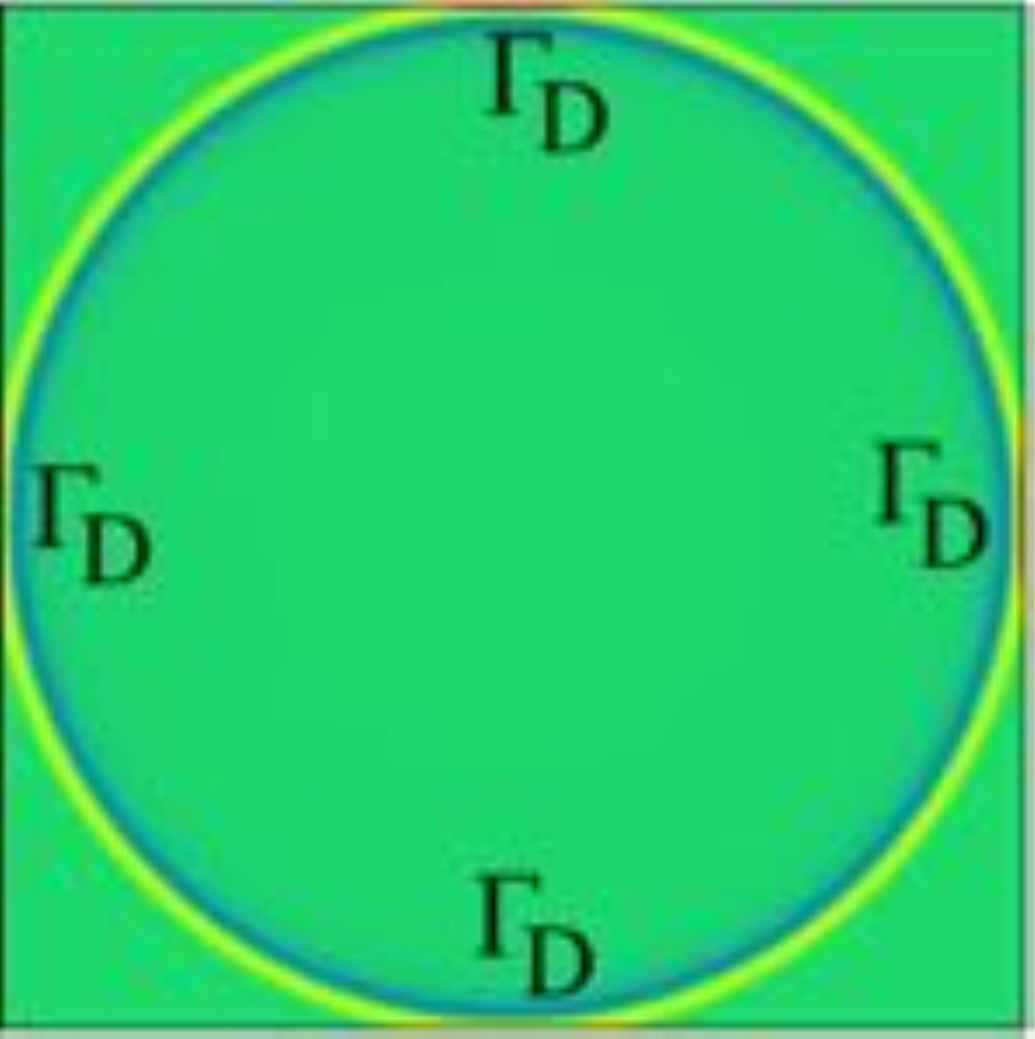}
{$t=50$}
\end{minipage}
\begin{minipage}[b]{0.155\linewidth}
\vspace{0.5cm}
\includegraphics[width=0.99\linewidth]{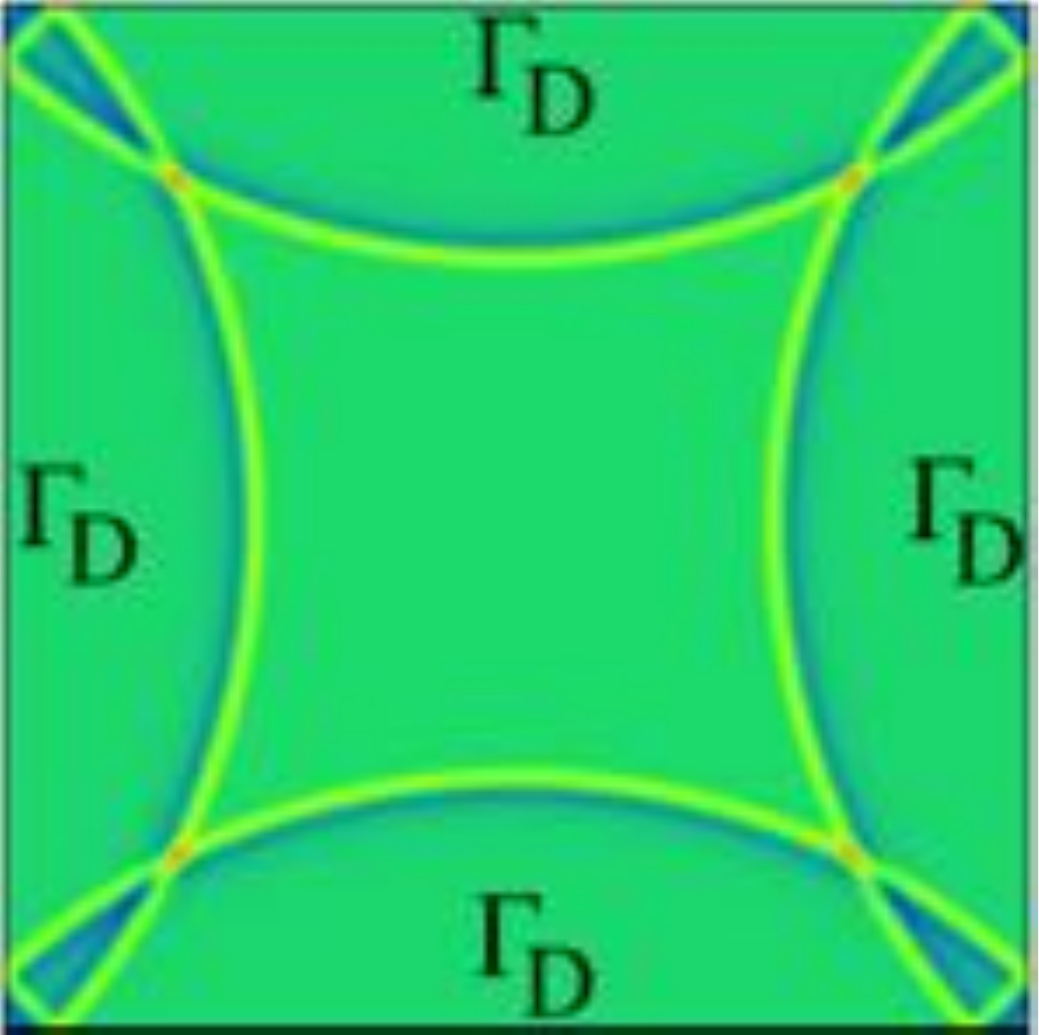}
{$t=75$}
\end{minipage}
\begin{minipage}[b]{0.155\linewidth}
\vspace{0.5cm}
\centering
\includegraphics[width=0.99\linewidth]{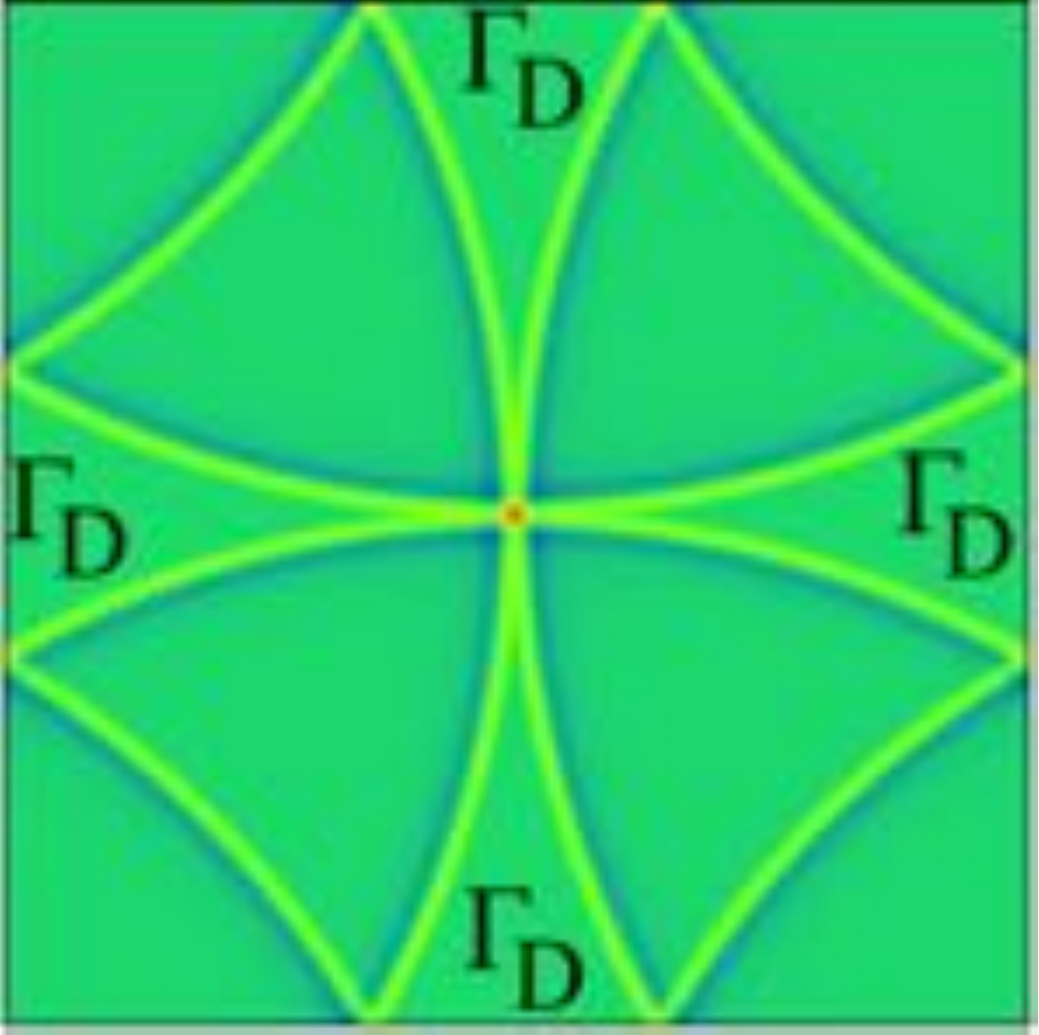}
{$t=100$}
\end{minipage}
\begin{minipage}[b]{0.062\linewidth}
\vspace{0.5cm}
\centering
\includegraphics[width=0.99\linewidth]{Colour_bar}
\vspace{0.00cm}
\end{minipage}

\begin{minipage}[b]{0.08\linewidth}
\centering
{$(ii)$}
\vspace{1cm}
\end{minipage}
\begin{minipage}[b]{0.155\linewidth}
\centering
\includegraphics[width=0.99\linewidth]{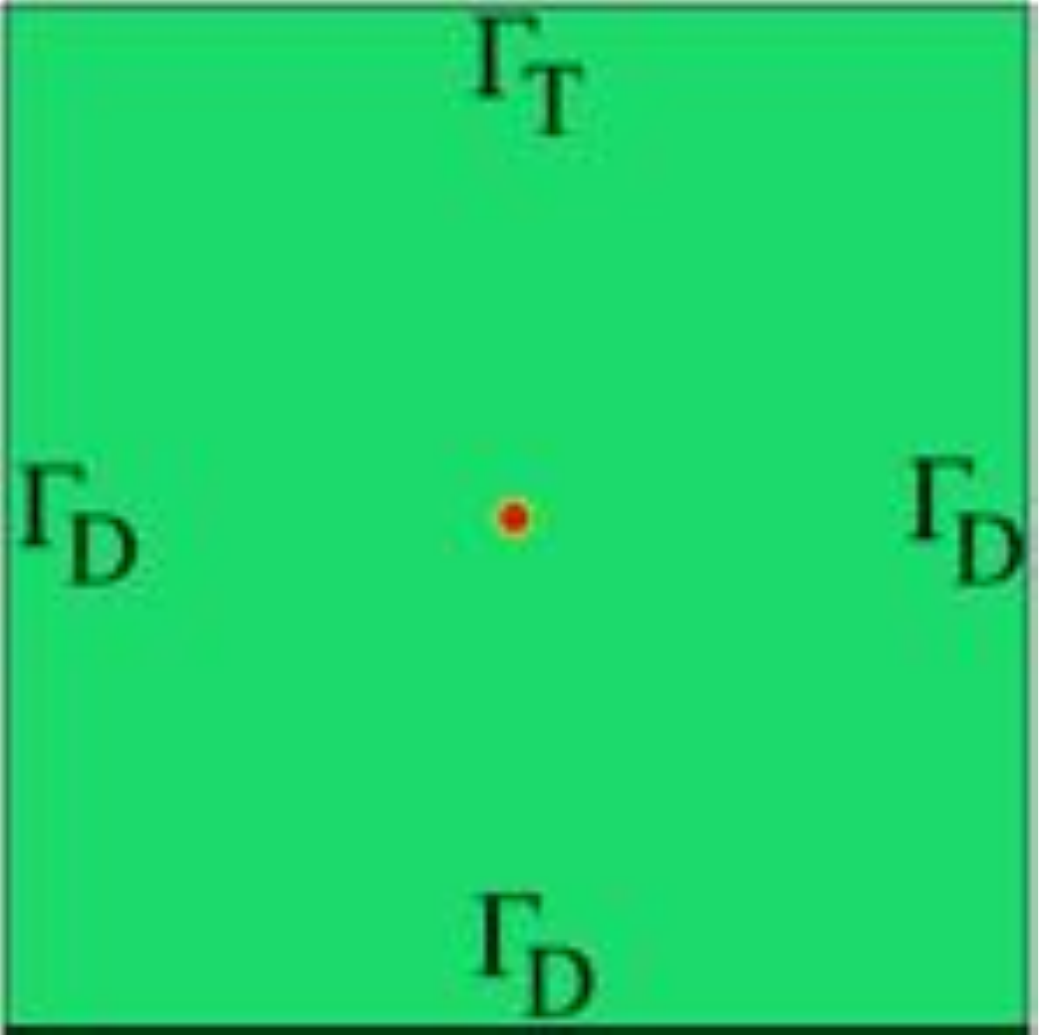}
{$t=0$}
\end{minipage}
\begin{minipage}[b]{0.155\linewidth}
\centering
\includegraphics[width=0.99\linewidth]{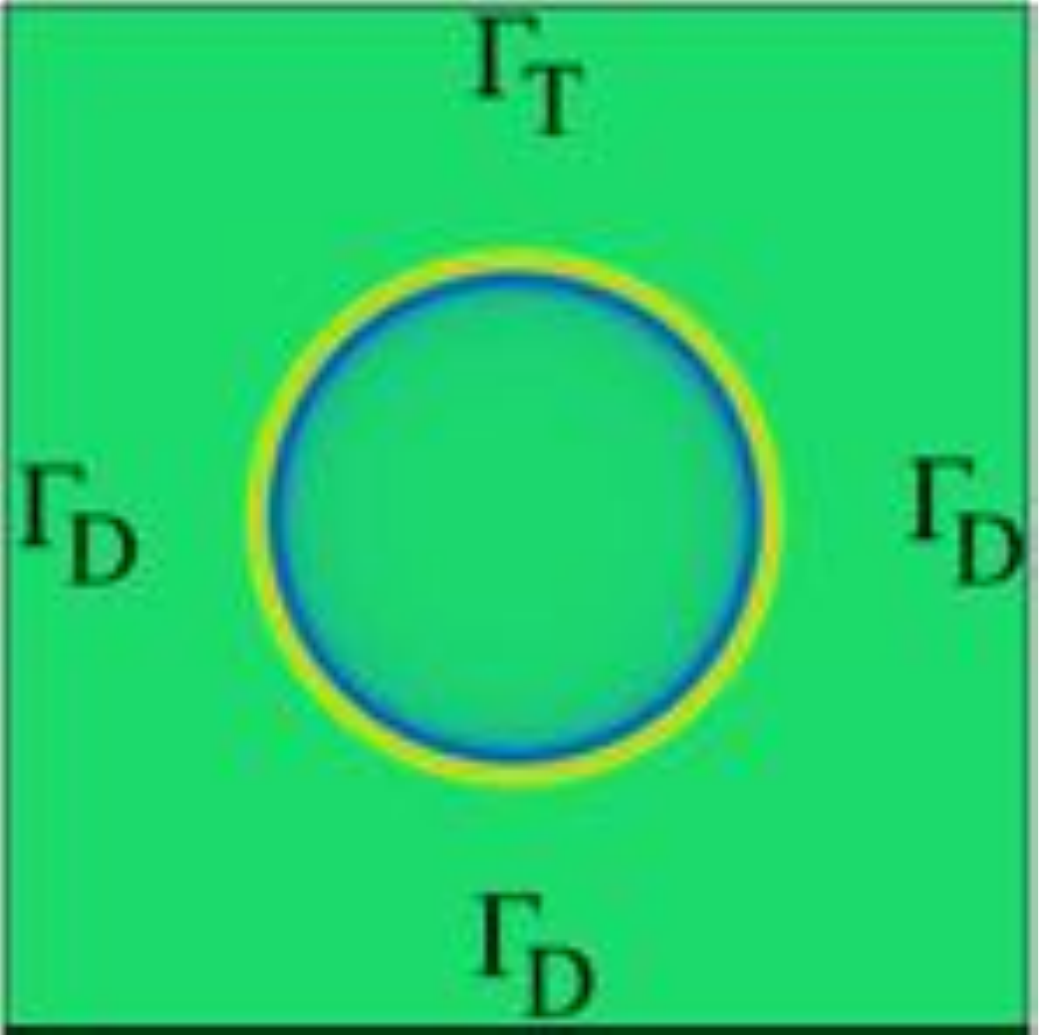}
{$t=25$}
\end{minipage}
\begin{minipage}[b]{0.155\linewidth}
\centering
\includegraphics[width=0.99\linewidth]{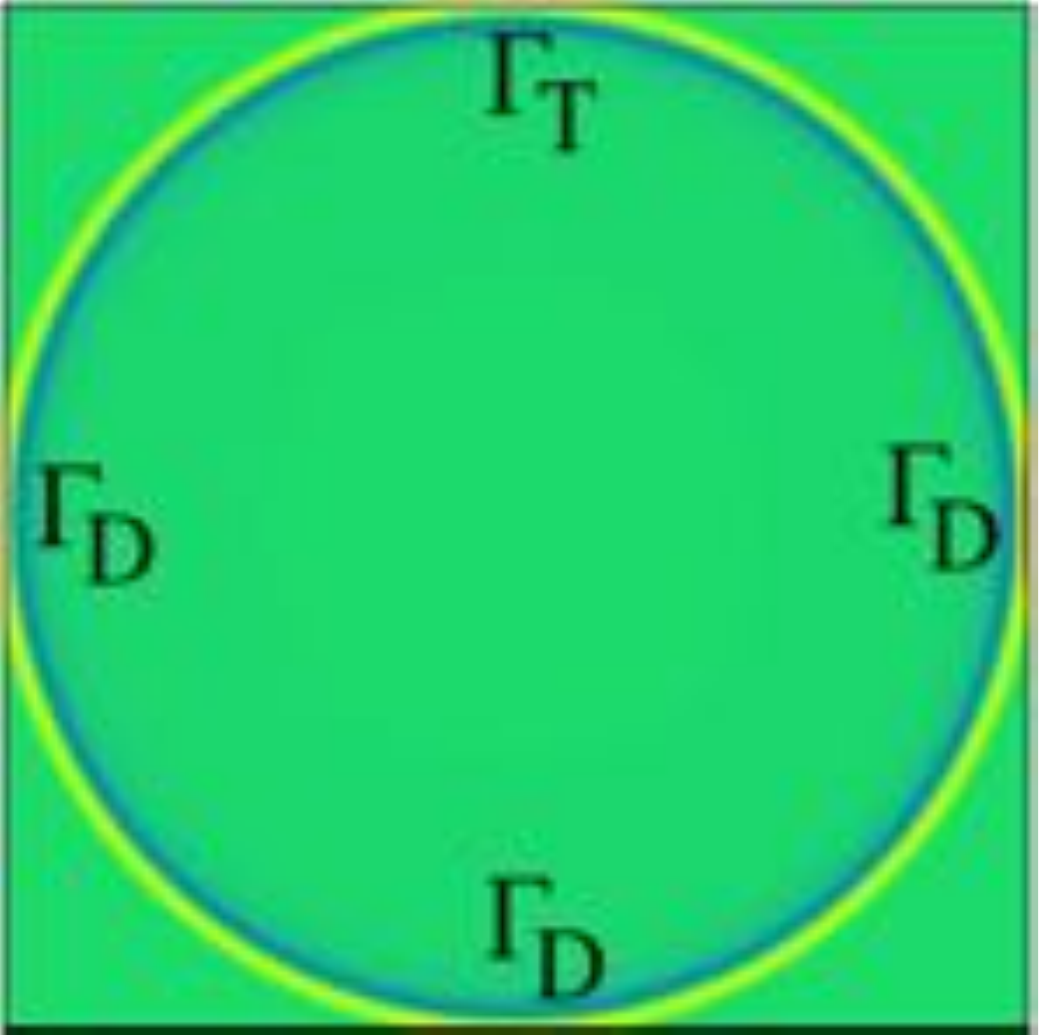}
{$t=50$}
\end{minipage}
\begin{minipage}[b]{0.155\linewidth}
\vspace{0.5cm}
\centering
\includegraphics[width=0.99\linewidth]{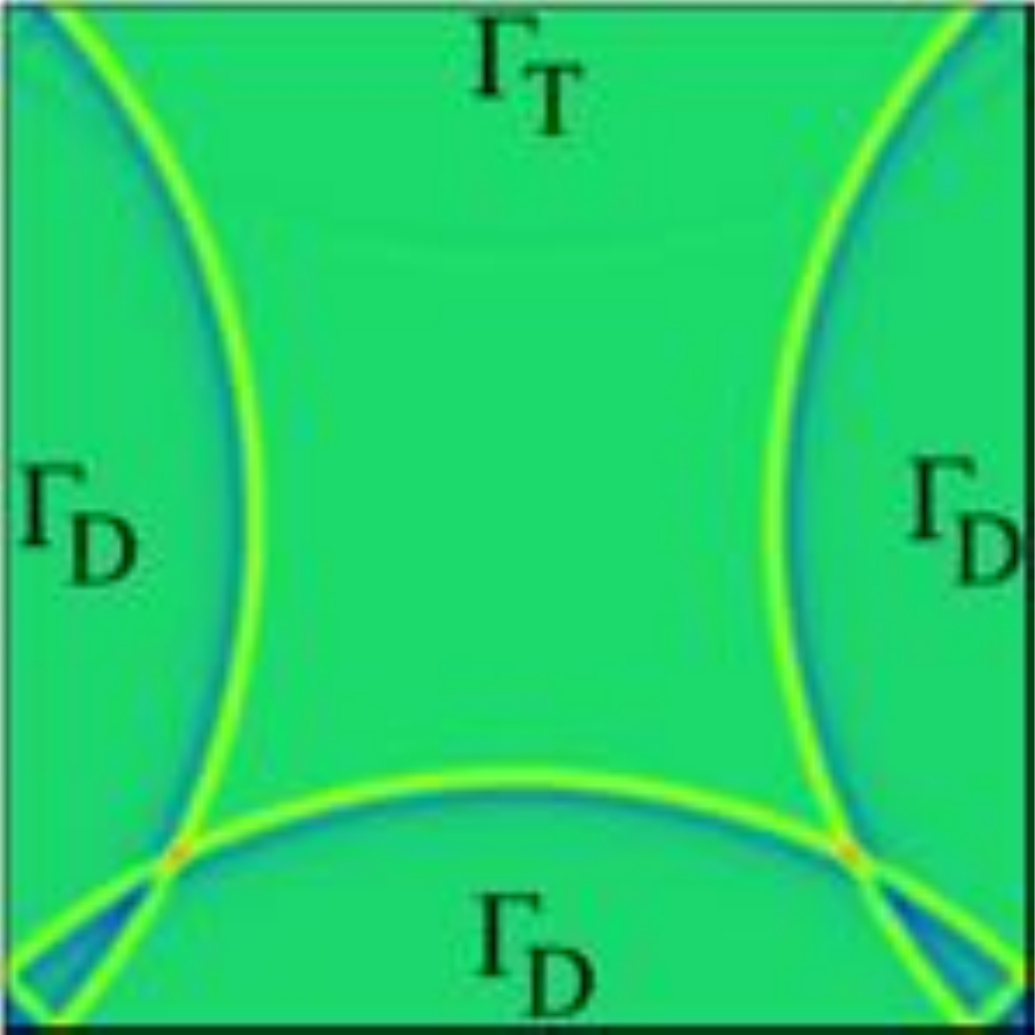}
{$t=75$}
\end{minipage}
\begin{minipage}[b]{0.155\linewidth}
\vspace{0.5cm}
\centering
\includegraphics[width=0.99\linewidth]{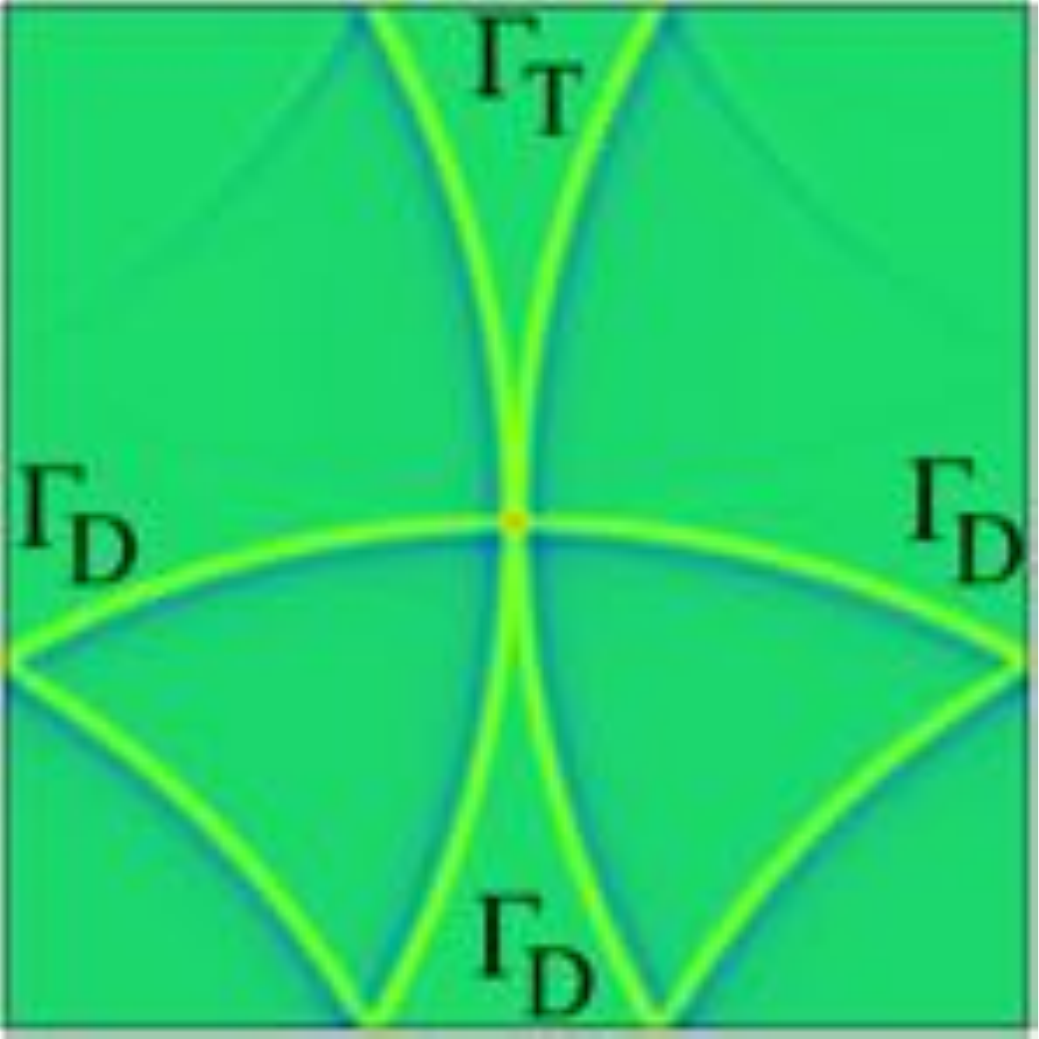}
{$t=100$}
\end{minipage}
\begin{minipage}[b]{0.062\linewidth}
\vspace{0.5cm}
\centering
\includegraphics[width=0.99\linewidth]{Colour_bar}
\vspace{0.00cm}
\end{minipage} 
  
\begin{minipage}[b]{0.08\linewidth}
\centering
{$(iii)$}
\vspace{1cm}
\end{minipage}
\begin{minipage}[b]{0.155\linewidth}
\centering
\includegraphics[width=0.99\linewidth]{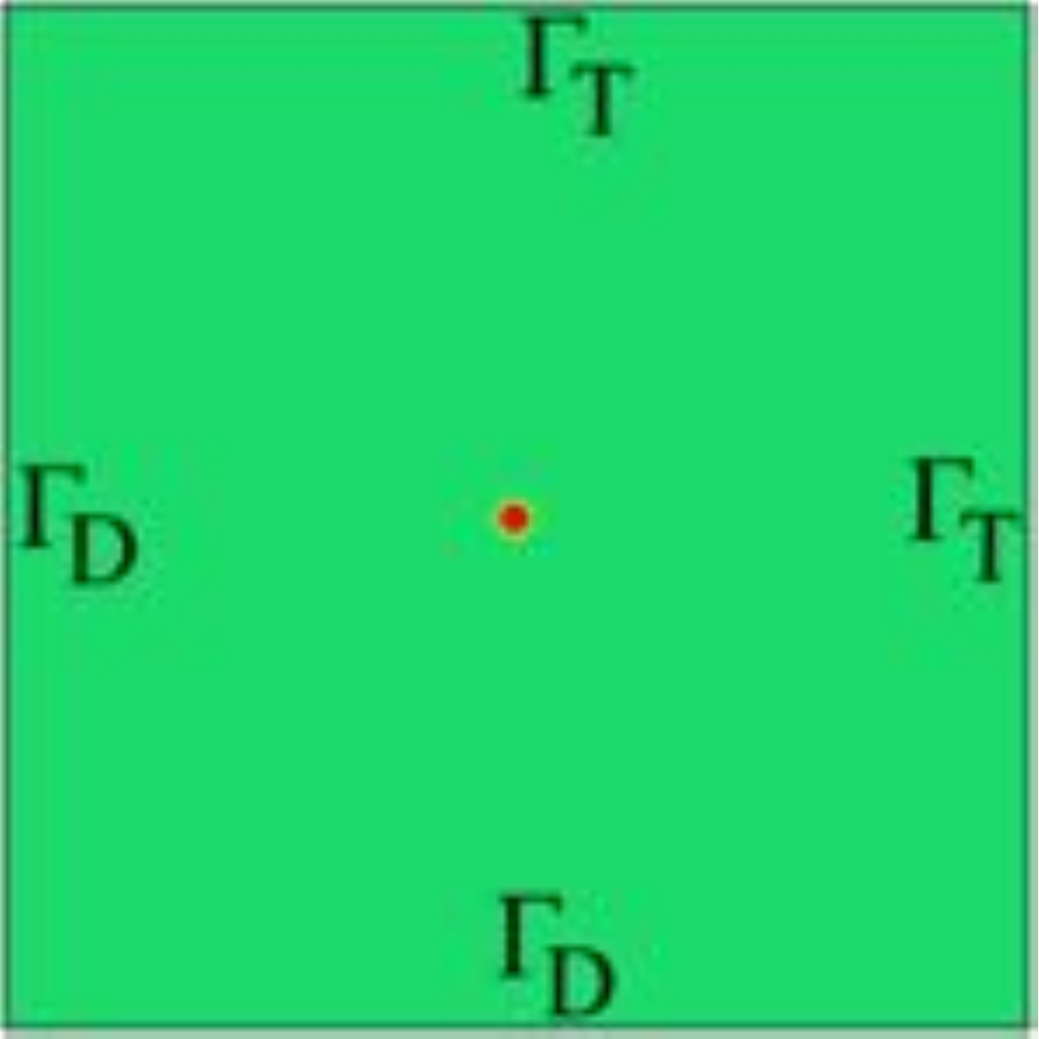}
{$t=0$}
\end{minipage}
\begin{minipage}[b]{0.155\linewidth}
\centering
\includegraphics[width=0.99\linewidth]{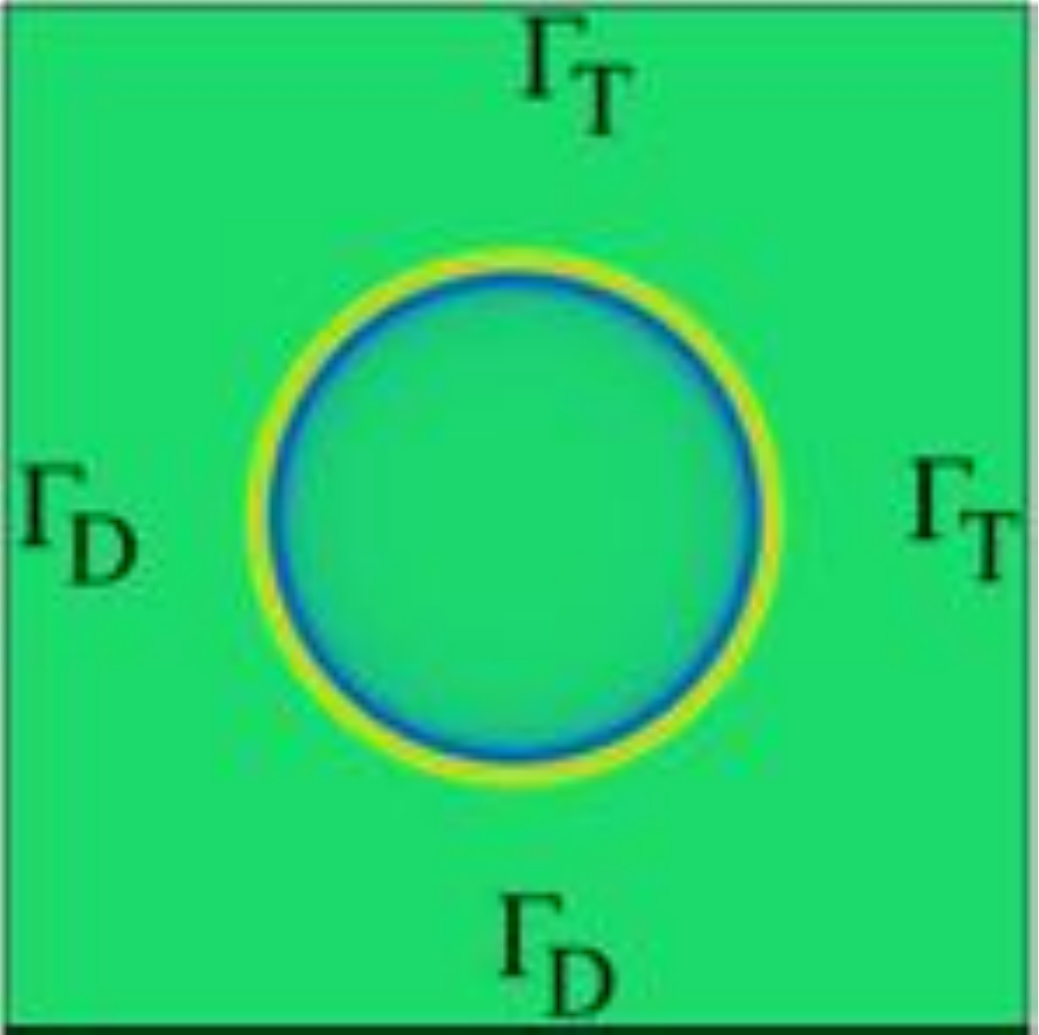}
{$t=25$}
\end{minipage}
\begin{minipage}[b]{0.155\linewidth}
\centering
\includegraphics[width=0.99\linewidth]{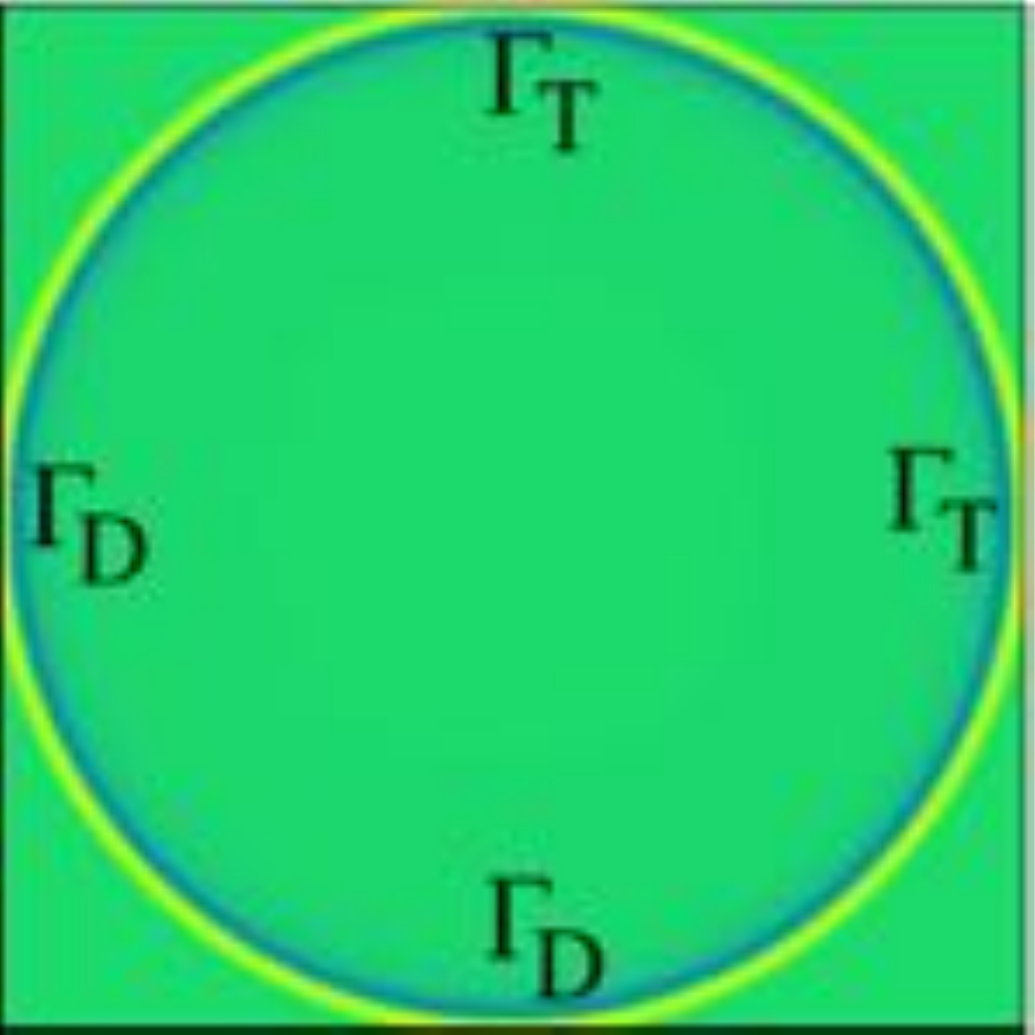}
{$t=50$}
\end{minipage}
\begin{minipage}[b]{0.155\linewidth}
\vspace{0.5cm}
\centering
\includegraphics[width=0.99\linewidth]{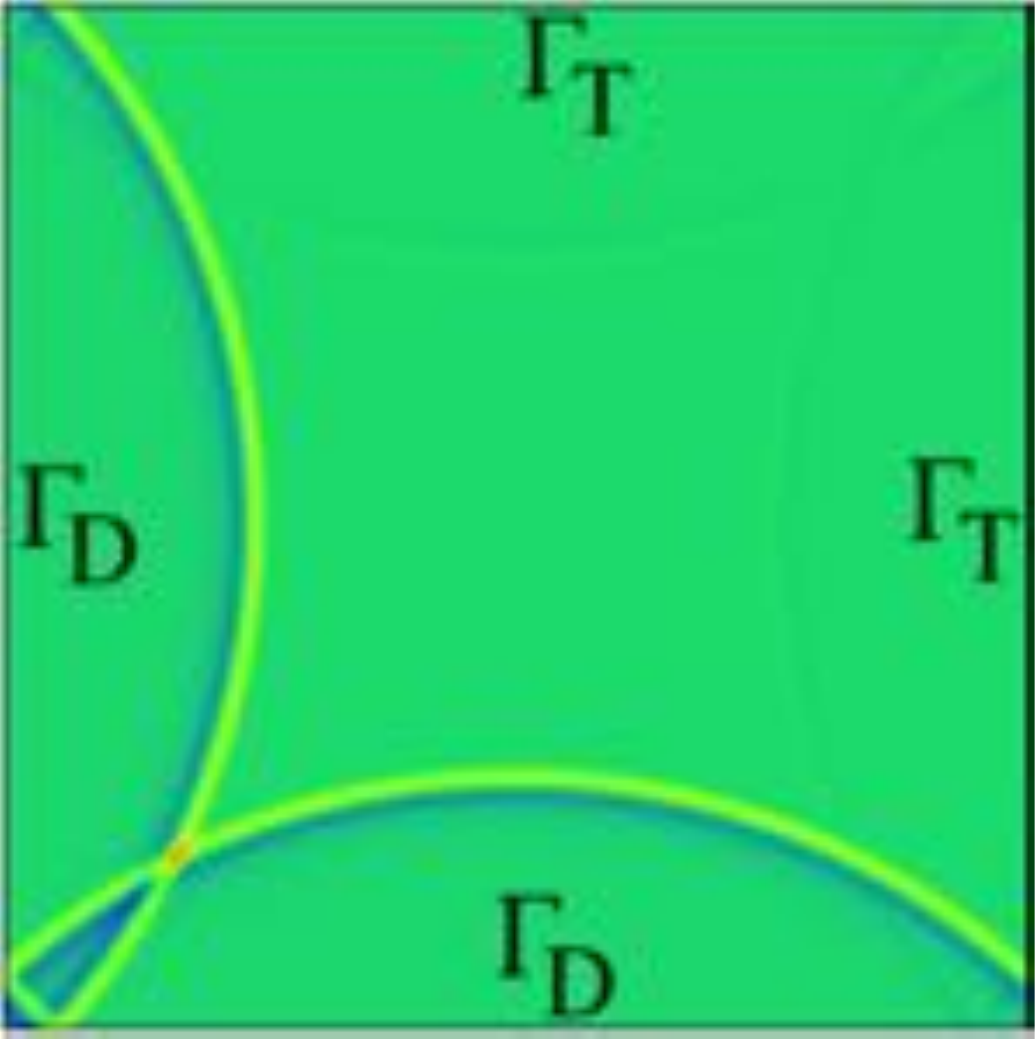}
{$t=75$}
\end{minipage}
\begin{minipage}[b]{0.155\linewidth}
\vspace{0.5cm}
\centering
\includegraphics[width=0.99\linewidth]{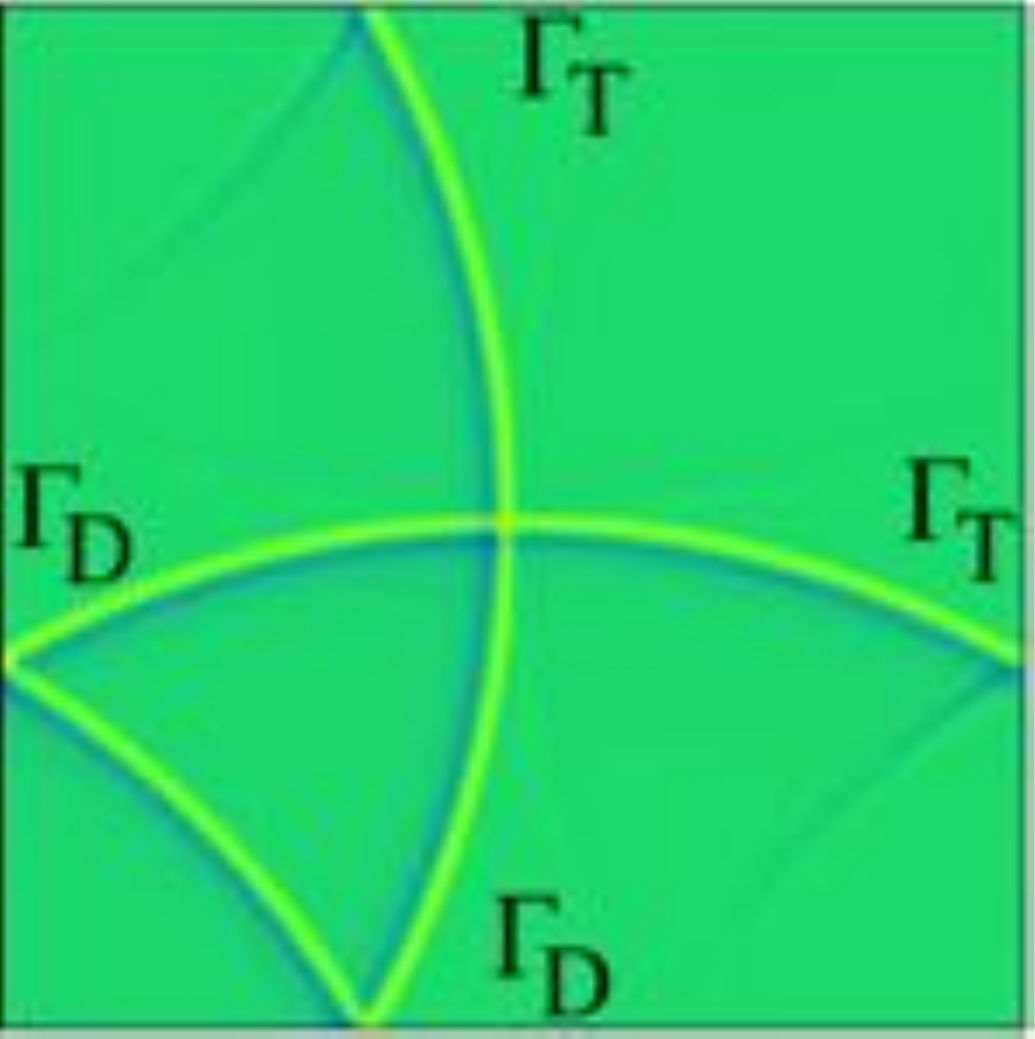}
{$t=100$}
\end{minipage}
\begin{minipage}[b]{0.062\linewidth}
\vspace{0.5cm}
\centering
\includegraphics[width=0.99\linewidth]{Colour_bar}
\vspace{0.00cm}
\end{minipage}
 
\begin{minipage}[b]{0.08\linewidth}
\centering
{$(iv)$}
\vspace{1cm}
\end{minipage}
\begin{minipage}[b]{0.155\linewidth}
\centering
\includegraphics[width=0.99\linewidth]{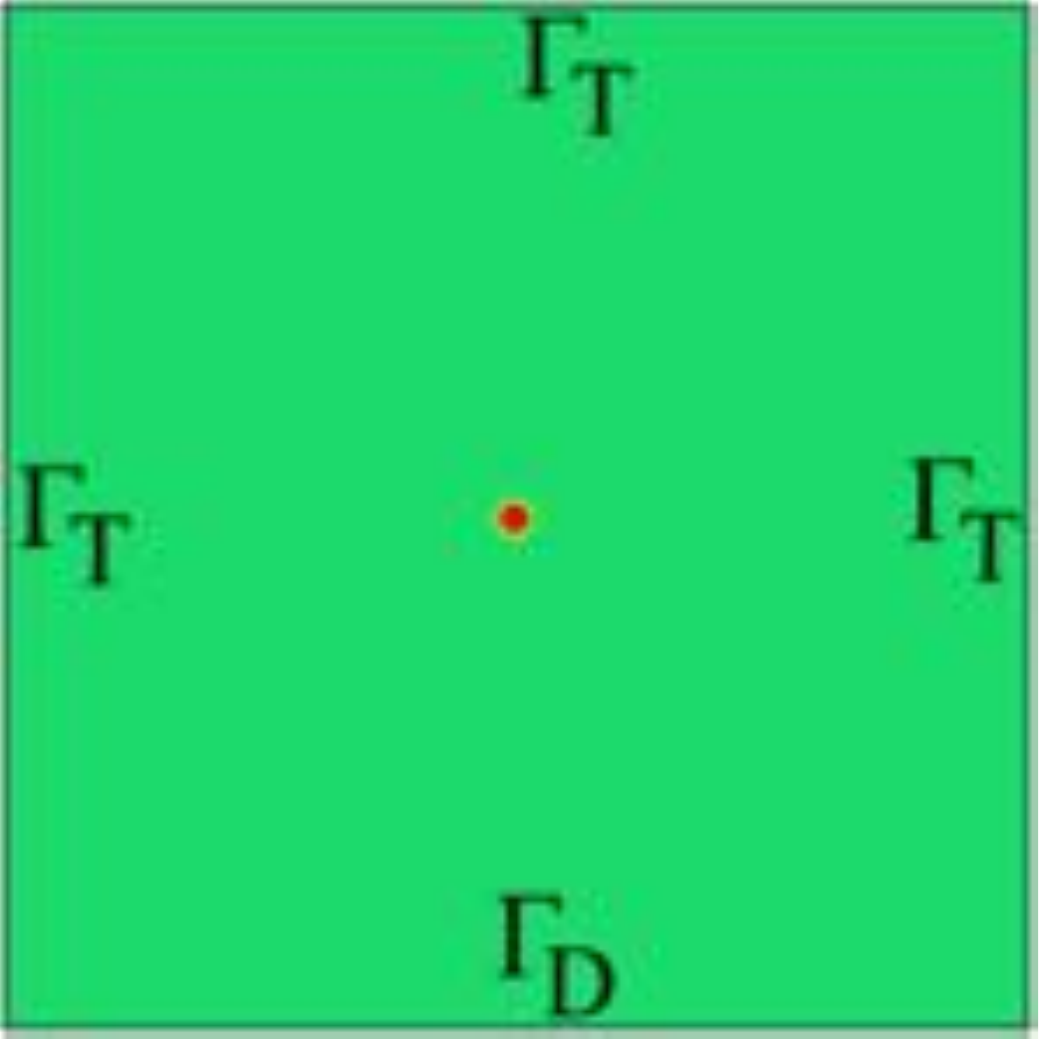}
{$t=0$}
\end{minipage}
\begin{minipage}[b]{0.155\linewidth}
\centering
\includegraphics[width=0.99\linewidth]{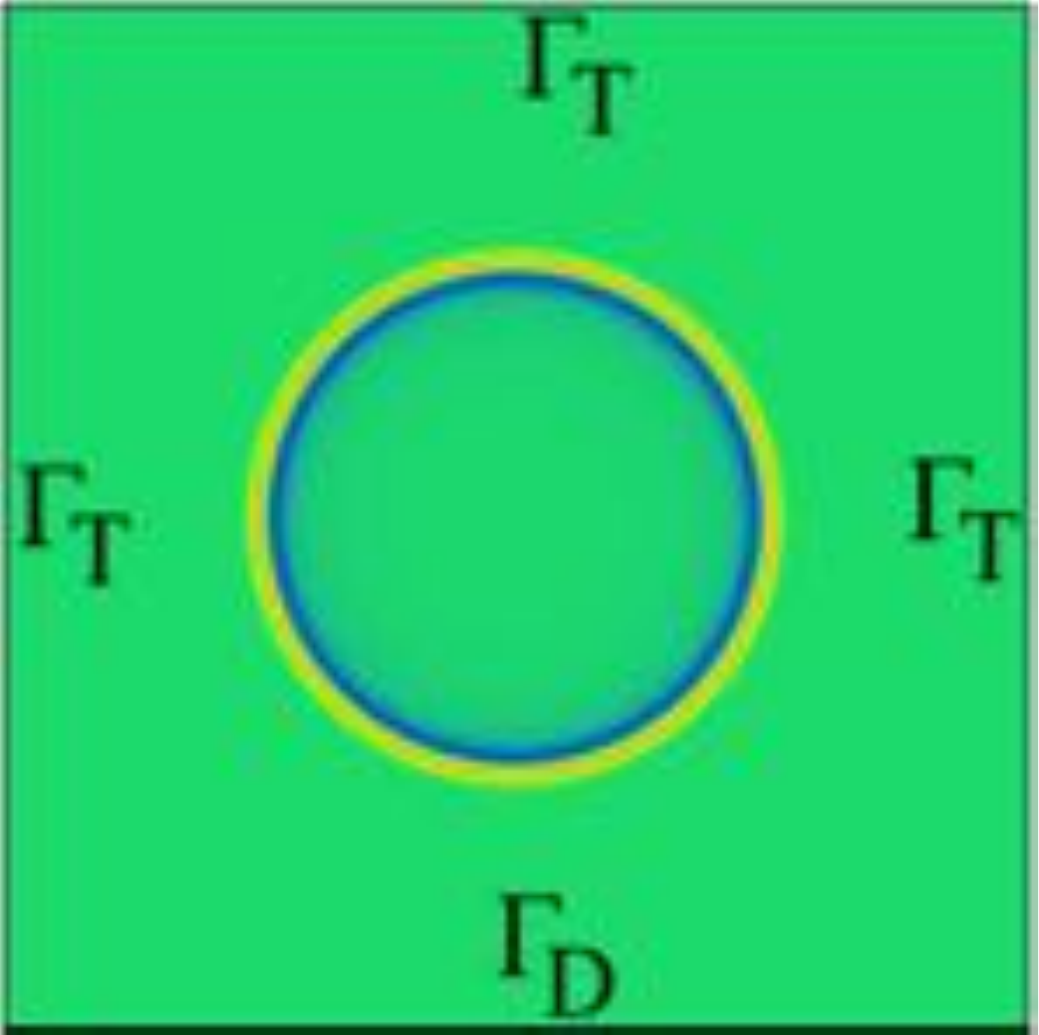}
{$t=25$}
\end{minipage}
\begin{minipage}[b]{0.155\linewidth}
\centering
\includegraphics[width=0.99\linewidth]{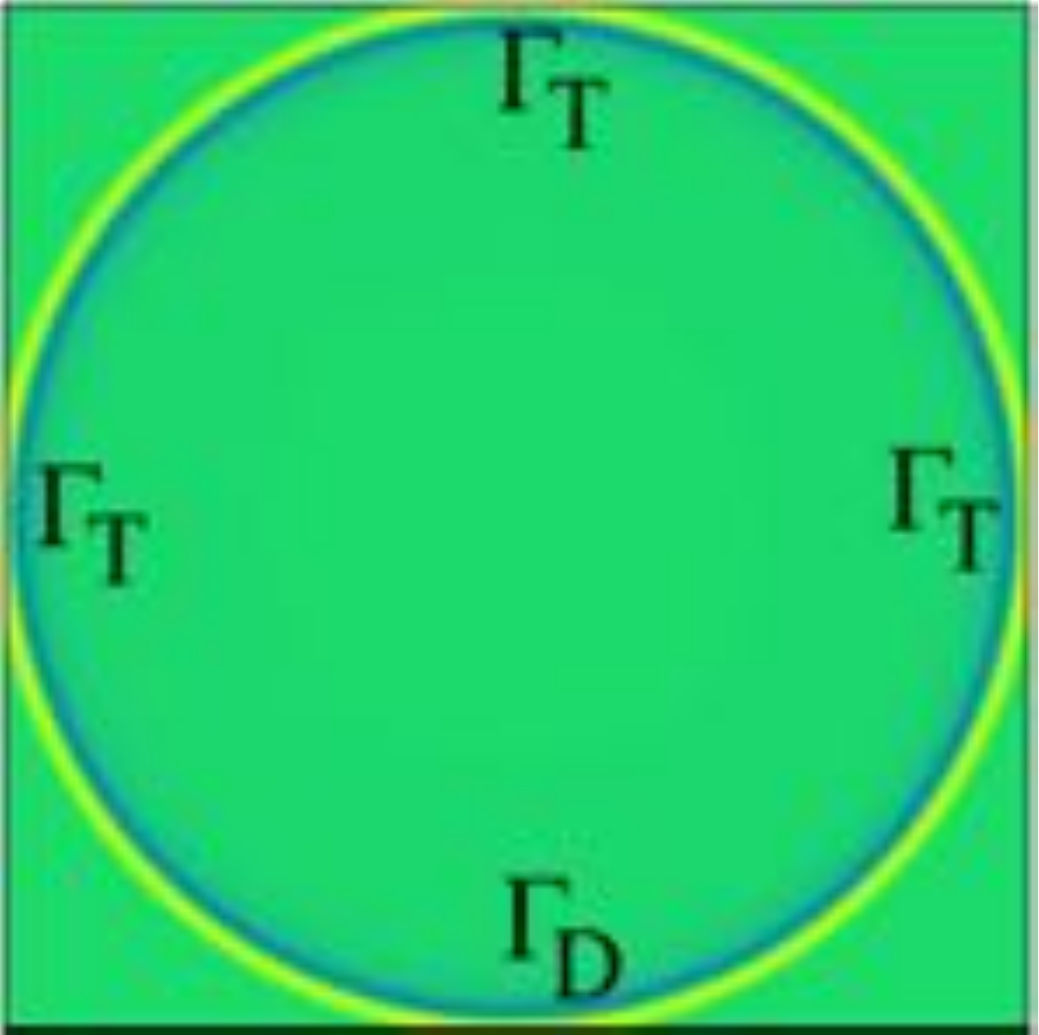}
{$t=50$}
\end{minipage}
\begin{minipage}[b]{0.155\linewidth}
\vspace{0.5cm}
\centering
\includegraphics[width=0.99\linewidth]{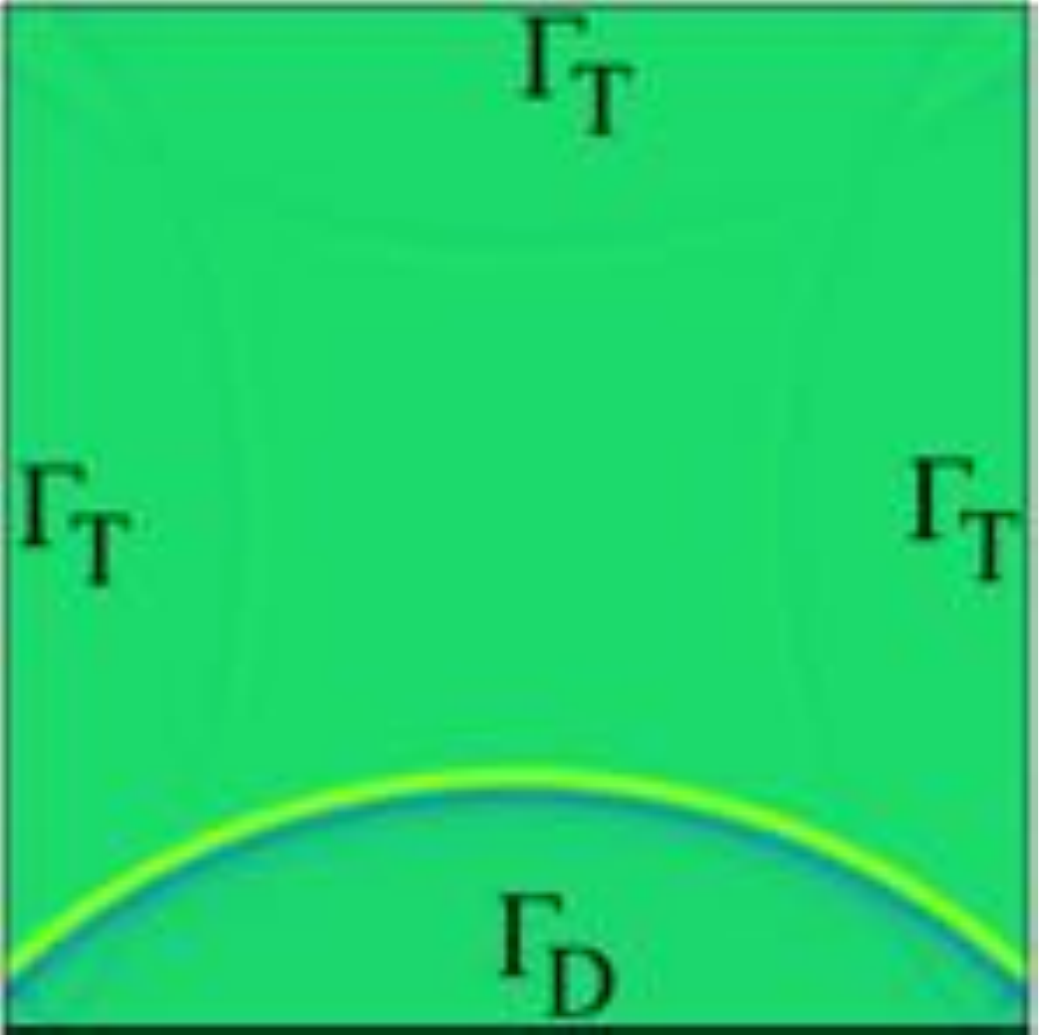}
{$t=75$}
\end{minipage}
\begin{minipage}[b]{0.155\linewidth}
\vspace{0.5cm}
\centering
\includegraphics[width=0.99\linewidth]{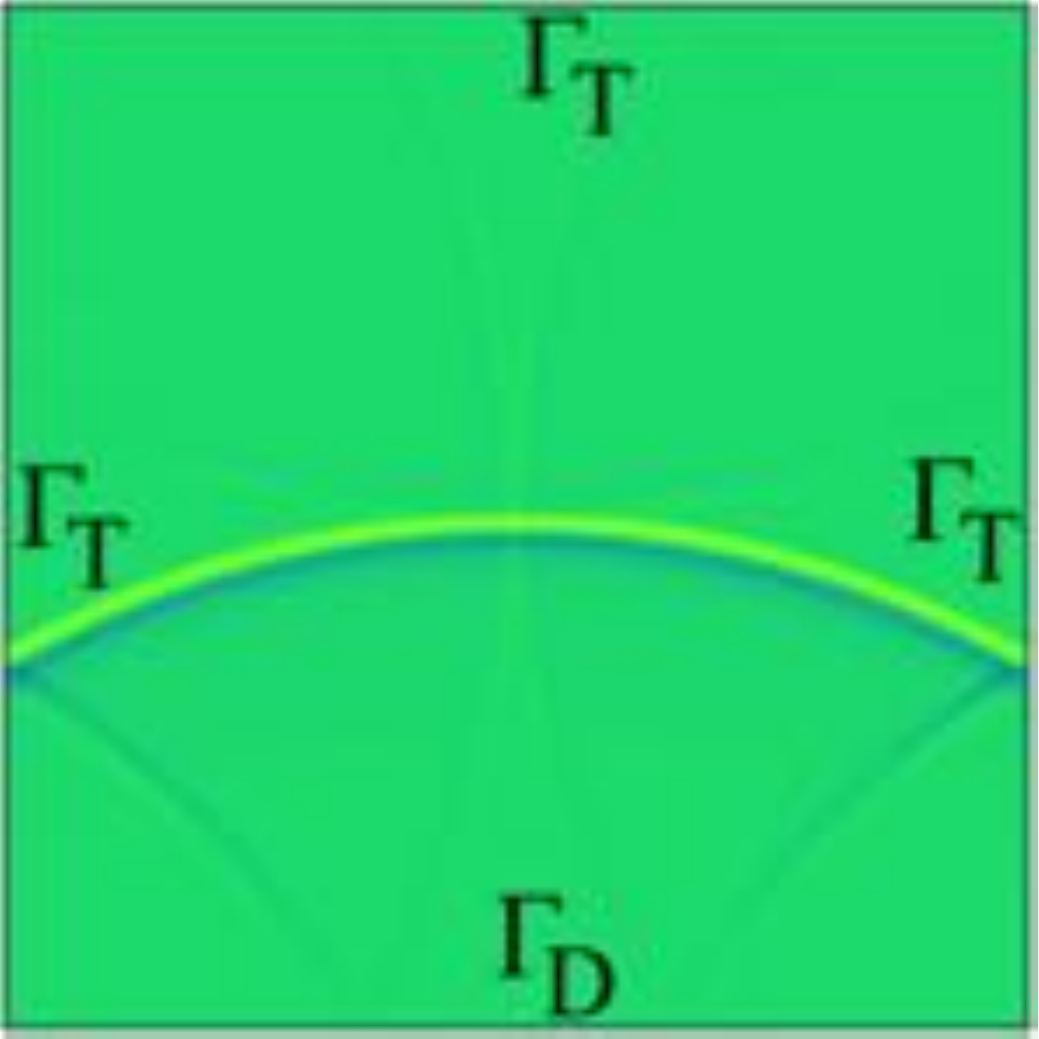}
{$t=100$}
\end{minipage}
\begin{minipage}[b]{0.062\linewidth}
\vspace{0.5cm}
\centering
\includegraphics[width=0.99\linewidth]{Colour_bar}
\vspace{0.00cm}
\end{minipage} 

\begin{minipage}[b]{0.08\linewidth}
\centering
{$(v)$}
\vspace{1cm}
\end{minipage}
\begin{minipage}[b]{0.155\linewidth}
\centering
\includegraphics[width=0.99\linewidth]{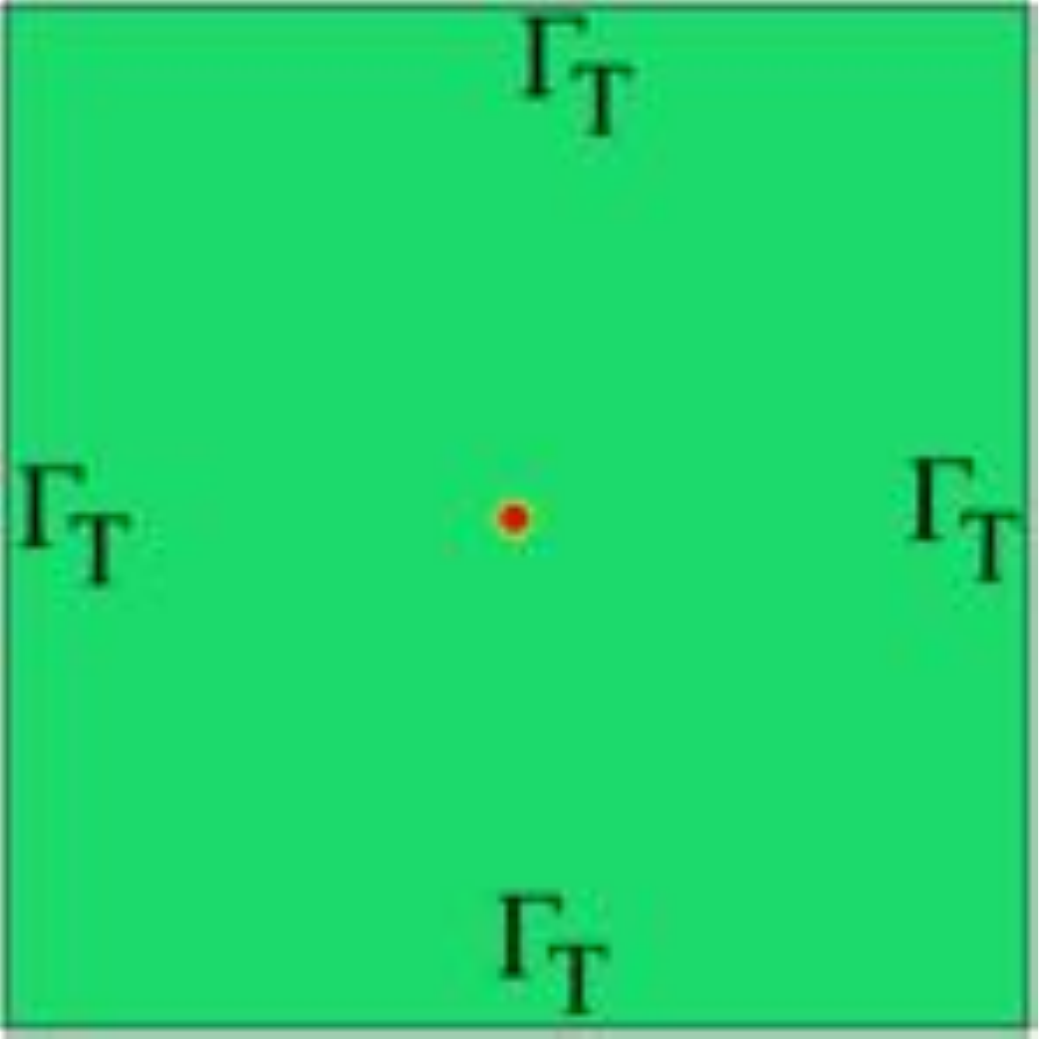}
{$t=0$}
\end{minipage}
\begin{minipage}[b]{0.155\linewidth}
\centering
\includegraphics[width=0.99\linewidth]{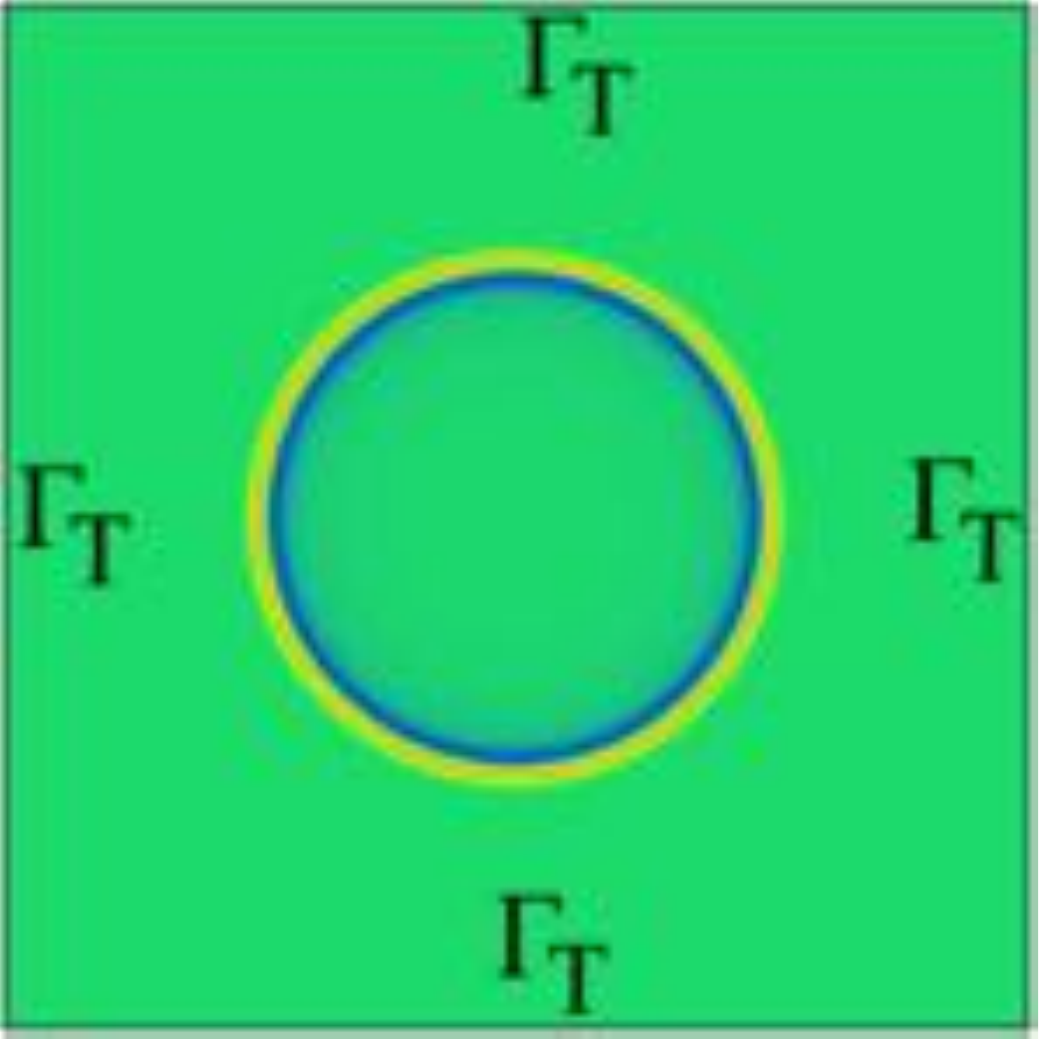}
{$t=25$}
\end{minipage}
\begin{minipage}[b]{0.155\linewidth}
\centering
\includegraphics[width=0.99\linewidth]{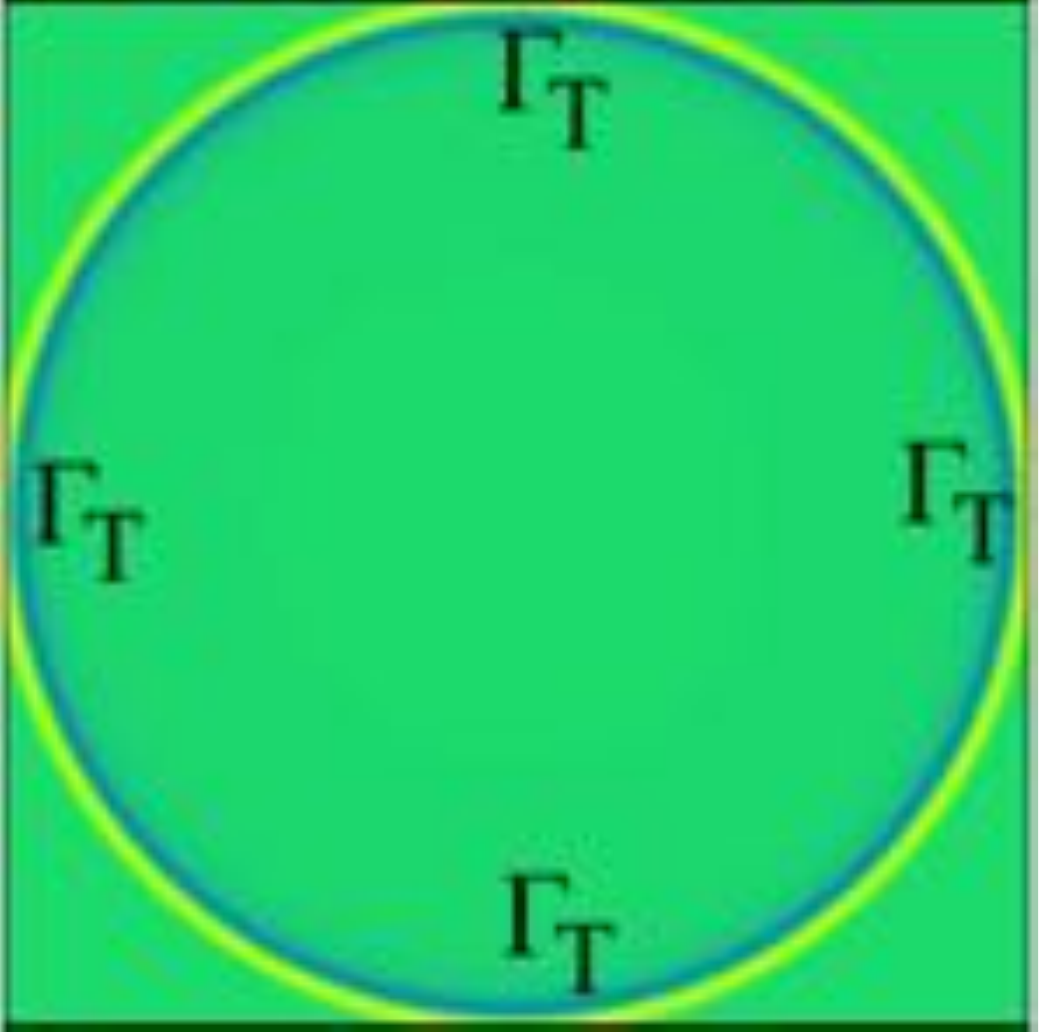}
{$t=50$}
\end{minipage}
\begin{minipage}[b]{0.155\linewidth}
\vspace{0.5cm}
\centering
\includegraphics[width=0.99\linewidth]{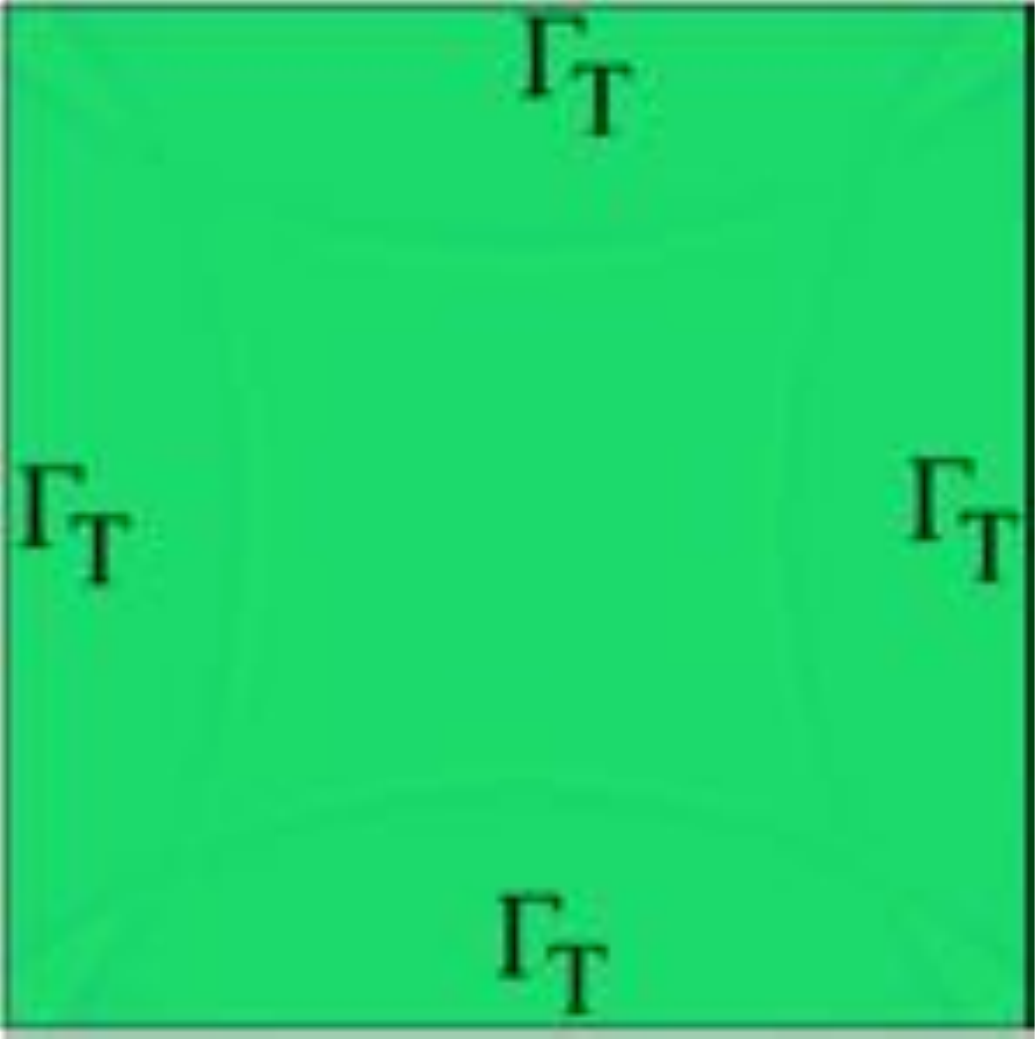}
{$t=75$}
\end{minipage}
\begin{minipage}[b]{0.155\linewidth}
\vspace{0.5cm}
\centering
\includegraphics[width=0.99\linewidth]{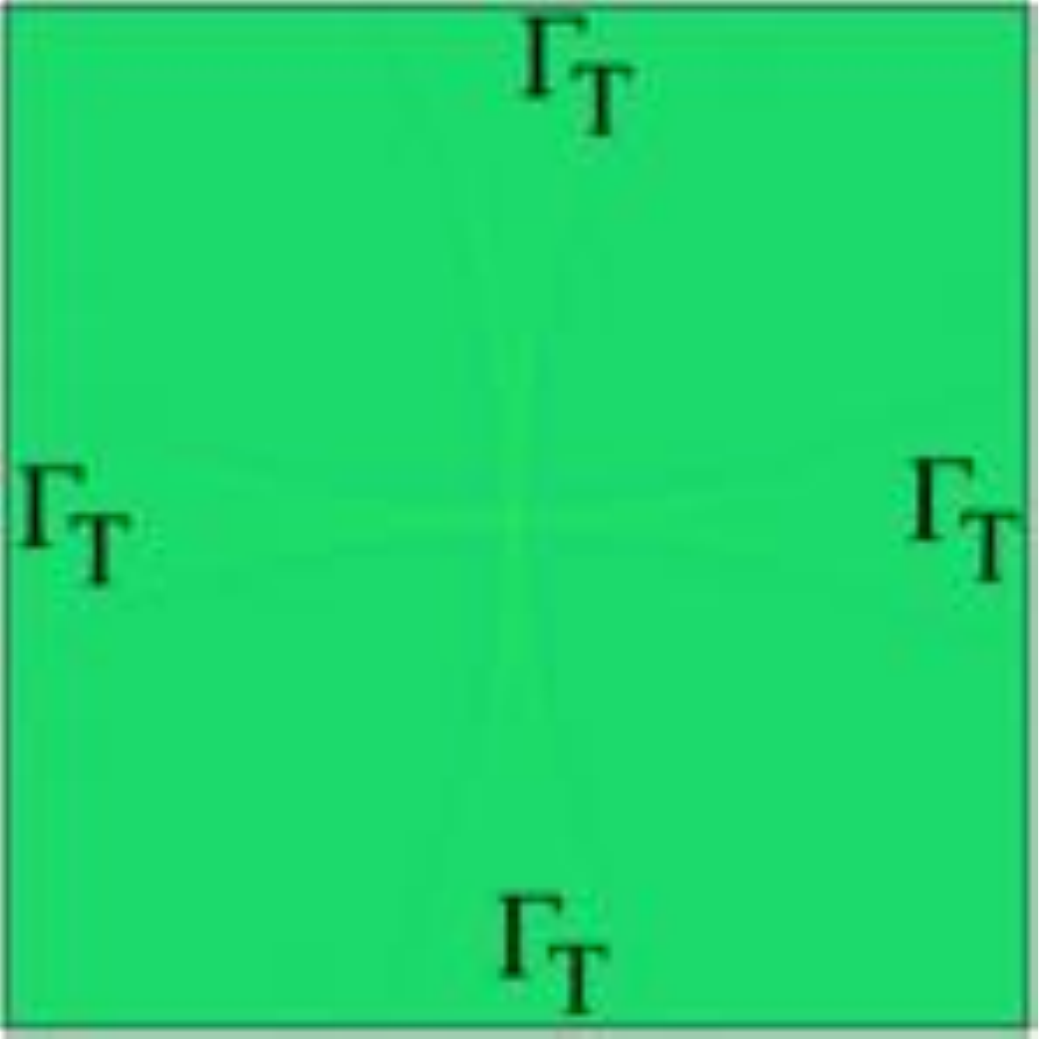}
{$t=100$}
\end{minipage}
\begin{minipage}[b]{0.062\linewidth}
\vspace{0.5cm}
\centering
\includegraphics[width=0.99\linewidth]{Colour_bar}
\vspace{0.00cm}
\end{minipage}

\caption{\cK Color contours of $\eta_h^k$ by Lagrange--Galerkin scheme~\eqref{scheme:LG} for the five cases~$(i)$-$(v)$ discussed in Section~\ref{sec5}.}
\label{Simulatedresult_LG}
\end{figure}
\section{Conclusions}\label{sec6}
Energy estimates of the SWEs with a transmission boundary condition have been studied mathematically and numerically.
For a suitable energy, we have obtained an equality that the time-derivative of the energy is equal to a sum of three line integrals and a domain integral in Theorem~\ref{th1}.
The theorem implies a (successful) energy estimate of the SWEs with the Dirichlet and the slip boundary conditions, cf. Corollary~\ref{cor1}-(ii).
After that, an inequality for the energy estimate of the SWEs with the transmission boundary condition has been proved in Theorem~\ref{th2}.
In the proof, it has been shown that a sum of two line integrals over the transmission boundary is non-positive under some conditions to be satisfied in practical computation.
Based on the theoretical results, the energy estimate of SWEs with the transmission boundary condition has been confirmed numerically by an FDM.
It is found that the transmission boundary condition works well numerically and that the transmission boundary condition reduces the energy drastically via the term $I_{h2}^k$.
The choice of a positive constant~$c_0$ used in the transmission boundary condition has been investigated additionally, and it has been observed that the suitable value lies in~$[0.7, 1.0]$ in the case of zero initial velocity.
Furthermore, we have presented numerical results by an LG scheme, which are similar to those by the FDM.
Completeness of the (theoretical) energy estimate of the SWEs with the transmission boundary condition is a future work.
%
%
%
%

\section*{Acknowledgements}
M.M.M. is supported by MEXT scholarship.
This work is partially supported by JSPS KAKENHI Grant Numbers JP16H02155, JP17H02857 and JP18H01135, JSPS A3 Foresight Program, and JST PRESTO Grant Number~JPMJPR16EA.


%
%
%
%
%
%
%
%
%
%

\medskip

\end{document}